\documentclass[a4paper,10pt]{article}
\usepackage{amscd}
\usepackage{amsmath}
\usepackage{bbding}
\usepackage{amsfonts}
\usepackage{cite}
\usepackage{caption}
\usepackage{mathrsfs}
\usepackage{graphicx}
\usepackage{subfigure}
\usepackage{amssymb}
\usepackage{float}
\usepackage{cases}
\usepackage{graphicx,latexsym,amsmath,amssymb,amsthm,enumerate}
\allowdisplaybreaks[2]   %允许公式分页显示
9.5in\textwidth 6.4in \hoffset -2.0cm \DeclareFontEncoding{OT2}{}{}
\DeclareTextSymbol{\cyrsftsn}{OT2}{126}
\DeclareTextSymbol{\textnumero}{OT2}{125}

\renewcommand{\theequation}{\thesection.\arabic{equation}}
\theoremstyle{definition}

\newtheorem{lemma}{Lemma}[section]

\newtheorem{proposition}{Proposition}[section]
\newtheorem{definition}{Definition}[section]
\newtheorem{remark}{Remark}[section]

\begin{document}
\title{{\LARGE\bf{ Asset Prices with Investor Protection and Survival Analysis of Shareholders in the Cross-Sectional Economy}\thanks{This work was supported by  the National Natural Science Foundation of China (11801462, 12171339).}}}
\author{Jia Yue$^a$, Ming-Hui Wang$^a$, Nan-Jing Huang$^b$\footnote{Corresponding author.  E-mail addresses: nanjinghuang@hotmail.com, njhuang@scu.edu.cn} and Ben-Zhang Yang$^c$ \\
{\small\it a. Department of Economic Mathematics, Southwestern University of Finance and Economics,}\\
{\small\it Chengdu, Sichuan 610074, P.R. China}\\{\small\it b. Department of Mathematics, Sichuan University, Chengdu,
Sichuan 610064, P.R. China}\\
{\small\it  c. National Institution for Finance $\&$ Development, Chinese Academy of Social Science, Beijing, P.R. China}}
\date{ }
\maketitle
\begin{flushleft}
\hrulefill\\
\end{flushleft}
 {\bf Abstract}.
In this paper, we consider a dynamic asset pricing model in a cross-sectional economy with two firms where a controlling shareholder cannot divert output in one firm with perfect investor protection for minority shareholders and where he can divert a fraction of output in the other firm with imperfect protection. After obtaining the parameters of asset prices by solving the shareholders' consumption-portfolio problems in equilibrium, our model features the effect of investor protection and cross-section in the economy. Furthermore, some survival analysis of the shareholders is presented and sufficient conditions on extinction of the shareholders are given in either firm. Our numerical results are in line with some empirical evidence: (i) poorer investor protection in the cross-sectional economy enables the controlling shareholder to hold less shares of the firm with perfect protection and more shares of the firm with imperfect protection, decreases stock gross returns of both firms, increases stock volatilities of both firms, and decreases interest rates of the economy; (ii) compared with the economy with the single relative firm, for the firm with perfect protection, cross-section enables the controlling shareholder to hold less shares, decreases stock returns, increases stock volatilities slightly and decreases interest rates, while for the firm with imperfect protection, cross-section enables the controlling shareholder to hold more shares, increases stock returns and volatilities and increases interest rates.
%\begin{flushleft} \hrulefill \end{flushleft}

\section{Introduction}\label{section1} \noindent
\setcounter{equation}{0}
Controlling shareholders who dominate corporations in many countries can benefit from diverting resources, and hence it is of much significance to protect minority shareholders and study the effect of  such investor protection on ownership concentration and asset prices. In \cite{Basak}, Basak et al. develop a dynamic continuous-time model where a controlling shareholder dominates a firm and diverts a fraction of its output and where shareholders' consumption-portfolio choice and equilibrium asset prices are derived with myopic preferences. Their theory features the effect of investor protection on the firm ownership concentration, stock returns and volatilities, interest rates and so on. As for the economy with an additional firm and different levels of investor protection, their research adopts a simpler, one-period model because of ``additional hedging demands in shareholder portfolios, nonlinear differential equations for value functions, and the non-convexity of the controlling shareholder's portfolio optimization" (see Subsection 5.3 in \cite{Basak}). As is showed in the empirical evidence (see, for instance, \cite{Chabakauri}), the effect of investor protection on firm stock returns is largely cross-sectional. However, an overarching theory that simultaneously studies the effects of investor protection and cross-section on the firm ownership concentration, cross-sectional stock returns and volatilities, interest rates and survival of shareholders is barely found.

Available literature shows that investor protection indeed simultaneously affects how shareholders hold stock shares and how processes of asset prices diffuse. Shleifer and Wolfenzon \cite{Shleifer} assume that the entrepreneur diverts a fraction of revenue and is caught with a probability and that such a probability is used to measure the legal protection of investors. Under their static two-date economy model for various countries, their results in equilibrium show that, the country with better investor protection has lower ownership concentration, larger external capital markets, larger firms, less diversion and higher market interest rates. By assuming the controlling shareholder ``steals" a fraction from gross output by incurring a cost, Albuquerue and  Wang \cite{Albuquerue} introduce a parameter to measure investor protection and capture law enforcement protection of minority investors. In their dynamic continuous-time model for a single firm, equilibrium results reveal that better investor protection reduces over-investment and leads to lower stock volatility, smaller risk premia and lower interest rate and that perfect investor protection would increase the stock market's value. Giannetti and Koskinen \cite{Giannetti} establish a static model for two countries (Home and Foreign) with different shareholders and different levels of investor protection and then introduce investor protection by the positive dependence of the utility of controlling shareholders on the private benefits of control. They find in equilibrium that the relation between the stock price and investor protection is non-monotonic and that with same wealth distribution, the country with better investor protection has less concentrated ownership and higher security returns (i.e., stock returns and interest rates). Basak et al. \cite{Basak} solve both the controlling shareholder's and the minority shareholder's consumption-portfolio problems in equilibrium, and draw a conclusion that better investor protection results in lower stock holdings of controlling shareholders, higher stock returns and volatilities, and higher interest rates. For more literature related to investor protection, one can refer to \cite{Bebchuk,Doidge,La2} for static models and \cite{Dow,Gompers} for dynamic models.

Financial markets are complex and investors cannot completely ignore the phenomena that firm stocks are highly correlated (see \cite{Cont,Solnik,Zhu}), which would also affect ownership concentration and asset prices. For one thing, consumption strategies determine investors' consumption-portfolio choice in equilibrium directly, and hence the states of investors' share in the aggregate consumption have a profound effect on market states, including stock return correlations. For example, in a Lucas economy with two stocks, Chabakauri \cite{Giannetti} uses the states of consumption shares to study margin and leverage constraints in equilibrium. The equilibria are Markovian in two state variables including consumption shares, and the analysis of equilibrium shows that consumption shares would affect stock return correlations and volatilities, but the effect is complex and differs from conditions. For more related literature, one can also refer to \cite{Basak,Chabakauri2015,Chabakauri2021}. For another, since there exist different firms in the cross-sectional economy, empirical evidence shows that firm stocks differ from market capitalization of firms. Banz \cite{Banz} studies the empirical relationship between the stock return and the total market value finding that, on average, smaller firms have higher risk returns than big firms. By examining various regions, Fama and French \cite{Fama} find that value premiums in average stock returns are lager for small stocks than big stocks and spreads in average momentum returns also decrease from smaller to bigger stocks.  More discussion on ``size effect" of firms can be found in \cite{Cakici,Fama1992,Groot} and references therein.

From above researches, not only is investor protection theoretically studied and empirically examined in corporate finance, but the cross-sectional economy with multiple stocks are also adopted in investors' consumption-portfolio problems. Therefore, just like the economy with a single firm in \cite{Basak}, it is of significance to establish a dynamic continuous-time model for a cross-sectional economy (where the word ``cross-sectional" is adopted form \cite{Basak,Fama1992}) and then address the related effect of investor protection and cross-section. In order to establish an economy which accounts for investor protection and cross-section, we develop a dynamic continuous-time model based on the model used in \cite{Basak}. On the one hand, we introduce two firms with different levels of investor protection and two stocks with cross-sectional returns. One firm is with perfect investor protection while there may exist diversion in the other firm, which enables us to analyze the effect of investor protection and cross-section through comparison. On the other hand, following \cite{Basak2000,Basak2005}, we restrict the form of risks in the economy and divide them into two parts---the fundamental risk and the nonfundamental risk. This not only helps us know better how market risks affect stocks but also allows us to obtain shareholders' stock holdings and parameters of asset prices in equilibrium. More precisely, our main results and contribution are threefold as follows.

Firstly, similar to the results in \cite{Basak}, not only is shareholders' consumption-portfolio choice with various regions obtained in a partial equilibrium, but their optimal strategies and asset price are also derived in equilibrium. Compared with the economy with a single firm, there are more market clear conditions to maintain the equilibrium in the cross-sectional economy, and hence, shareholders' consumption-portfolio choice and asset prices have to be represented with more state variables. Indeed, the consumption share is the only state variable used in \cite{Basak} where the economy has investor protection and a single stock, and in \cite{Chabakauri2021} and \cite{Chabakauri}, the consumption share and the dividend ratio are used to derive dynamic equilibrium where the economy has two stocks. Since our cross-sectional economy has both investor protection and two stocks, except for the consumption share and the dividend ratio, we add an additional state variable--stock value ratio in equilibrium. As is discussed above, the stock value ratio represents the ``size effect" and enriches the results in equilibrium.
Moreover, with the aid of Karush-Kuhn-Tucker conditions, we give some sufficient conditions which are able to guarantee the unique existence of the partial equilibrium.

Secondly, we consider survival analysis of shareholders in the cross-sectional economy and give some sufficient conditions on extinction of shareholders in either firm. Such survival analysis in the cross-sectional economy can help us to know why shareholders have extreme consumption-portfolio choice. In contrast with the economy of \cite{Basak} where shareholders would always survive in the long run, the shareholders in our cross-sectional economy could survive in one firm and become extinct in the other firm. One reason is that different firms and investors in an economy are in competition. For one thing, numbers of researches are undertaken to study such competition (see \cite{Aghion,Dixit,Dumas,Lucas} and references therein). For another, there are also lots of researches on survival analysis of firms and investors (see, for example, \cite{Browne,Jiang,Laitinen,Lane}). Hence, it is of importance to study whether shareholders would survive, and our survival analysis of shareholders in the cross-sectional economy provides some new methods.

Finally, making use of numerical analysis, we study the effect of investor protection and cross-section in equilibrium. The results are consistent with existing researches and enable us to analyze how firms and shareholders behave in the cross-sectional economy. As is expected, investor protection indeed protects minority shareholders in the cross-sectional economy, for the controlling shareholder tends to hold more shares of the firm with imperfect protection and less shares of the firm with perfect protection. Better investor protection has a general effect on asset prices and leads to higher stock gross returns, lower stock volatilities and higher interest rates. Cross-section brings fiercer competition to the economy, but the competition differs from firms. Compared with the economy with the single relative firm, competition caused by cross-section decreases the growth of the firm stock with perfect protection and increases the growth of the firm stock with imperfect protection. Both negative and positive correlations between competition and growth exist in markets (see \cite{Aghion,Grossman,Nickell,Porter} and references therein). Furthermore, we also study the ``size effect" and extinction conditions of shareholders in the cross-sectional economy.

The rest of the paper is organized as follows. In next section, we introduce dynamics and shareholders' consumption-portfolio problems in a cross-sectional economy with investor protection. Then in Section \ref{section3}, market parameters of asset prices are derived and survival of shareholders in either firm is discussed. Further, the effect of investor protection and cross-section in equilibrium is analyzed through numerical methods in Section \ref{section4}. Before we summarize this paper in Section \ref{section6},  our basic economy is extended in Section \ref{section5} by considering investor protection in both firms.

\section{The cross-sectional economy with investor protection}\label{section2}\noindent
\setcounter{equation}{0}
Let $(\Omega,\mathcal{F},\mathbb{P})$ be a probability space and $\mathbb{F}=(\mathcal{F}_t,t\geq 0)$ be a given filtration on that space satisfying the usual conditions. We consider a continuous-time infinite-horizon economy on such a filtered probability space. More precisely, we assume, following \cite{Basak2000,Basak2005}, that the fundamental aggregate risk $Z_{Df}$ in the economy is driven by an $\mathbb{F}-$Brownian motion $W$ and that the nonfundamental risk $Z_{Dn}$ in the economy is driven by another $\mathbb{F}-$Brownian motion $B$ (independent of $W$), and they may affect output and stocks in firms differently.

Two firms (firm 1 and firm 2) in the economy produce exogenous streams of output $\widehat{D}_j\;(j=1,2)$ which can be viewed as the consumption supply, and there are two types of investors (controlling shareholders and  minority shareholders) in the economy. Two firms could have different levels of investor protection for minority shareholders: controlling shareholders in firm 1 cannot divert any output of firm 1 while controlling shareholders in firm 2 can divert a fraction of output in firm 2 (if firm 2 is with imperfect protection). Firms $1$ and $2$ may have different controlling shareholders. However, since controlling shareholders in firm $1$ cannot divert output in firm $1$, they faces the same consumption-portfolio choice as minority shareholders. Hence, such controlling shareholders in firm $1$ can be subsumed within minority shareholders, which is the same as the assumption in Subsection 5.3 of \cite{Basak}. Therefore, we assume, for the sake of simplicity, that there are only a representative controlling shareholder $C$ and a representative minority shareholder $M$ in both firms. Furthermore, firm 1 is affected only by the fundamental risk while firm 2 is affected by both the fundamental and the nonfundamental risk. Following the work \cite{Basak2000} and Subsection 2.1 of \cite{Garleanu}, we assume that the fundamental risk and the nonfundamental risk in the economy affect the output in the following way:
\begin{align*}
dZ_{Df}(t)=Z_{Df}(t)\sigma_DdW(t),\quad dZ_{Dn}(t)=Z_{Dn}(t)\delta_DdB(t),
\end{align*}
where $\sigma_D$ and $\delta_D$ are positive constants actually representing the volatility of output caused by relative risk. Then the dynamics of output can be determined as follows. Letting the output mean-growth of firm $1$ be $\mu_{1D}$ (i.e., without consideration of risk its output $\widetilde{D}_1$ can be expressed as $d\widetilde{D}_1(t)=\widetilde{D}_1(t)\mu_{1D}dt$) and its real output $\widehat{D}_1$ be given by $\widehat{D}_1=\widetilde{D}_1Z_{Df}$, It\^o's Lemma (see, for example, \cite{Oksendal}) shows that the output in firm $1$ satisfies the following stochastic differential equation (SDE),
\begin{equation}\label{D1}
d\widehat{D}_1(t)=\widehat{D}_1(t)[\mu_{1D}dt+\sigma_{D}dW(t)].
\end{equation}
On the other hand, letting the output mean-growth of firm $2$ be $\mu_{2D}$
(i.e., without consideration of risk its output $\widetilde{D}_2$ can be expressed as $d\widetilde{D}_2(t)=\widetilde{D}_2(t)\mu_{2D}dt$) and its real output $\widehat{D}_2$ be given by $\widehat{D}_2=\widetilde{D}_2Z_{Df}Z_{Dn}$, It\^o's Lemma similarly shows that the output in firm $2$ satisfies the following SDE,
\begin{equation}\label{D2}
d\widehat{D}_2(t)=\widehat{D}_2(t)[\mu_{2D}dt+\sigma_{D}dW(t)+\delta_DdB(t)].
\end{equation}

Three securities, a bond $S_0$ (normalized so that $S_0(0)=1$) and two stocks $S_1,S_2$ (normalized to $\tau>0$ units and one unit, respectively), are traded by shareholders in the security market, and the wealth processes of shareholders $C$ and $M$ at time $t$, denoted by $X_C(t)$ and $X_M(t)$, are made up of portfolios of the bond and the stocks. The fundamental risk and nonfundamental risk in the economy affect the stocks in the following way:
\begin{align*}
dZ_{Sf}(t)=Z_{Sf}(t)\sigma(t)dW(t),\quad dZ_{Sn}(t)=Z_{Sn}(t)\delta(t)dB(t),
\end{align*}
where $\sigma$ and $\delta$ are positive processes actually representing the volatility of stocks caused by relative risk. Then the dynamics of stocks can be similarly determined as follows. For the stock of firm $1$, the stock return (or the stock growth rate) is $\mu_{1}$ (i.e., without consideration of risk its stock price $\widetilde{S}_1$ can be expressed as $d\widetilde{S}_1(t)=\widetilde{S}_1(t)\mu_{1}(t)dt$) and its real price ${S}_1$ is given by ${S}_1=\widetilde{S}_1Z_{Sf}$. And for the stock of firm $2$, the stock return is $\mu_{2}$ (i.e., without consideration of risk its stock price $\widetilde{S}_2$ can be expressed as $d\widetilde{S}_2(t)=\widetilde{S}_2(t)\mu_{2}(t)dt$) and its real price ${S}_2$ is given by ${S}_2=\widetilde{S}_2Z_{Sf}Z_{Sn}$. Now we can express the dynamics of securities in the economy as the following system of SDEs:
\begin{align}
dS_0(t)=&S_0(t)r(t)dt,\label{bond}\\
dS_1(t)=&S_1(t)\left[\mu_1(t)dt+\sigma(t)dW(t)\right],\label{stock1}\\
dS_2(t)=&S_2(t)\left[\mu_2(t)dt+\sigma(t)dW(t)+\delta(t)dB(t)\right],\label{stock2}
\end{align}
where $r$ (the interest rate), $\mu_i\; (i=1,2),\;\sigma$ and $\delta$ are $\mathbb{F}$-adapted and are the price coefficients to be determined in equilibrium. We note here that the stock volatility of firm $2$ is $\sqrt{\sigma^2+\delta^2}$ and that the stock dynamics can prevent the stock in firm $2$ with more volatile output being priced with lower volatility than one in firm $1$.

With the dynamics of the bond and stocks, shareholders' wealth for $i=C,M$  can be expressed as following equations:
\begin{align}
X_i(t)=&b_{i}(t)S_0(t)+\sum_{j=1}^2n_{ji}(t)S_j(t);\label{W}\\
dX_i(t)=&\bigg[X_i(t)r(t)+n_{1i}(t)\left(S_1(t)(\mu_{1}(t)-r(t))+\frac{D_1(t)}{\tau}\right)\nonumber\\
&+n_{2i}(t)\left(S_2(t)(\mu_{2}(t)-r(t))+(1-x(t))D_2(t)\right)+l_{1i}\widehat{D}_1(t)+l_{2i}\widehat{D}_2(t)-c_i(t)\nonumber\\
&+\mathbf{1}_{\{i=C\}}(x(t)D_2(t)-f(x(t),D_2(t)))+\mathbf{1}_{\{i=M\}}f(x(t),D_2(t))\bigg]dt\nonumber\\
&+(n_{1i}(t)S_1(t)\sigma(t)+n_{2i}(t)S_2(t)\sigma(t))dW(t)
+n_{2i}(t)S_2(t)\delta(t)dB(t)\label{X},
\end{align}
where, for $j=1,2$, $b_{i}$ is the number of units of bounds in the portfolio of the shareholder $i$,
$n_{ji}$ is the number of stock shares of firm $j$ in the portfolio of the shareholder $i$,
$c_{i}$ is the consumption of the shareholder $i$, $l_{ji}$ is the fraction of the output paid to the shareholder $i$ as labor incomes in firm $j$,
$D_j(t)=(1-l_{jC}-l_{jM})\widehat{D}_j(t)$ represents the net output in firm $j$,
$f(x,D)=\frac{kx^2D}{2}$ is assumed to be a pecuniary cost from diverting output with a constant $k$ (see \cite{Basak}) and the parameter $k$ captures the magnitude of the cost, $\mathbf{1}$ is the indicator function, and $x$ is the fraction of diverted output satisfying the investor protection constraint with a parameter $p\in[0,1]$ (interpreted as protection of the minority shareholder)
\begin{equation}\label{protectC}
    0\leq x(t)\leq (1-p)n_{2C}(t).
\end{equation}

Following \cite{Basak}, we assume shareholders are myopic and introduce their objective functions (utility functions) in the cross-sectional economy as follows:
\begin{equation}\label{wealth}
V_i(c_{i}(t),X_{i}(t),X_{i}(t+dt))=\rho u_i(c_{i}(t))dt+(1-\rho dt)\mathbb{E}_t\left[u_i(X_{i}(t+dt))\right],
\end{equation}
for $i=C,M$, where $\rho>0$ is a time-preference parameter, $\mathbb{E}_t[\cdot]=\mathbb{E}[\cdot|\mathcal{F}_t]$ and $\mathbb{E}$ is the expectation relative to probability $\mathbb{P}$, and $u_i (i=C,M)$ are the following CRRA (constant relative risk aversion) utility functions:
\[
u_i(c)=\frac{c^{1-\gamma_i}-1}{1-\gamma_i}
\]
with the risk aversion parameters $\gamma_M\geq \gamma_C>0$ and $\gamma_C\neq 1,\gamma_M\neq 1$. The condition $\gamma_M\geq \gamma_C$ implies that the minority shareholder is more risk-averse than the controlling shareholder (see \cite{Basak,Kihlstrom} for more details).
As is demonstrated in \cite{Basak} and references therein, we suppose that ``the constraint $X_{t+dt} \geq 0$ naturally emerges in such economies to prevent defaults by requiring investors to collateralize their risky positions with pledgeable financial securities, even when investors have an infinite planning horizon".

The controlling shareholder maximizes $V_C$ through strategies of the investments $n_{1C}$ and $n_{2C}$, the consumption $c_{C}$ and the diverting fraction $x$:
\begin{equation}\label{problemC}
\max\limits_{n_{1C}(t),n_{2C}(t),c_{C}(t),x(t)}V_C(c_{C}(t),X_{C}(t),X_{C}(t+dt)),
\end{equation}
where $V_C$ is given by (\ref{wealth}) in the case $i=C$, subject to the constraints (\ref{X}) and (\ref{protectC}), and share constraints $0\leq n_{1C}(t)\leq \tau, 0\leq n_{2C}(t)\leq1$. Similarly, the minority shareholder $M$ maximizes $V_M$ through strategies of the investments $n_{1M}$ and $n_{2M}$ and the consumption $c_{M}$:
\begin{equation}\label{problemM}
\max\limits_{n_{1M}(t),n_{2M}(t),c_{M}(t)}V_M(c_{M}(t),X_{M}(t),X_{M}(t+dt)),
\end{equation}
where $V_M$ is given by (\ref{wealth}) in the case $i=M$, subject to the constraint (\ref{X}).

We do not require $0\leq n_{1M}(t)\leq \tau,0\leq n_{2M}(t)\leq 1$  in optimization problem (\ref{problemM}). This is because the constraints $0\leq n_{1C}(t)\leq \tau, 0\leq n_{2C}(t)\leq1$ simplify deriving of equilibrium asset prices and are also sufficient to guarantee share constraints $0\leq n_{1M}(t)\leq \tau$ and $0\leq n_{2M}(t)\leq 1$ in equilibrium. Furthermore, compared with the optimization problem used in \cite{Basak}, the controlling shareholder in the cross-sectional economy not only considers how to divert output in firm $2$ optimally, but also takes account of the effect of the cross-section setting (i.e., the existence of firm $1$) on consumption-portfolio choice and investor protection.

\section{Asset prices and survival analysis in cross-sectional economy}\label{section3}\noindent
\setcounter{equation}{0}
In this section, we solve shareholders' optimization problems (\ref{problemC}) and (\ref{problemM}) and derive shareholders' optimal strategies and dynamics of asset prices in equilibrium. To this end, we start from a partial equilibrium setting and solve shareholders' consumption-portfolio problems.
Then, making use of market clearing conditions, parameters of asset prices are truly obtained in equilibrium. Finally, we consider whether shareholders would survive in the cross-sectional economy.

\subsection{Optimal strategies in the partial equilibrium}\label{subsection3.1}
First of all, we assume $X_i(t)>0$ for $i=C,M$, and the extreme case $X_i(t)=0$ would be discussed later in equilibrium. Then It\^{o}'s Lemma allows us to decompose problems (\ref{problemC}) and (\ref{problemM}) into some equivalent but more precise problems. For the shareholder $i\;(i=C,M)$, applying It\^{o}'s Lemma to
\begin{equation*}
\mathbb{E}_t[u_i(X_i(t+dt)]=u_i(X_i(t))+\mathbb{E}_t[du_i(X_i(t))]
\end{equation*}
gives
\begin{align*}
&V_i(c_{i}(t),X_{i}(t),X_{i}(t+dt))\\
=&\left(\rho\frac{c_{i}(t)^{1-\gamma_i}-1}{1-\gamma_i}-X_{i}(t)^{-\gamma_i}c_{i}(t)\right)dt+(1-\rho dt)\frac{X_{i}(t)^{1-\gamma_i}-1}{1-\gamma_i}+X_{i}(t)^{-\gamma_i}\sum_{j=1}^2l_{ji}\widehat{D}_j(t)dt\\
&+X_{i}(t)^{1-\gamma_i}\left\{r(t)+\sum_{j=1}^2\frac{n_{ji}(t)S_j(t)}{X_i(t)}(\mu_{j}(t)-r(t))
+{n_{1i}(t)}\frac{D_1(t)}{\tau X_i(t)}+n_{2i}(t)(1-x(t))\frac{D_2(t)}{X_i(t)}\right.\\
&+l_{1i}\frac{\widehat{D}_1(t)}{X_i(t)}+l_{2i}\frac{\widehat{D}_2(t)}{X_i(t)}+\mathbf{1}_{\{i=C\}}\left(x(t)-\frac{kx(t)^2}{2}\right)\frac{D_2(t)}{X_i(t)}
+\mathbf{1}_{\{i=M\}}\frac{kx(t)^2}{2}\frac{D_2(t)}{X_i(t)}\\
&\left.-\frac{\gamma_i}{2}\left[ \left(\frac{n_{1i}(t)S_1(t)\sigma(t)}{X_i(t)}\right)^2+\frac{2n_{1i}(t)n_{2i}(t)S_1(t)S_2(t)\sigma(t)^2}{X_i(t)^2} +\left(\frac{n_{2i}(t)S_2(t)}{X_i(t)}\right)^2(\sigma(t)^2+\delta(t)^2)\right]\right\}dt.
\end{align*}
Hence, the optimization problem (\ref{problemC}) for shareholder $C$ is equivalent to
\begin{equation}\label{problemCc}
\max\limits_{c_{C}(t)}\left\{\rho\frac{c_{C}(t)^{1-\gamma_C}-1}{1-\gamma_C}-X_{C}(t)^{-\gamma_C}c_{C}(t)\right\};
\end{equation}
\begin{equation}\label{problemCn}
\max\limits_{x(t),n_{1C}(t),n_{2C}(t)}J_C(n_{1C}(t),n_{2C}(t),x(t)),
\end{equation}
where $J_C$ is given by
\begin{align}
J_C(x(t),n_{1C}(t),n_{2C}(t))&=\sum_{j=1}^2\frac{n_{jC}(t)S_j(t)}{X_C(t)}(\mu_{j}(t)-r(t))
+n_{2C}(t)(1-x(t))\frac{D_2(t)}{X_C(t)}\nonumber\\
&+{n_{1C}(t)}\frac{D_1(t)}{\tau X_C(t)}+\left(x(t)-\frac{kx(t)^2}{2}\right)\frac{D_2(t)}{X_C(t)}
-\frac{\gamma_C}{2}\left[ \left(\frac{n_{1C}(t)S_1(t)\sigma(t)}{X_C(t)}\right)^2\right.\nonumber\\
&\left.+\frac{2n_{1C}(t)n_{2C}(t)S_1(t)S_2(t)\sigma(t)^2}{X_C(t)^2} +\left(\frac{n_{2C}(t)S_2(t)}{X_C(t)}\right)^2(\sigma(t)^2+\delta(t)^2)\right].\label{JCC}
\end{align}
And similarly, the optimization problem (\ref{problemM}) for shareholder $M$ is equivalent to
\begin{equation}\label{problemMc}
\max\limits_{c_{M}(t)}\left\{\rho\frac{c_{M}(t)^{1-\gamma_M}-1}{1-\gamma_M}
-X_{M}(t)^{-\gamma_M}c_{M}(t)\right\};
\end{equation}
\begin{equation}\label{problemMn}
\max\limits_{n_{1M}(t),n_{2M}(t)}J_M(n_{1M}(t),n_{2M}(t)),
\end{equation}
where $J_M$ is given by
\begin{align}
&J_M(n_{1M}(t),n_{2M}(t))\nonumber\\
=&\sum_{j=1}^2\frac{n_{jM}(t)S_j(t)}{X_M(t)}(\mu_{j}(t)-r(t))
+{n_{1M}(t)}\frac{D_1(t)}{\tau X_M(t)}+n_{2M}(t)(1-x(t))\frac{D_2(t)}{X_M(t)}
-\frac{\gamma_M}{2}\left[ \left(\frac{n_{1M}(t)S_1(t)\sigma(t)}{X_M(t)}\right)^2\right.\nonumber\\
&\left.+\frac{2n_{1M}(t)n_{2M}(t)S_1(t)S_2(t)\sigma(t)^2}{X_M(t)^2} +\left(\frac{n_{2M}(t)S_2(t)}{X_M(t)}\right)^2(\sigma(t)^2+\delta(t)^2)\right].\label{JMM}
\end{align}

From now on, we would omit the index $t$ in this section for the sake of simplicity, i.e., whenever there is not any time index for the process $A$, we use it to represent $A(t)$ in SDEs and equations. Furthermore, the constraint (\ref{protectC}) deduces $n_{2C}\geq 0$ in the case of $0\leq p<1$, while it is naturally satisfied if conditions $p=1$ and $x=0$ are set in (\ref{problemCn}) and (\ref{JCC}). Summarizing, problems (\ref{problemCc}) and (\ref{problemMc}) are easy to solve, and we can rewrite problems (\ref{problemCn}) and (\ref{problemMn}) as the following optimization problems:
for the controlling shareholder in the case of $p=1$,
\begin{equation}\label{pCC}
\left\{
\begin{aligned}
\min \;&-J_{0C}(n_{1C},n_{2C})\equiv -J_{C}(0,n_{1C},n_{2C})\\
s.t.\;& g_1(n_{1C},n_{2C})\equiv -n_{1C}\leq 0;\\
&g_2(n_{1C},n_{2C})\equiv n_{1C}-\tau\leq 0;\\
&g_3(n_{1C},n_{2C})\equiv -n_{2C}\leq 0;\\
&g_4(n_{1C},n_{2C})\equiv n_{2C}- 1\leq 0,
\end{aligned}
\right.
\end{equation}
and in the case of $0\leq p<1$,
\begin{equation}\label{pC}
\left\{
\begin{aligned}
\min \;&-J_C(x,n_{1C},n_{2C})\\
s.t.\;&f_0(x,n_{1C},n_{2C})\equiv -x\leq 0;\\
& f_1(x,n_{1C},n_{2C})\equiv x-(1-p)n_{2C}\leq 0;\\
&f_2(x,n_{1C},n_{2C})\equiv -n_{1C}\leq 0;\\
&f_3(x,n_{1C},n_{2C})\equiv n_{1C}-\tau\leq 0;\\
&f_4(x,n_{1C},n_{2C})\equiv n_{2C}- 1\leq 0,
\end{aligned}
\right.
\end{equation}
and for the minority shareholder,
\begin{equation}\label{pM}
\min\limits_{n_{1M}\in \mathbb{R},n_{2M}\in \mathbb{R}} \;-J_M(n_{1M},n_{2M})
\end{equation}

For simplicity, we will always use the following notations in our paper. Denote the transpose of any matrix $A$ by $A^\top$. Let $\mathbf{e}_{ij}$ be the $2\times 2$-matrix with zero elements except that the $(i,j)$-element is one;
and set $\mathbf{e}_1=(1,0)^\top,\mathbf{e}_2=(0,1)^\top$ and for $i=C,M,$
\begin{align*}
&\mathbf{n}_i=\left(
\begin{aligned}
n_{1i}\\
n_{2i}
\end{aligned}
\right),
&&\mu=\left(
\begin{aligned}
\mu_1\\
\mu_2
\end{aligned}
\right),
&&\theta_i=\left(
\begin{aligned}
\theta_{1i}\\
\theta_{2i}
\end{aligned}
\right),
&&\xi_i=\left(
\begin{aligned}
&\xi_{1i}&&\quad \xi_{0i}\\
&\xi_{0i}&&\quad \xi_{2i}
\end{aligned}
\right),\\
%%%%%%%%%%%%%%%%%%%%%
&\alpha_{1i}=\frac{D_1}{\tau X_i},&&\alpha_{2i}=\frac{D_2}{X_i},
&&\theta_{1i}=\frac{S_1}{X_i}(\mu_{1}-r),&&\theta_{2i}=\frac{S_2}{X_i}(\mu_{2}-r),\\
%%%%%%%%%%%%%%%%%%%%%%%%%
&\xi_{0i}=\frac{\gamma_iS_1S_2\sigma^2}{X_i^2},&&\xi_{1i}=\gamma_i\left(\frac{S_1\sigma}{X_i}\right)^2,&&
\xi_{2i}=\gamma_i\left(\frac{S_2}{X_i}\right)^2(\sigma^2+\delta^2).&&
\end{align*}
Parameters $\alpha_{1i}$ and $\alpha_{2i}$ are ratios of output to shareholders' wealth, and similar ratios are also used in \cite{Basak}. Parameters $\theta_{1i}$ and $\theta_{1i}$ are ratios relative to stock risk premium of shareholders and parameters $\xi_{0i}, \xi_{1i}$ and $\xi_{2i}$ are ratios related to stock risk of shareholders, and we would discuss later how they determine shareholders' optimal consumption-portfolio strategies.
Then functions (\ref{JCC}) and (\ref{JMM}) can be rewritten as
\begin{align}
J_C(x,n_{1C},n_{2C})=&\mathbf{n}_C^\top\theta_C-\frac{1}{2}\mathbf{n}_C^\top\xi_C\mathbf{n}_C
+{n_{1C}}\alpha_{1C}+n_{2C}(1-x)\alpha_{2C}+\left(x-\frac{kx^2}{2}\right)\alpha_{2C};\label{JC}\\
J_M(n_{1M},n_{2M})=&\mathbf{n}_M^\top\theta_M-\frac{1}{2}\mathbf{n}_M^\top\xi_M\mathbf{n}_M
+{n_{1M}}\alpha_{1M}+n_{2M}(1-x)\alpha_{2M}. \label{JM}
\end{align}

Under the partial equilibrium setting where for $i=C,M$ and $j=1,2$, the processes $r,\mu_{ji},\sigma$ and $\delta$ are regarded as given processes, we can solve (\ref{pCC})-(\ref{pM}) to obtain shareholders' optimal strategies.
\begin{proposition}\label{th1}{\it
In the cross-sectional economy with perfect investor protection $p=1$, the optimal consumptions $\overline{c}_{i}^*$, the fraction of diverted output $\overline{x}^*$ and the optimal stock holding $\overline{\mathbf{n}}_{M}^*$ of the shareholder $M$ are given by
\begin{align}
\overline{c}_{i}^*=&\rho^{\frac{1}{\gamma_i}}X_{i},\quad i=C,M;\label{ps_c_p1}\\
\overline{x}^*=&0;\quad\quad \overline{\mathbf{n}}_M^*=(\xi_M)^{-1}\cdot\left[\theta_M+{\alpha_{1M}}\mathbf{e}_1+\alpha_{2M}\mathbf{e}_2\right],\nonumber
\end{align}
%\overline{\mathbf{n}}_{C}^*=&\overline{\eta}_{v}^*=(\overline{\eta}_{1,v}^*,\overline{\eta}_{2,v}^*)^\top,\quad (\mathrm{Region\;}v) \quad\text{if}\, J_{0C}^*=J_{0C}(\overline{\eta}_{v}^*), v=1,\cdots,9;\label{ps_nC_p1}\\
and the optimal stock holding $\overline{\mathbf{n}}_{C}^*$ of the shareholder $C$ exists uniquely and is determined as follows:
\begin{equation*}
\overline{\mathbf{n}}_{C}^*=\left\{
\begin{aligned}
%region 1
&\overline{\eta}_1^*=\left(\xi_C\right)^{-1}\cdot\left[\theta_C
+{\alpha_{1C}}\mathbf{e}_1+\alpha_{2C}\mathbf{e}_{2}\right],\quad\text{if} \;\; 0\leq\overline{\eta}_{1,1}^*\leq \tau, \;0\leq\overline{\eta}_{2,1}^*\leq 1;&(\mathrm{Region\; 1})\\
%region 2
&\overline{\eta}_2^*=\left(0,\frac{\theta_{2C}+\alpha_{2C}}{\xi_{2C}}
\right)^\top,\quad\text{if}\;\;
0\leq\overline{\eta}_{2,2}^*\leq1,
\;\frac{\theta_{1C}+{\alpha_{1C}}}{\xi_{0C}}<\overline{\eta}_{2,2}^*;&(\mathrm{Region\; 2})\\
%region 3
&\overline{\eta}_3^*=\left(\tau,\frac{\theta_{2C}+\alpha_{2C}-\tau\xi_{0C}}{\xi_{2C}}
\right)^\top,\quad\text{if}\;\;
0\leq\overline{\eta}_{2,3}^*\leq 1,\;\overline{\eta}_{2,3}^*<\frac{\theta_{1C}+{\alpha_{1C}}-\tau\xi_{1C} }{\xi_{0C}};&(\mathrm{Region\; 3})\\
%region 4
&\overline{\eta}_4^*=\left(\frac{\theta_{1C}+{\alpha_{1C}}}{\xi_{1C}},0\right)^\top,\quad\text{if}\;\;
0\leq\overline{\eta}_{1,4}^*\leq\tau,\;\theta_{2C}+\alpha_{2C}<\xi_{0C}\overline{\eta}_{1,4}^*;
&(\mathrm{Region\; 4})\\
%region 5
&\overline{\eta}_5^*=\left(\frac{\theta_{1C}+{\alpha_{1C}}-\xi_{0C}}
{\xi_{1C}},1\right)^\top, \quad\text{if}\;\;
\overline{\eta}_{1,5}^*<\frac{\theta_{2C}+\alpha_{2C}-\xi_{2C}}{\xi_{0C}},\;0\leq\overline{\eta}_{1,5}^*\leq \tau;&(\mathrm{Region\; 5})\\
 %region 6
&\overline{\eta}_{6}^*=\left(0,0\right)^\top,\quad\text{if}\;\;
      \theta_{1C}+{\alpha_{1C}}<0,\;\theta_{2C}+\alpha_{2C}<0;&(\mathrm{Region\; 6})\\
  %region 7
&\overline{\eta}_{7}^*=\left(0,1\right)^\top,\quad\text{if}\;\;
      \theta_{1C}+{\alpha_{1C}}<\xi_{0C},\;\xi_{2C}<\theta_{2C}+\alpha_{2C};&(\mathrm{Region\; 7})\\
  %region 8
&\overline{\eta}_{8}^*=\left(\tau,0\right)^\top,\quad\text{if}\;\;
      \tau\xi_{1C}<\theta_{1C}+{\alpha_{1C}},\;\theta_{2C}+\alpha_{2C}<\tau\xi_{0C};&(\mathrm{Region\; 8})\\
  %region 9
&\overline{\eta}_{9}^*=\left(\tau,1\right)^\top,\quad\text{if}\;\;
      \xi_{0C}+\tau\xi_{1C}<\theta_{1C}+{\alpha_{1C}},\;\tau\xi_{0C}+\xi_{2C}<\theta_{2C}+\alpha_{2C},&(\mathrm{Region\; 9})
      \end{aligned}
      \right.
\end{equation*}
where for $v=1,\cdots,9$, $\overline{\eta}_{v}^*=(\overline{\eta}_{1,v}^*,\overline{\eta}_{2,v}^*)^\top$.
}\end{proposition}

Since the minority shareholder's portfolio choice is simply similar to the controlling shareholder's in Region 1, we only discuss the case of the controlling shareholder. According to (\ref{X}), for the controlling shareholder in the cross-sectional economy with perfect investor protection (i.e., $p=1,x^*=0$), the risk premium related to his stock holdings could be divided into two parts: $\theta_{1C}+\alpha_{1C}$ caused by the stock and the output in firm $1$, and $\theta_{2C}+\alpha_{2C}$ caused by the stock and the output in firm $2$. However, holding stock shares also brings three kinds of risk: $\xi_{1C}$ caused by the stock variance in firm $1$, $\xi_{2C}$ caused by the stock variance in firm $2$, and $\xi_{0C}$ caused by the stock covariance between firm $1$ and firm $2$.  When the risk premium $\theta_{1C}+\alpha_{1C}$ (resp., $\theta_{2C}+\alpha_{2C}$) is negative, the stock in firm $1$ (resp., firm $2$) is not worth holding, i.e. $n^*_{1C}=0$\;(resp., $n^*_{2C}=0$). Hence, we focus on the case of $\theta_{1C}+\alpha_{1C}>0, \theta_{2C}+\alpha_{2C}>0$. If there is only firm $1$ (resp., firm $2$) in the economy, then it is shown in lots of researches, for example \cite{Basak,Merton}, that the optimal stock holdings are mainly determined by the ratio of the risk premium to the stock variance, i.e., the benchmark stock holding $\frac{\theta_{1C}+\alpha_{1C}}{\xi_{1C}}$
(resp., $\frac{\theta_{2C}+\alpha_{2C}}{\xi_{2C}}$) in Proposition \ref{th1}. In the cross-sectional economy, the controlling shareholder, however, has to take the covariance between two firms into account and balances the profit and related risk, which is threefold  as follows:
\begin{description}
  \item[(1)] If the benchmark stock holding $\frac{\theta_{1C}+\alpha_{1C}}{\xi_{1C}}$ in firm $1$ (resp., $\frac{\theta_{2C}+\alpha_{2C}}{\xi_{2C}}$ in firm $2$) is even greater than the maximal shares leading to much higher profit with much lower risk, i.e.,
\begin{equation}\label{conditiona}
\frac{\theta_{1C}+\alpha_{1C}}{\xi_{1C}}\geq \tau
\end{equation}
holds for related parameters (resp.,
\begin{equation}\label{conditionb}
\frac{\theta_{2C}+\alpha_{2C}}{\xi_{2C}}\geq 1
\end{equation}
holds for related parameters),
then the controlling shareholder would hold the maximal shares in firm $1$ (resp., firm $2$) in spite of the covariance risk of holding the stocks in firm $2$ (resp., firm $1$).
  \item[(2)] If holding the stocks in firm $1$ brings more profit with lower risk (resp., less profit with higher risk) and holding the stocks in firm $2$ simultaneously leads to less profit with higher covariance risk (resp., more profit with lower covariance risk), i.e.,
\begin{equation}\label{condition1}
\frac{\theta_{1C}+\alpha_{1C}}{\theta_{2C}+\alpha_{2C}}\geq\frac{\xi_{1C}}{\xi_{0C}}
\end{equation}
holds for related parameters (resp.,
\begin{equation}\label{condition2}
\frac{\theta_{1C}+\alpha_{1C}}{\theta_{2C}+\alpha_{2C}}<\frac{\xi_{1C}}{\xi_{0C}}
\end{equation}
holds for related parameters),
then the controlling shareholder would prefer the stocks in firm $1$ (resp., firm $2$) and hold no stocks in firm $2$ (resp., hold maximal shares of firm $2$ and keep less shares of firm $1$).
  \item[(3)] If holding the stocks in firm $2$ brings more profit with lower risk (resp., less profit with higher risk) and holding the stocks in firm $1$ simultaneously leads to less profit with higher covariance risk (resp., more profit with lower covariance risk), i.e.,
\begin{equation}\label{condition3}
\frac{\theta_{2C}+\alpha_{2C}}{\theta_{1C}+\alpha_{1C}}\geq\frac{\xi_{2C}}{\xi_{0C}}
\end{equation}
holds for related parameters (resp.,
\begin{equation}\label{condition4}
\frac{\theta_{2C}+\alpha_{2C}}{\theta_{1C}+\alpha_{1C}}<\frac{\xi_{2C}}{\xi_{0C}}
\end{equation}
holds for related parameters),
then the controlling shareholder would prefer the stocks in firm $2$ (resp., firm $1$) and hold no stocks in firm $1$ (resp., hold maximal shares of firm 1 and keep less shares of firm 2).
\end{description}

Therefore, the reasons why the controlling shareholder divides his optimal stock holdings into different regions in Proposition \ref{th1} can be explained as follows. In Region 1, none of (\ref{conditiona})-(\ref{condition4}) is satisfied such that no portfolio constraints are binding, so the controlling shareholder could optimally hold both stocks in two firms. In this region, $\theta_C
+{\alpha_{1C}}\mathbf{e}_1+\alpha_{2C}\mathbf{e}_{2}$ is the vector of the risk premium and $\xi_C$ is the related covariance matrix. In Region 2, (\ref{condition3}) is satisfied, and hence the controlling shareholder holds no stocks in firm 1 and holds the stocks in firm 2 optimally. In Region 3, the condition $\overline{\eta}_{2,3}^*<\frac{\theta_{1C}+{\alpha_{1C}}-\tau\xi_{1C} }{\xi_{0C}}$ gives
\[
\frac{\theta_{2C}+\alpha_{2C}}{\theta_{1C}+\alpha_{1C}}
+\tau\frac{\xi_{1C}\xi_{2C}-\xi_{0C}^2}{\xi_{0C}(\theta_{1C}+\alpha_{1C})}<\frac{\xi_{2C}}{\xi_{0C}},
\]
which implies (\ref{condition4}). Thus, the controlling shareholder holds maximal shares of the stocks in firm 1 and decreases shares of the stocks in firm 2. In Region 4, (\ref{condition1}) is satisfied, and hence the controlling shareholder holds no stocks in firm 2 and holds the stocks in firm 1 optimally. In Region 5, the condition $\overline{\eta}_{1,5}^*<\frac{\alpha_{2C}+\theta_{2C}-\xi_{2C}}{\xi_{0C}}$ gives
\[
\frac{\theta_{1C}+\alpha_{1C}}{\theta_{2C}+\alpha_{2C}}
+\frac{\xi_{1C}\xi_{2C}-\xi_{0C}^2}{\xi_{0C}(\theta_{2C}+\alpha_{2C})}<\frac{\xi_{1C}}{\xi_{0C}},
\]
which implies (\ref{condition2}). Thus, the controlling shareholder holds maximal shares of the stocks in firm 2 and decreases shares of the stocks in firm 1. In Region 6, since the risk premia $\theta_{1C}+\alpha_{1C}$ and $\theta_{2C}+\alpha_{2C}$ are both negative, the controlling shareholder holds the stocks neither in firm 1 nor in firm 2. In Region 7, for firm 2, since (\ref{conditionb}) holds, the controlling shareholder holds all stock shares in firm 2, and as for firm 1, the reason is twofold: if $\theta_{1C}+\alpha_{1C}\leq0$, then the controlling shareholder holds no stocks in firm 1, and if $\theta_{1C}+\alpha_{1C}>0$, then it is easy to see that (\ref{condition3}) is satisfied such that the controlling shareholder holds no stocks in firm 1. In Region 8, for firm 1, since (\ref{conditiona}) holds, the controlling shareholder holds all stock shares in firm 1, and as for firm 2, the reason is also twofold: if $\theta_{2C}+\alpha_{2C}\leq0$, then the controlling shareholder holds no stocks in firm 2, and if $\theta_{2C}+\alpha_{2C}>0$, then it is easy to see that (\ref{condition1}) holds such that the controlling shareholder holds no stocks in firm 2. In Region 9, since both (\ref{conditiona}) and (\ref{conditionb}) hold, the controlling shareholder holds all stock shares in both firm 1 and firm 2.

\begin{proposition}\label{th2}{\it
In the cross-sectional economy with investor protection $0\leq p<1$, the optimal consumptions $c_{i}^*$, the fraction of diverted output $x^*$ and the optimal stock holdings $\mathbf{n}_{i}^*$ for $i=C,M$ are given by
\begin{align}
c_{i}^*=&\rho^{\frac{1}{\gamma_i}}X_{i},\quad i=C,M;\label{ps_c}\\
x^*=&x^*(n_{2C}^*)=\min\left\{\frac{1-n_{2C}^*}{k},(1-p)n_{2C}^*\right\};\label{xx}\\
\mathbf{n}_{C}^*=&\eta_{v}^*=(\eta_{1,v}^*,\eta_{2,v}^*)^\top,\quad (\mathrm{Region\;}v) \quad\text{if}\; J_C^*=J_C(x^*,\eta_{v}^*), v=1,2,\cdots,12;\label{ps_nC}\\
\mathbf{n}_M^*=&(\xi_M)^{-1}\cdot\left[\theta_M+{\alpha_{1M}}\mathbf{e}_1+(1-x^*)\alpha_{2M}\mathbf{e}_2\right],\label{ps_nM}
\end{align}
where Regions $1-12$ are given by
\begin{description}
%region 1
  \item [] {$\mathrm{Region\; 1}:\; \eta_1^*=\left(\xi_C+(1-p)(2+k(1-p))\alpha_{2C}\mathbf{e}_{22}\right)^{-1}\cdot\left[\theta_C+
{\alpha_{1C}}\mathbf{e}_1+\left(2-p\right)\alpha_{2C}\mathbf{e}_{2}\right]$,\\
\text{\qquad\qquad\qquad\quad if}\; $0\leq\eta_{1,1}^*\leq \tau,\; 0\leq\eta_{2,1}^*<\frac{1}{1+(1-p)k}$;}
%region 2
  \item [] {$\mathrm{Region\; 2}:\;
 \eta_2^*=\left(0,\frac{\theta_{2C}+(2-p)\alpha_{2C}}{\xi_{2C}+(1-p)[2+(1-p)k]\alpha_{2C}}
\right)^\top$,\text{\quad if}\; $0\leq\eta_{2,2}^*\leq 1,\;\frac{\theta_{1C}+{\alpha_{1C}}}{\xi_{0C}}<\eta_{2,2}^*<\frac{1}{1+(1-p)k}$;}
%region 3
  \item [] {$\mathrm{Region\; 3}:\;
 \eta_3^*=\left(\tau,\frac{\theta_{2C}+(2-p)\alpha_{2C}-\tau\xi_{0C}}{\xi_{2C}+(1-p)[2+(1-p)k]\alpha_{2C}}
\right)^\top$,\text{\quad if}\;
$0\leq\eta_{2,3}^*<\frac{1}{1+(1-p)k},\;\eta_{2,3}^*<\frac{\theta_{1C}+{\alpha_{1C}}-\tau\xi_{1C} }{\xi_{0C}};$}
%region 4
  \item [] {$\mathrm{Region\; 4}:\;
  \eta_4^*=\left(\frac{\theta_{1C}+{\alpha_{1C}}}{\xi_{1C}},0\right)^\top$,\text{\quad if}\;
$0\leq\theta_{1C}+{\alpha_{1C}}\leq\tau\xi_{1C},\;\xi_{1C}(\theta_{2C}+(2-p)\alpha_{2C})<\xi_{0C}(\theta_{1C}+{\alpha_{1C}})$;}
%region 5
  \item [] {$\mathrm{Region\; 5}:\;
  \eta_5^*=\left(\frac{\theta_{1C}+{\alpha_{1C}}-\xi_{0C}}
{\xi_{1C}},1\right)^\top$,\text{\quad if}\;
$\eta_{1,5}^*<\frac{\theta_{2C}+\alpha_{2C}-\xi_{2C}}{\xi_{0C}},\;0\leq\eta_{1,5}^*\leq \tau$;}
%region 6
  \item [] {$\mathrm{Region\; 6}:\;
  \eta_{6}^*=\left(0,0\right)^\top$,\text{\quad  if}\; $\theta_{1C}+{\alpha_{1C}}<0,\;\theta_{2C}+(2-p)\alpha_{2C}<0$;}
  %region 7
  \item [] {$\mathrm{Region\; 7}:\;
  \eta_{7}^*=\left(0,1\right)^\top$,\text{\quad if}\;
  $\theta_{1C}+{\alpha_{1C}}<\xi_{0C},\;\xi_{2C}<\theta_{2C}+\alpha_{2C}$;}
  %region 8
  \item [] {$\mathrm{Region\; 8}:\;
  \eta_{8}^*=\left(\tau,0\right)^\top$,\text{\quad if}\;
  $\tau\xi_{1C}<\theta_{1C}+{\alpha_{1C}},\;\theta_{2C}+(2-p)\alpha_{2C}<\tau\xi_{0C}$;}
  %region 9
  \item [] {$\mathrm{Region\; 9}:\;
  \eta_{9}^*=\left(\tau,1\right)^\top$,\text{\quad if}\;
  $\xi_{0C}+\tau\xi_{1C}<\theta_{1C}+{\alpha_{1C}},\;\tau\xi_{0C}+\xi_{2C}<\theta_{2C}+\alpha_{2C}$;}
  %region 10
  \item [] {$\mathrm{Region\; 10}:\;
  \eta_{10}^*=\left(\xi_C-\frac{\alpha_{2C}}{k}\mathbf{e}_{22}\right)^{-1}\cdot\left[\theta_C
+{\alpha_{1C}}\mathbf{e}_1+\left(1-\frac{1}{k}\right)\alpha_{2C}\mathbf{e}_{2}\right]$,\text{\quad if}\;
$0\leq\eta_{1,10}^*\leq \tau$,\;\\
\text{\qquad\qquad\qquad\quad}\;$\frac{1}{1+(1-p)k}\leq \eta_{2,10}^*\leq 1,\;\det\left(\xi_C-\frac{\alpha_{2C}}{k}\mathbf{e}_{22}\right)\neq 0$;\\
  \text{\quad} or\;\; $\eta_{10}^*=(\eta_{1,10}^*,\eta_{2,10}^*)^\top$,\text{\quad if}\; $0\leq\eta_{1,10}^*\leq \tau,\frac{1}{1+(1-p)k}\leq \eta_{2,10}^*\leq 1,\det\left(\xi_C-\frac{\alpha_{2C}}{k}\mathbf{e}_{22}\right)=0$,\\
\text{\qquad\qquad\qquad\quad}\;$\xi_{1C}\eta_{1,10}^*+\xi_{0C}\eta_{2,10}^*=\theta_{1C}+{\alpha_{1C}},\;
\frac{\xi_{0C}(\theta_{1C}+{\alpha_{1C}})}{\xi_{1C}}
=\theta_{2C}+\left(1-\frac{1}{k}\right)\alpha_{2C}$;}
 %region 11
  \item [] {$\mathrm{Region\; 11}:\;
  \eta_{11}^*=\left(0,\frac{\theta_{2C}+(1-\frac{1}{k})\alpha_{2C}}{\xi_{2C}-
\frac{\alpha_{2C}}{k}}\right)^\top
$,\text{\quad if}\;
$\frac{1}{1+(1-p)k}\leq\eta_{2,11}^*\leq 1,\;\eta_{2,11}^*>\frac{\theta_{1C}+{\alpha_{1C}}}{\xi_{0C}},\;\xi_{2C}-
\frac{\alpha_{2C}}{k}\neq 0$;\\
%region 11a
  \text{\quad} or\;\; $\eta_{11}^*=(0,\eta_{2,11}^*)^\top$,\text{\quad if}\;
$\frac{1}{1+(1-p)k}\leq\eta_{2,11}^*\leq 1,\;\eta_{2,11}^*>\frac{\theta_{1C}+{\alpha_{1C}}}{\xi_{0C}},\;\xi_{2C}-
\frac{\alpha_{2C}}{k}=0,$\\
\text{\qquad\qquad\qquad\quad}\;$\theta_{2C}+(1-\frac{1}{k})\alpha_{2C}=0$;}
%region 12
  \item [] {$\mathrm{Region\; 12}:\;
  \eta_{12}^*=\left(\tau,\frac{\theta_{2C}+(1-\frac{1}{k})\alpha_{2C}-\tau\xi_{0C}}
{\xi_{2C}-\frac{\alpha_{2C}}{k}}\right)^\top$,\text{\quad if}\;$\frac{1}{1+(1-p)k}\leq\eta_{2,12}^*\leq 1,\;\eta_{2,12}^*<\frac{\theta_{1C}+{\alpha_{1C}}-\tau\xi_{1C}}{\xi_{0C}}$,\\
\text{\qquad\qquad\qquad\quad}\;$\xi_{2C}-
\frac{\alpha_{2C}}{k}\neq 0$;\\
%region 12a
  \text{\quad} or\;\; $\eta_{12}^*=(\tau,\eta_{2,12}^*)^\top$,\text{\quad if}\;$\frac{1}{1+(1-p)k}\leq\eta_{2,12}^*\leq 1,\;\eta_{2,12}^*<\frac{\theta_{1C}+{\alpha_{1C}}-\tau\xi_{1C}}{\xi_{0C}}$,\\
\text{\qquad\qquad\qquad\quad}\;$\xi_{2C}-
\frac{\alpha_{2C}}{k}= 0,\;\theta_{2C}+(1-\frac{1}{k})\alpha_{2C}-\tau\xi_{0C}=0$;}
\end{description}
and in Region $v$ we have $J_C(x^*_v,\eta_{v}^*)=-\infty$ if the related conditions do not hold, and $J_{C}^*$ is defined by $J_{C}^*=\max\{J_{C}(x^*_v,\eta_{v}^*):v=1,2,\cdots,12\}$.
Furthermore, if, in Region $v$ ($v=1,2,\cdots,11,12$), the condition $\det(\xi_C)>\frac{\alpha_{2C}\xi_{1C}}{k}$ holds and the related conditions are satisfied, then the optimal stock holding $\mathbf{n}_{C}^*$ exists uniquely and is located in Region $v$.
}\end{proposition}

Taking $p=1$ and $x^*=0$, Regions 1-9 in Proposition \ref{th2} essentially degenerate to Regions 1-9 in Proposition \ref{th1}, which is why we use the same region marks in two propositions. Comparing Proposition \ref{th2} with Proposition \ref{th1}, diverting output enables the controlling shareholder to have more optimization choices. For one thing, there are new regions in Proposition \ref{th2} related to imperfect protection directly. For another, owing to diverting output, the risk premium and related risk of controlling shareholder's wealth in firm 2 have been modified as $\theta_{2C}+(2-p)\alpha_{2C}$ and $\xi_{2C}+\alpha_{2C}(1-p)[2+(1-p)k]$ in the case of $x^*=(1-p)n_{2C}^*$ (or equivalently, $0\leq n_{2C}^*\leq\frac{1}{1+(1-p)k}$), and $\theta_{2C}+(1-\frac{1}{k})\alpha_{2C}$ and $\xi_{2C}-\frac{\alpha_{2C}}{k}$ in the case of $x^*=\frac{1-n_{2C}^*}{k}$ (or equivalently, $\frac{1}{1+(1-p)k}\leq n_{2C}^*\leq1$). In the former case, the investor protection constraint (\ref{protectC}) binds and  the controlling shareholder has lower stock holding and diverting output brings him more profit with higher risk, i.e., \[\theta_{2C}+(2-p)\alpha_{2C}>\theta_{2C},\;\xi_{2C}+\alpha_{2C}(1-p)[2+(1-p)k]>\xi_{2C},\]
while in the later case, the investor protection constraint (\ref{protectC}) does not bind and the controlling shareholder has higher stock holding and diverting output brings him more profit with lower risk, i.e.,
\[\theta_{2C}+\left(1-\frac{1}{k}\right)\alpha_{2C}>\theta_{2C},\;\xi_{2C}-\frac{\alpha_{2C}}{k}<\xi_{2C}.\]
Hence, similar to the discussion of Proposition \ref{th1}, the controlling shareholder balances the profit and related risk in a threefold way as follows:
\begin{description}
  \item[(1)] Conditions (\ref{conditiona})-(\ref{conditionb}) and their related conclusions still hold true, because they are irrelevant to $x^*$.
  \item[(2)] If holding the stocks in firm 1 brings more profit with lower risk (resp., less profit with higher risk) and holding the stocks in firm 2 leads to less profit with higher covariance risk (resp., more profit with lower covariance risk), i.e.,
\begin{equation}\label{con1}
\left\{
\begin{aligned}
&\frac{\theta_{1C}+\alpha_{1C}}{\theta_{2C}+(2-p)\alpha_{2C}}\geq\frac{\xi_{1C}}{\xi_{0C}},&& \text{if}\;x^*=(1-p)n_{2C}^*;\\
&\frac{\theta_{1C}+\alpha_{1C}}{\theta_{2C}+(1-\frac{1}{k})\alpha_{2C}}\geq\frac{\xi_{1C}}{\xi_{0C}},&& \text{if}\;x^*=\frac{1-n_{2C}^*}{k}
\end{aligned}
\right.
\end{equation}
holds for related parameters (resp.,
\begin{equation}\label{con2}
\left\{
\begin{aligned}
&\frac{\theta_{1C}+\alpha_{1C}}{\theta_{2C}+(2-p)\alpha_{2C}}<\frac{\xi_{1C}}{\xi_{0C}},&& \text{if}\;x^*=(1-p)n_{2C}^*;\\
&\frac{\theta_{1C}+\alpha_{1C}}{\theta_{2C}+(1-\frac{1}{k})\alpha_{2C}}<\frac{\xi_{1C}}{\xi_{0C}},&& \text{if}\;x^*=\frac{1-n_{2C}^*}{k}
\end{aligned}
\right.
\end{equation}
holds for related parameters),
then the controlling shareholder would prefer the stocks in firm 1 (resp., firm 2) and hold optimal shares of firm 1 and keep minimal shares of firm 2 (resp., hold maximal shares of firm 2 and keep less shares of firm 1).
  \item[(3)] If holding the stocks in firm 2 brings more profit with lower risk (resp., less profit with higher risk) and holding the stocks in firm 1 leads to less profit with higher covariance risk (resp., more profit with lower covariance risk), i.e.,
\begin{equation}\label{con3}
\left\{
\begin{aligned}
&\frac{\theta_{2C}+(2-p)\alpha_{2C}}{\theta_{1C}+\alpha_{1C}}\geq\frac{\xi_{2C}
+\alpha_{2C}(1-p)[2+(1-p)k]}{\xi_{0C}},&& \text{if}\;x^*=(1-p)n_{2C}^*;\\
&\frac{\theta_{2C}+(1-\frac{1}{k})\alpha_{2C}}{\theta_{1C}+\alpha_{1C}}\geq\frac{\xi_{2C}-
\frac{\alpha_{2C}}{k}}{\xi_{0C}},&& \text{if}\;x^*=\frac{1-n_{2C}^*}{k}
\end{aligned}
\right.
\end{equation}
holds for related parameters (resp.,
\begin{equation}\label{con4}
\left\{
\begin{aligned}
&\frac{\theta_{2C}+(2-p)\alpha_{2C}}{\theta_{1C}+\alpha_{1C}}<\frac{\xi_{2C}
+\alpha_{2C}(1-p)[2+(1-p)k]}{\xi_{0C}},&& \text{if}\;x^*=(1-p)n_{2C}^*;\\
&\frac{\theta_{2C}+(1-\frac{1}{k})\alpha_{2C}}{\theta_{1C}+\alpha_{1C}}<\frac{\xi_{2C}-
\frac{\alpha_{2C}}{k}}{\xi_{0C}},&& \text{if}\;x^*=\frac{1-n_{2C}^*}{k}
\end{aligned}
\right.
\end{equation}
holds for related parameters),
then the controlling shareholder would prefer the stocks in firm 2 (resp., firm 1) and hold optimal shares of firm 2 and keep minimal shares of firm 1 (resp., hold maximal shares of firm 1 and keep less shares of firm 2).
\end{description}

Thus, regions in Proposition \ref{th2} can be explained as follows. In Region 1, owing to the investor protection constraint, the controlling shareholder achieves the optimal strategy by modifying the vector of the risk premium and the related covariance matrix as $\theta_C+{\alpha_{1C}}\mathbf{e}_1+\left(2-p\right)\alpha_{2C}\mathbf{e}_{2}$ and $\xi_C+(1-p)(2+k(1-p))\alpha_{2C}\mathbf{e}_{22}$ respectively. In Region 2, the first case of (\ref{con3}) is satisfied, so the controlling shareholder holds no shares of firm 1 and optimal shares of firm 2. In Region 3, the condition $\eta_{2,3}^*<\frac{\theta_{1C}+{\alpha_{1C}}-\tau\xi_{1C} }{\xi_{0C}}$ gives
\[
\frac{\theta_{2C}+(2-p)\alpha_{2C}}{\theta_{1C}+\alpha_{1C}}
+\tau\frac{\xi_{1C}\xi_{2C}-\xi_{0C}^2}{\xi_{0C}(\theta_{1C}+\alpha_{1C})}
<\frac{\xi_{2C}+(1-p)(2+k(1-p))\alpha_{2C}}{\xi_{0C}},
\]
which implies the first case of (\ref{con4}). Thus, the controlling shareholder holds maximal shares of  firm 1 and decreases shares of firm 2.
In Regions 4-9, $x^*=0$ is always kept and making use of (\ref{conditiona})-(\ref{conditionb}) and (\ref{con1})-(\ref{con2}), the analysis is similar to Regions 4-9 in Proposition \ref{th1}. In Region 10, due to diverting output, the controlling shareholder adjusts the vector of the risk premium and the related covariance matrix as $\theta_C
+{\alpha_{1C}}\mathbf{e}_1+\left(1-\frac{1}{k}\right)\alpha_{2C}\mathbf{e}_{2}$ and $\xi_C-\frac{\alpha_{2C}}{k}\mathbf{e}_{22}$ respectively. In Region 11, the second case of (\ref{con3}) is satisfied, so the controlling shareholder holds no shares of firm 1 and optimal shares of firm 2. In Region 12, since $\frac{1}{(1-p)k+1}\leq\eta_{2,12}^*\leq 1$ means higher $\eta_{2,12}^*$ and ownership concentration usually causes higher $\xi_{2C}$ (see Section \ref{section4}), we may assume $\xi_{1C}(\xi_{2C}-\frac{\alpha_{2C}}{k})\geq\xi_{0C}^2$. Then  the condition $\eta_{2,12}^*<\frac{(\theta_{1C}+{\alpha_{1C}})-\tau\xi_{1C}}{\xi_{0C}}$ gives
the second case of (\ref{con4}). Thus, the controlling shareholder holds maximal shares of firm 1 and decreases shares of firm 2.

\begin{remark}{\it
The conditions in all regions of both Propositions \ref{th1} and \ref{th2} help us to know which regions the controlling shareholder's optimal stock holdings are located in, and some of them are not indispensable, for example, no conditions in related regions of \cite{Basak}. However, we shall see later in Subsection \ref{subsection3.3} that such conditions play a crucial role in survival analysis of shareholders.
}\end{remark}

\subsection{Asset prices in equilibrium}\label{subsection3.2}
We define the equilibrium in the cross-sectional economy with investor protection as follows.

\begin{definition}\label{equilibrium}{\it
An equilibrium in the cross-sectional economy with investor protection is the consumption-portfolio processes $c_i^*,b_i^*$ and $\mathbf{n}_{i}^*$ for $i=C,M$ and a price system $r,\mu,\sigma$ and $\delta$ satisfying the following market clearing conditions:
  \begin{align}
&n_{1C}^*+n_{1M}^*=\tau,\label{ec-1}\\
&n_{2C}^*+n_{2M}^*=1,\label{ec-2}\\
&b_{C}^*+b_{M}^*=0,\label{ec-3}\\
&c_{C}^*+c_{M}^*=\widehat{D}.\label{ec-4}
\end{align}
}\end{definition}

In order to obtain asset prices in equilibrium, we first define three state variables:
\[y(t)=\frac{c_M^*(t)}{\widehat{D}(t)},\;\; y_1(t)=\frac{\widehat{D}_1(t)}{\widehat{D}(t)},\;\;y_2(t)=\frac{\tau S_1(t)}{\tau S_1(t)+S_2(t)},\]
and it is seen $y\in[0,1],y_1\in(0,1)$ and $y_2\in(0,1)$.
The process $y$ is the ratio of the minority shareholder's share in the aggregate consumption to the aggregate output, and the process $y_1$ is the ratio of the output in firm 1 to the aggregate output, and the process $y_2$ is the ratio of the stock value in firm 1 to the aggregate stock value. The ratio $y$ is used in \cite{Basak}, and ratios $y$ and $y_1$ are adopted by \cite{Chabakauri}.
All processes (except $c_i^*,\;b_i^*,\;i=C, M$) in equilibrium of the cross-sectional economy are derived as functions of $y,y_1$ and $y_2$, and following \cite{Basak}, we assume and then verify that they satisfy the form of
\begin{align}
dy(t)=&\mu_{y}(t)dt+\sigma_{y}(t)dW(t)+\delta_{y}(t)dB(t),\label{dy}\\
dy_1(t)=&\mu_{1y}(t)dt+\sigma_{1y}(t)dW(t)+\delta_{1y}(t)dB(t),\label{dy1}\\
dy_2(t)=&\mu_{2y}(t)dt+\sigma_{2y}(t)dW(t)+\delta_{2y}(t)dB(t),\label{dy2}
\end{align}
where processes $\mu_{y},\sigma_{y},\delta_{y}$ and $\mu_{2y},\sigma_{2y},\delta_{2y}$ are functions of $y,y_1$ and $y_2$ to be determined in equilibrium and $\mu_{1y},\sigma_{1y}$ and $\delta_{1y}$ are exogenous functions of $y_1$. For the sake of simplicity, denote a parameter of shareholders' risk aversion by
\[\Gamma_0=y\rho^{-\frac{1}{\gamma_M}}+(1-y)\rho^{-\frac{1}{\gamma_C}}.\]

\begin{proposition}\label{th3}{\it
Provided $y\in(0,1)$ in equilibrium with perfect investor protection ($p=1$), the shareholders' optimal consumptions $\overline{c}_{i}^*$ ($i=C,M$) are given by (\ref{ps_c_p1}), and the stock growth rate $\overline{\mu}$, the interest rate $\overline{r}$, and the stock volatilities $\overline{\sigma}$ and $\overline{\delta}$ are given by:
\begin{align}
%\mu
\overline{\mu}=&\overline{r}(\mathbf{e}_1+\mathbf{e}_2)+\frac{\frac{\tau y}{y_2}\mathbf{e}_{11}+\frac{y}{1-y_2}\mathbf{e}_{22}}
{\rho^{\frac{1}{\gamma_M}}\Gamma_0}[\overline{\xi}_M(\tau\mathbf{e}_1+\mathbf{e}_2-\overline{\mathbf{n}}_C^*)
-\alpha_{1M}\mathbf{e}_1-\alpha_{2M}\mathbf{e}_2];\label{bmu}\\
%r
\overline{r}=&y_1\mu_{1D}+(1-y_1)\mu_{2D}-y_1\left(\rho^{\frac{1}{\gamma_C}}l_{1C}+\rho^{\frac{1}{\gamma_M}}l_{1M}\right)
-(1-y_1)\left(\rho^{\frac{1}{\gamma_C}}l_{2C}+\rho^{\frac{1}{\gamma_M}}l_{2M}\right)
+(1-y)\rho^{\frac{1}{\gamma_C}}+y\rho^{\frac{1}{\gamma_M}}\nonumber\\
-&\frac{\overline{\Gamma}_1}{\tau}\left[y_2\Gamma_0(\overline{\mu}_1-\overline{r})+(1-l_{1M}-l_{1C})y_1\right]
-\overline{\Gamma}_2[(1-y_2)\Gamma_0(\overline{\mu}_2-\overline{r})+(1-l_{2M}-l_{2C})(1-y_1)];\\
%\sigma
\overline{\sigma}=&\frac{\sigma_D}{\Gamma_0\left[\frac{y_2}{\tau}\overline{\Gamma}_1+(1-y_2)\overline{\Gamma}_2\right]};\\
%\delta
\overline{\delta}=&\frac{(1-y_1)\delta_D}{(1-y_2)\Gamma_0\overline{\Gamma}_2}\label{bdelta}
\end{align}
(it is seen from (\ref{bmu}) that $\overline{\mu}_1-\overline{r}$ and $\overline{\mu}_2-\overline{r}$ are irrelevant to $\overline{r}$, and here $\overline{\xi}_M$ is the notation $\xi_M$ related to $\overline{\sigma}$ and $\overline{\delta}$) with
\[\quad\overline{\Gamma}_1=\rho^{\frac{1}{\gamma_M}}(\tau-\overline{n}^*_{1C})+\rho^{\frac{1}{\gamma_C}}\overline{n}^*_{1C},
\quad \overline{\Gamma}_2=\rho^{\frac{1}{\gamma_M}}(1-\overline{n}^*_{2C})+\rho^{\frac{1}{\gamma_C}}\overline{n}^*_{2C}.
\]
The parameters $\overline{\mu}_y,\overline{\sigma}_y$ and $\overline{\delta}_y$ in (\ref{dy}) are given by
\begin{equation}\label{py}
\left\{
\begin{aligned}
%\mu_y
\overline{\mu}_y=&yr-yy_1\mu_{1D}-y(1-y_1)\mu_{2D}-\overline{\sigma}_y\sigma_D
-(1-y_1)\delta_D\overline{\delta}_y+\rho^{\frac{1}{\gamma_M}}
[l_{1M}y_1+l_{2M}(1-y_1)-y]\\
&+\rho^{\frac{1}{\gamma_M}}\frac{\tau-\overline{n}_{1C}^*}{\tau}
\left[{y_2}\Gamma_0(\overline{\mu}_1-\overline{r})+(1-l_{1M}-l_{1C})y_1\right]\\
&+\rho^{\frac{1}{\gamma_M}}(1-n_{2C}^*)[(1-y_2)\Gamma_0(\overline{\mu}_2-\overline{r})+(1-l_{2M}-l_{2C})(1-y_1)];\\
\overline{\sigma}_y=&\rho^{\frac{1}{\gamma_M}}\Gamma_0\left[(\tau-\overline{n}_{1C}^*)\frac{y_2}{\tau}+(1-\overline{n}_{2C}^*)(1-y_2)\right]\sigma-y\sigma_D;\\
\overline{\delta}_y=&\rho^{\frac{1}{\gamma_M}}\Gamma_0(1-\overline{n}_{2C}^*)(1-y_2)\overline{\delta}-y(1-y_1)\delta_D,
\end{aligned}
\right.
\end{equation}
and parameters $\mu_{1y},\sigma_{1y}$ and $\delta_{1y}$ in (\ref{dy1}) are given by
\begin{equation}\label{py1}
\overline{\mu}_{1y}=y_1(1-y_1)(\mu_{1D}-\mu_{2D})+y_1(1-y_1)^2(\delta_D)^2;
\quad\overline{\sigma}_{1y}=0;
\quad\overline{\delta}_{1y}=-y_1(1-y_1)\delta_D,
\end{equation}
and parameters $\mu_{2y},\sigma_{2y}$ and $\delta_{2y}$ in (\ref{dy2}) are given by
\begin{equation}\label{py2}
\overline{\mu}_{2y}=y_2(1-y_2)(\overline{\mu}_1-\overline{\mu}_2)+y_2(1-y_2)^2\overline{\delta}^2;
\quad\overline{\sigma}_{2y}=0;
\quad\overline{\delta}_{2y}=-y_2(1-y_2)\overline{\delta}.
\end{equation}
The minority shareholder's optimal stock holding $\overline{n}_{M}^*$ is given by
\begin{equation}
\overline{n}_{1M}^*=\tau-\overline{n}_{1C}^*,\quad \overline{n}_{2M}^*=1-\overline{n}_{2C}^*,\label{bnmt}
\end{equation}
and the controlling shareholder's optimal stock holding $\overline{n}_{C}^*$ can be obtained by solving fixed-point equation
\begin{align*}
\overline{\mathbf{n}}_C^*=&\mathop{\mathrm{argmax}}\limits_{\mathbf{n}_C\in[0,\tau]\times[0,1]}
\left\{\sum_{j=1}^2n_{jC}(\mu_{j}-r)\frac{S_j}{X_C}
+n_{2C}\frac{D_2}{X_C}+{n_{1C}}\frac{D_1}{\tau X_C}\right.\\
&\left.-\frac{\gamma_C}{2}\left[ \left(\frac{n_{1C}S_1\sigma}{X_C}\right)^2+\frac{2n_{1C}n_{2C}S_1S_2\sigma^2}{X_C^2} +\left(\frac{n_{2C}S_2}{X_C}\right)^2(\sigma^2+\delta^2)\right]\right\}
\end{align*}
with parameters $\overline{\mu},\overline{r},\overline{\sigma},\overline{\delta}$ given by (\ref{bmu})-(\ref{bdelta}) and the following ratios
\begin{equation}\label{ratio}
\left\{
\begin{aligned}
&\frac{D_1}{X_M}=\frac{(1-l_{1M}-l_{1C})y_1}{\rho^{-\frac{1}{\gamma_M}}y},
&&\frac{D_1}{X_C}=\frac{(1-l_{1M}-l_{1C})y_1}{\rho^{-\frac{1}{\gamma_C}}(1-y)};\\
&\frac{D_2}{X_M}=\frac{(1-l_{2M}-l_{2C})(1-y_1)}{\rho^{-\frac{1}{\gamma_M}}y},
&&\frac{D_2}{X_C}=\frac{(1-l_{2M}-l_{2C})(1-y_1)}{\rho^{-\frac{1}{\gamma_C}}(1-y)};\\
&\frac{S_1}{X_M}=\frac{\Gamma_0y_2}
{\tau\rho^{-\frac{1}{\gamma_M}}y},
&&\frac{S_1}{X_C}=\frac{\Gamma_0y_2}
{\tau\rho^{-\frac{1}{\gamma_C}}(1-y)};\\
&\frac{S_2}{X_M}=\frac{\Gamma_0(1-y_2)}
{\rho^{-\frac{1}{\gamma_M}}y},
&&\frac{S_2}{X_C}=\frac{\Gamma_0(1-y_2)}
{\rho^{-\frac{1}{\gamma_C}}(1-y)}.
\end{aligned}\right.
\end{equation}
}
\end{proposition}

\begin{proposition}\label{th4}{\it
Provided $p\in[0,1)$ and $y\in(0,1)$ in equilibrium, the shareholders' optimal consumptions $c_{i}^*$ ($i=C,M$) are given by (\ref{ps_c}), and the stock growth rate $\mu$, the interest rate $r$, and the stock volatilities $\sigma$ and $\delta$ are given by:
\begin{align}
%\mu
\mu=&r(\mathbf{e}_1+\mathbf{e}_2)+\frac{\frac{\tau y}{y_2}\mathbf{e}_{11}+\frac{y}{1-y_2}\mathbf{e}_{22}}
{\rho^{\frac{1}{\gamma_M}}\Gamma_0}[\xi_M(\tau\mathbf{e}_1+\mathbf{e}_2-\mathbf{n}_C^*)
-\alpha_{1M}\mathbf{e}_1-(1-x^*)\alpha_{2M}\mathbf{e}_2];\label{mu}\\
%r
r=&y_1\mu_{1D}+(1-y_1)\mu_{2D}-y_1\left(\rho^{\frac{1}{\gamma_C}}l_{1C}+\rho^{\frac{1}{\gamma_M}}l_{1M}\right)
-(1-y_1)\left(\rho^{\frac{1}{\gamma_C}}l_{2C}+\rho^{\frac{1}{\gamma_M}}l_{2M}\right)\nonumber\\
&-\frac{\Gamma_1}{\tau}\left[y_2\Gamma_0(\mu_1-r)+(1-l_{1M}-l_{1C})y_1\right]
-\Gamma_2[(1-y_2)\Gamma_0(\mu_2-r)+(1-x^*)(1-l_{2M}-l_{2C})(1-y_1)]\nonumber\\
&+\rho^{\frac{1}{\gamma_C}}(1-y)+\rho^{\frac{1}{\gamma_M}}y
-(1-l_{2M}-l_{2C})(1-y_1)\left[\rho^{\frac{1}{\gamma_C}}\left(x^*-\frac{k}{2}(x^*)^2\right)
+\rho^{\frac{1}{\gamma_M}}\frac{k}{2}(x^*)^2\right];\label{r}\\
%\sigma
\sigma=&\frac{\sigma_D}{\Gamma_0\left[\frac{y_2}{\tau}\Gamma_1+(1-y_2)\Gamma_2\right]};\label{sigma}\\
%\delta
\delta=&\frac{(1-y_1)\delta_D}{(1-y_2)\Gamma_0\Gamma_2}\label{delta}
\end{align}
(it is seen from (\ref{mu}) that $\mu_1-r$ and $\mu_2-r$ are irrelevant to $r$) with
\[\Gamma_1=\rho^{\frac{1}{\gamma_M}}(\tau-n_{1C}^*)+\rho^{\frac{1}{\gamma_C}}n_{1C}^*,\quad
\Gamma_2=\rho^{\frac{1}{\gamma_M}}(1-n_{2C}^*)+\rho^{\frac{1}{\gamma_C}}n_{2C}^*.\]
The parameters $\mu_y,\sigma_y$ and $\delta_y$ in (\ref{dy}) are given by
\begin{equation}\label{y}
\left\{
\begin{aligned}
%\mu_y
\mu_y=&yr-yy_1\mu_{1D}-y(1-y_1)\mu_{2D}-\sigma_y\sigma_D-(1-y_1)\delta_D\delta_y+\rho^{\frac{1}{\gamma_M}}
[l_{1M}y_1+l_{2M}(1-y_1)-y]\\
&+\frac{k}{2}(x^*)^2\rho^{\frac{1}{\gamma_M}}(1-l_{2M}-l_{2C})(1-y_1)+\rho^{\frac{1}{\gamma_M}}\frac{\tau-n_{1C}^*}{\tau}
\left[{y_2}\Gamma_0(\mu_1-r)+(1-l_{1M}-l_{1C})y_1\right]\\
&+\rho^{\frac{1}{\gamma_M}}(1-n_{2C}^*)[(1-y_2)\Gamma_0(\mu_2-r)+(1-x^*)(1-l_{2M}-l_{2C})(1-y_1)];\\
\sigma_y=&\rho^{\frac{1}{\gamma_M}}\Gamma_0\left[(\tau-n_{1C}^*)\frac{y_2}{\tau}+(1-n_{2C}^*)(1-y_2)\right]\sigma-y\sigma_D;\\
\delta_y=&\rho^{\frac{1}{\gamma_M}}\Gamma_0(1-n_{2C}^*)(1-y_2)\delta-y(1-y_1)\delta_D,
\end{aligned}
\right.
\end{equation}
and parameters $\mu_{1y},\sigma_{1y}$ and $\delta_{1y}$ in (\ref{dy1}) are given by
\begin{equation}\label{y1}
\mu_{1y}=y_1(1-y_1)(\mu_{1D}-\mu_{2D})+y_1(1-y_1)^2(\delta_D)^2;
\quad\sigma_{1y}=0;
\quad\delta_{1y}=-y_1(1-y_1)\delta_D,
\end{equation}
and parameters $\mu_{2y},\sigma_{2y}$ and $\delta_{2y}$ in (\ref{dy2}) are given by
\begin{equation}\label{y2}
\mu_{2y}=y_2(1-y_2)(\mu_1-\mu_2)+y_2(1-y_2)^2\delta^2;
\quad\sigma_{2y}=0;
\quad\delta_{2y}=-y_2(1-y_2)\delta.
\end{equation}
The minority shareholder's optimal stock holding $n_{M}^*$ is given by
\begin{equation}
n_{1M}^*=\tau-n_{1C}^*,\quad n_{2M}^*=1-n_{2C}^*,\label{nmt}
\end{equation}
and the controlling shareholder's optimal stock holding $n_{C}^*$ can be obtained by solving fixed-point equation
\begin{align*}
\mathbf{n}_C^*=&\mathop{\mathrm{argmax}}\limits_{\mathbf{n}_C\in[0,\tau]\times[0,1]}
\left\{\sum_{j=1}^2n_{jC}(\mu_{j}-r)\frac{S_j}{X_C}
+n_{2C}(1-x^*(n_{2C}))\frac{D_2}{X_C}+{n_{1C}}\frac{D_1}{\tau X_C}\right.\\
&\left.+\left(x^*(n_{2C})-\frac{kx^*(n_{2C})^2}{2}\right)\frac{D_2}{X_C}
-\frac{\gamma_C}{2}\left[ \left(\frac{n_{1C}S_1\sigma}{X_C}\right)^2+\frac{2n_{1C}n_{2C}S_1S_2\sigma^2}{X_C^2} +\left(\frac{n_{2C}S_2}{X_C}\right)^2(\sigma^2+\delta^2)\right]\right\}
\end{align*}
with ratios (\ref{ratio}), $x^*(n_{2C})=\min\left\{\frac{1-n_{2C}}{k},(1-p)n_{2C}\right\}$ and parameters ${\mu},{r},{\sigma},{\delta}$ given by (\ref{mu})-(\ref{delta}).

}\end{proposition}

Compared with the economy with a single firm in \cite{Basak}, the effect of cross-section is clear: not only do additional stock holdings of firm 2 emerge from the shareholders' consumption-portfolio problems, but the parameters in equilibrium could be divided into two parts---one related to parameters in the firm itself and the other associated with parameters in the other firm. Indeed, directly computing in Propositions \ref{th3} and \ref{th4} gives
\begin{align}
\overline{\mu}_1-\overline{r}=&\bigg\{\frac{\tau y}{y_2\rho^{\frac{1}{\gamma_M}}\Gamma_0}
\left[\xi_{1M}(\tau-\overline{n}_{1C}^*)-\alpha_{1M}\right]\bigg\}+\bigg\{\frac{\tau y}{y_2\rho^{\frac{1}{\gamma_M}}\Gamma_0}\xi_{0M}(1-\overline{n}_{2C}^*)\bigg\},\nonumber\\
\overline{\mu}_2-\overline{r}=&\bigg\{\frac{ y}{(1-y_2)\rho^{\frac{1}{\gamma_M}}\Gamma_0}
\left[\xi_{2M}(1-\overline{n}_{2C}^*)-\alpha_{2M}\right]\bigg\}+\bigg\{\frac{ y}{(1-y_2)\rho^{\frac{1}{\gamma_M}}\Gamma_0}\xi_{0M}(\tau-\overline{n}_{1C}^*)\bigg\},\nonumber\\
%r
\overline{r}=&\bigg\{\rho^{\frac{1}{\gamma_M}}y+y_1\mu_{1D}-y_1\left(\rho^{\frac{1}{\gamma_C}}l_{1C}+\rho^{\frac{1}{\gamma_M}}l_{1M}\right)
-\frac{\overline{\Gamma}_1}{\tau}\left[y_2\Gamma_0(\overline{\mu}_1-\overline{r})+(1-l_{1M}-l_{1C})y_1\right]\bigg\}\nonumber\\
&+\bigg\{\rho^{\frac{1}{\gamma_C}}(1-y)+(1-y_1)\mu_{2D}-(1-y_1)\left(\rho^{\frac{1}{\gamma_C}}l_{2C}+\rho^{\frac{1}{\gamma_M}}l_{2M}\right)\nonumber\\
&-\overline{\Gamma}_2[(1-y_2)\Gamma_0(\overline{\mu}_2-\overline{r})+(1-l_{2M}-l_{2C})(1-y_1)]\bigg\},\nonumber\\
%\sigma
\frac{1}{\overline{\sigma}}=&\bigg\{\frac{\Gamma_0}{\sigma_D}{\frac{y_2}{\tau}\overline{\Gamma}_1\bigg\}
+\bigg\{\frac{\Gamma_0}{\sigma_D}(1-y_2)\overline{\Gamma}_2}\bigg\}
\end{align}
in the case of $p=1$, and
\begin{align}
{\mu}_1-{r}=&\bigg\{\frac{\tau y}{y_2\rho^{\frac{1}{\gamma_M}}\Gamma_0}
\left[\xi_{1M}(\tau-{n}_{1C}^*)-\alpha_{1M}\right]\bigg\}+\bigg\{\frac{\tau y}{y_2\rho^{\frac{1}{\gamma_M}}\Gamma_0}\xi_{0M}(1-{n}_{2C}^*)\bigg\},\nonumber\\
{\mu}_2-{r}=&\bigg\{\frac{ y}{(1-y_2)\rho^{\frac{1}{\gamma_M}}\Gamma_0}
\left[\xi_{2M}(1-{n}_{2C}^*)-(1-x^*)\alpha_{2M}\right]\bigg\}+\bigg\{\frac{ y}{(1-y_2)\rho^{\frac{1}{\gamma_M}}\Gamma_0}\xi_{0M}(\tau-{n}_{1C}^*)\bigg\},\nonumber\\
%r
{r}=&\bigg\{\rho^{\frac{1}{\gamma_M}}y+y_1\mu_{1D}-y_1\left(\rho^{\frac{1}{\gamma_C}}l_{1C}+\rho^{\frac{1}{\gamma_M}}l_{1M}\right)
-\frac{{\Gamma}_1}{\tau}\left[y_2\Gamma_0(\mu_1-r)+(1-l_{1M}-l_{1C})y_1\right]\bigg\}\nonumber\\
&+\bigg\{\rho^{\frac{1}{\gamma_C}}(1-y)+(1-y_1)\mu_{2D}-(1-y_1)\left(\rho^{\frac{1}{\gamma_C}}l_{2C}+\rho^{\frac{1}{\gamma_M}}l_{2M}\right)\nonumber\\
&-{\Gamma}_2[(1-y_2)\Gamma_0(\mu_2-r)+(1-x^*)(1-l_{2M}-l_{2C})(1-y_1)]\nonumber\\
&-(1-l_{2M}-l_{2C})(1-y_1)\left[\rho^{\frac{1}{\gamma_C}}\left(x^*-\frac{k}{2}(x^*)^2\right)
+\rho^{\frac{1}{\gamma_M}}\frac{k}{2}(x^*)^2\right]\bigg\},\nonumber\\
%\sigma
\frac{1}{{\sigma}}=&\bigg\{\frac{\Gamma_0}{\sigma_D}{\frac{y_2}{\tau}{\Gamma}_1\bigg\}
+\bigg\{\frac{\Gamma_0}{\sigma_D}(1-y_2){\Gamma}_2}\bigg\}
\end{align}
in the case of $0\leq p<1$. On the other hand, the parameters of nonfundamental risk (i.e., $\overline{\delta}$ and $\delta$) are formally independent of firm $1$ and directly related to firm $2$. This can be understood easily, because nonfundamental risk only affects firm $2$ in the economy setting.

\begin{remark}\label{remark3.2}{\it
Taking advantage of Proposition \ref{th2} (respectively, Proposition \ref{th1}), the controlling shareholder's optimal stock holdings can be obtained by considering regions therein in the following steps:
\begin{description}
  \item[Step 1]{Substitute ratios (\ref{ratio}) and (\ref{mu})-(\ref{delta}) in each Region $v$ for $v=1,2,\cdots,12$ (resp., (\ref{bmu})-(\ref{bdelta}) in each Region $v$ for $v=1,2,\cdots,9$) and then solve related fixed-point equation to obtain $\eta_v^*$ if it exists.}
  \item[Step 2]{Define related optimization function in Region $v$
  \begin{align*}
  J^v(n_1,n_2)=&\sum_{j=1}^2n_{j}(\mu_{j}-r)\frac{S_j}{X_C}
+n_{2}(1-x^*(n_{2}))\frac{D_2}{X_C}+{n_{1}}\frac{D_1}{\tau X_C}\\
+&\left(x^*(n_{2})-\frac{kx^*(n_{2})^2}{2}\right)\frac{D_2}{X_C}
-\frac{\gamma_C}{2}\left[ \left(\frac{n_{1}S_1\sigma}{X_C}\right)^2+\frac{2n_{1}n_{2}S_1S_2\sigma^2}{X_C^2} +\left(\frac{n_{2}S_2}{X_C}\right)^2(\sigma^2+\delta^2)\right]
\end{align*}
where $x^*(n_{2})=\min\left\{\frac{1-n_{2}}{k},(1-p)n_{2}\right\}$, and related ratios and parameters ${\mu},{r},{\sigma},{\delta}$ are given by (\ref{ratio}) and (\ref{mu})-(\ref{delta}) with $\mathbf{n}_{C}^*=\eta_v^*$ obtained in Step 1. (Respectively, define
 \begin{align*}
  \overline{J}^v(n_1,n_2)=&\sum_{j=1}^2n_{j}(\overline{\mu}_{j}-\overline{r})\frac{S_j}{X_C}
+n_{2}\frac{D_2}{X_C}+{n_{1}}\frac{D_1}{\tau X_C}\\
&-\frac{\gamma_C}{2}\left[ \left(\frac{n_{1}S_1\overline{\sigma}}{X_C}\right)^2+\frac{2n_{1}n_{2}S_1S_2\overline{\sigma}^2}{X_C^2} +\left(\frac{n_{2}S_2}{X_C}\right)^2(\overline{\sigma}^2+\overline{\delta}^2)\right]
\end{align*}
where related ratios and parameters $\overline{\mu},{r},\overline{\sigma},\overline{\delta}$ are given by (\ref{ratio}) and (\ref{bmu})-(\ref{bdelta}) with $\overline{\mathbf{n}}_{C}^*=\overline{\eta}_v^*$ obtained in Step 1.) With above parameters, Proposition \ref{th2} (resp., Proposition \ref{th1}) also shows all candidate maximum-points of $J^v$ (resp., $\overline{J}^v$), denoted by $\eta^v_j$ for $j=1,2,\cdots,12$ (resp., $\overline{\eta}^v_j$ for $j=1,2,\cdots,9$). Noticing that $\eta_v^*=\eta_v^v\;(\text{resp.,}\;\overline{\eta}_v^*=\overline{\eta}_v^v).$}
  \item[Step 3] {In Region $v$, if
  \begin{equation}\label{verify1}
J^v(\eta_v^*)=\max\{J^v(\eta_j^v): j=1,2,\cdots,12\}\end{equation}
  \begin{equation}\label{verify2}
(\text{resp.,}\;\overline{J}^v(\overline{\eta}_v^*)=\max\{\overline{J}^v(\overline{\eta}_j^v): j=1,2,\cdots,9\})\end{equation}
  holds, then the controlling shareholder's optimal stock holding $\mathbf{n}^*_{C}$ (resp., $\overline{\mathbf{n}}^*_{C}$) is located in Region $v$.
  }
\end{description}

}\end{remark}

\begin{remark}\label{remark3.3}{\it
Due to the complexity of the equilibrium, there may not exist any solution $\eta_v^*$ (resp., $\overline{\eta}_v^*$) to the related fixed-point equation for some Regions $v$ in Step 1 of Remark \ref{remark3.2}. Fortunately, for the region where there exists a solution $\eta_v^*$ to the related fixed-point equation, whether such solution is the controlling shareholder's optimal stock holding in equilibrium is irrelevant to the existence of solution in other regions. Therefore, for Region $v$, if the solution $\eta_v^*$ (resp., $\overline{\eta}_v^*$) of Step 1 exists and satisfies (\ref{verify1})(resp., (\ref{verify2})), then $\eta_v^*$ (resp., $\overline{\eta}_v^*$) is indeed the controlling shareholder's optimal stock holding in equilibrium.
}\end{remark}

\begin{remark}\label{boundary}{\it
Since the optimal stock holdings in equilibrium may not exist uniquely (and even do not exist), the parameters $\mu,r,\sigma$ and $\delta$ could be discontinuous functions of $y,y_1$ or $y_2$. In our paper, we will, however, assume that the optimal stock holdings exist uniquely and that the parameters $\mu,r,\sigma$ and $\delta$ are continuous functions of $y,y_1$ and $y_2$ in Propositions \ref{th3} and \ref{th4}. As a result, the parameters in the extreme cases of $y=1$ and $y=0$ (or equivalently, $X_C=0$ and $X_M=0$) can be defined as the limits of associated parameters in the cases of $y\in(0,1)$, and similar methods are used in \cite{Chabakauri}.
In brief, for any process $A$ we use $A\langle0\rangle$ (resp., $A\langle1\rangle$) to denote the limit of $A$ when $y$ converges to $0$ (resp., $1$).
The case of $y=1$ is easy to deal with and the related parameters can be obtained by substituting $\mathbf{n}_C^*=(0,0)^\top$ and $y=1$ in (\ref{mu})-(\ref{delta}) of Proposition \ref{th4} and $\frac{D_j}{X_M},\frac{S_j}{X_M},j=1,2$ of (\ref{ratio}) respectively, because aforementioned parameters are actually well-defined functions in $\mathbf{n}^*=(0,0)^\top$ and $y=1$. The results are list as follows:
\begin{align}
\mathbf{n}^*\langle1\rangle&=(0,0)^\top;\nonumber\\
\mu\langle1\rangle&=\left\{(\mathbf{e}_1+\mathbf{e}_2)r+\left(\frac{\tau \mathbf{e}_{11}}{y_2}+\frac{\mathbf{e}_{22}}{1-y_2}\right)[\xi_M(\tau\mathbf{e}_1+\mathbf{e}_2)
-\alpha_{1M}\mathbf{e}_1-\alpha_{2M}\mathbf{e}_2]\right\}\langle1\rangle;\nonumber\\
r\langle1\rangle&=y_1\mu_{1D}+(1-y_1)\mu_{2D}+(\rho^{\frac{1}{\gamma_M}}-\rho^{\frac{1}{\gamma_C}})(y_1l_{1C}+(1-y_1)l_{2C})
-(y_2\mathbf{e}_1^\top+(1-y_2)\mathbf{e}_2^\top)\Delta; \nonumber\\
\sigma\langle1\rangle&=\sigma_D;\nonumber\\
\delta\langle1\rangle&=\frac{1-y_1}{1-y_2}\delta_D;\nonumber\\
\mu_{y}\langle1\rangle&=-\rho^{\frac{1}{\gamma_C}}\left[y_1l_{1C}+(1-y_1)l_{2C}\right];\label{a1}\\
\sigma_y\langle1\rangle&=0;\label{a2}\\
\delta_y\langle1\rangle&=0,\label{a3}
\end{align}
where $\Delta=\left\{\left(\frac{\tau \mathbf{e}_{11}}{y_2}+\frac{\mathbf{e}_{22}}{1-y_2}\right)[\xi_M(\tau\mathbf{e}_1+\mathbf{e}_2)
-\alpha_{1M}\mathbf{e}_1-\alpha_{2M}\mathbf{e}_2]\right\}\langle1\rangle$ and the results for Proposition \ref{th3} are same as those in Proposition \ref{th4}. The case of $y=0$ is more complex, and the results are listed as follows (note that the former and later cases of Regions 10 and 12 are discussed separately, and the proofs can be found in Appendices, and for the results in Proposition \ref{th3}, we use the same notations as Proposition \ref{th4}):
\begin{align}
\mathbf{n}^*\langle0\rangle&=(\tau,1)^\top;\label{b1}\\
\mu\langle0\rangle&=(\mathbf{e}_1+\mathbf{e}_2)r\langle0\rangle+\Lambda;\label{b2}\\
r\langle0\rangle&=y_1\mu_{1D}+(1-y_1)\mu_{2D}-(\rho^{\frac{1}{\gamma_M}}-\rho^{\frac{1}{\gamma_C}})(y_1l_{1M}+(1-y_1)l_{2M})
-(y_2\mathbf{e}_1^\top+(1-y_2)\mathbf{e}_2^\top)\Lambda;\label{b3}\\
\sigma\langle0\rangle&=\sigma_D;\label{b4}\\
\delta\langle0\rangle&=\frac{1-y_1}{1-y_2}\delta_D;\label{b5}\\
\mu_{y}\langle0\rangle&=\rho^{\frac{1}{\gamma_M}}\left[y_1l_{1M}+(1-y_1)l_{2M}\right];\label{b6}\\
\sigma_y\langle0\rangle&=0;\label{b7}\\
\delta_y\langle0\rangle&=0,\label{b8}
\end{align}
where the process $\Lambda$ is given by following $\Lambda^v=(\Lambda^v_1,\Lambda^v_2)^\top\;$($v=5,9$ in the case of Proposition \ref{th3}, and $v=5,9,10,12$ in the case of Proposition \ref{th4}) if there exists a constant $0<\delta_v<1$ such that $\mathbf{n}_C^*$ is located in Region $v$ for all $y\in(0,\delta_v)$:
\begin{align}
&\Lambda^5:
\left\{
\begin{aligned}
\Lambda^5_1=&\left\{\frac{1}{S_1/X_C}(\tau\xi_{1C}-\alpha_{1C}+\xi_{0C})\right\}\langle0\rangle,\\
\Lambda^5_2=&\left\{\frac{\xi_{0M}/\xi_{1M}}{S_2/X_M}\bigg(\frac{S_1}{X_M}\Lambda^5_1+\alpha_{1M}\bigg)
-\frac{\alpha_{2M}}{S_2/X_M}\right\}\langle0\rangle;
\end{aligned}
\right.&& \label{bb5}\\
&\Lambda^9=-\rho^{\frac{1}{\gamma_C}}\left(
\begin{aligned}
&\frac{y_1}{y_2}(1-l_{1M}-l_{1C})\\
&\frac{1-y_1}{1-y_2}(1-l_{2M}-l_{2C})
\end{aligned}
\right);&& \label{bb9}\\
&\Lambda^{10}=\left\{\left(\frac{\tau\mathbf{e}_{11}}{y_2}
+\frac{\mathbf{e}_{22}}{1-y_2}\right)(\xi_C(\tau\mathbf{e}_{1}+\mathbf{e}_{2})
-\alpha_{1C}\mathbf{e}_{1}-\alpha_{2C}\mathbf{e}_{2})\right\}\langle0\rangle,&&\text{if\;} \det\left(\xi_C-\frac{\alpha_{2C}}{k}\mathbf{e}_{22}\right)=0; \label{bb10}\\
&\Lambda^{10}:
\left\{
\begin{aligned}
\Lambda^{10}_1=&\left\{\frac{1}{S_1/X_C}(\xi_{0C}+\tau\xi_{1C}-\alpha_{1C})\right\}\langle0\rangle,\\
\Lambda^{10}_2=&\left\{\frac{\xi_{0C}/\xi_{1C}}{S_2/X_C}(\xi_{0C}+\tau\xi_{1C})
-\left(1-\frac{1}{k}\right)\frac{\alpha_{2C}}{S_2/X_C}\right\}\langle0\rangle,
\end{aligned}
\right.&&\text{if\;} \det\left(\xi_C-\frac{\alpha_{2C}}{k}\mathbf{e}_{22}\right)\neq0;\label{bb0}\\
&\Lambda^{12}:
\left\{
\begin{aligned}
\Lambda^{12}_1=&\left\{\frac{\xi_{0M}/\xi_{2M}}{S_1/X_M}\left(\frac{S_2}{X_M}\Lambda^{12}_2+\alpha_{2M}\right)
-\frac{\alpha_{1M}}{S_1/X_M}\right\}\langle0\rangle,\\
\Lambda^{12}_2=&\left\{\frac{1}{S_2/X_C}(\xi_{2C}+\tau\xi_{0C}-\alpha_{2C})\right\}\langle0\rangle,
\end{aligned}
\right.&& \text{if\;} \;\xi_{2C}-\frac{\alpha_{2C}}{k}= 0;\\
&\Lambda^{12}:
\left\{
\begin{aligned}
\Lambda^{12}_1=&\left\{\frac{\xi_{0M}/\xi_{2M}}{S_1/X_M}\left(\frac{S_2}{X_M}\Lambda^{12a}_2+\alpha_{2M}\right)
-\frac{\alpha_{1M}}{S_1/X_M}\right\}\langle0\rangle,\\
\Lambda^{12}_2=&\left\{\frac{1}{S_2/X_C}(\tau\xi_{0C}-\left(1-\frac{1}{k}\right)\alpha_{2C})\right\}\langle0\rangle,
\end{aligned}
\right. &&\text{if\;} \;\xi_{2C}-\frac{\alpha_{2C}}{k}\neq 0.
\end{align}\label{bb12a}
Note that, from ratios in (\ref{ratio}), all limitations above are well-defined.
}\end{remark}

\subsection{Survival analysis of shareholders}\label{subsection3.3}
Remark \ref{boundary} shows that the volatilities of the state process $y$ (i.e., $\sigma_y$ and $\delta_y$) are zero at the boundaries $y=1$ and $y=0$, while the drift
$\mu_{y}\langle1\rangle=-\rho^{\frac{1}{\gamma_C}}\left[y_1l_{1C}+(1-y_1)l_{2C}\right]<0$ at the boundary $y=1$ and $\mu_{y}\langle0\rangle=\rho^{\frac{1}{\gamma_M}}\left[y_1l_{1M}+(1-y_1)l_{2M}\right]>0$ at the boundary $y=0$.
This means the boundaries $y=1$ and $y=0$ are repulsive: when the consumption share $y$ approaches the boundary $y=1$, its volatility decreases and the negative drift pulls it back into the internal region, while when the consumption share $y$ approaches the boundary $y=0$, its volatility decreases and the positive drift pushes it back into the internal region again. Hence, Remark \ref{boundary} indicates that no shareholders become extinct (here extinction means the shareholder holds stock shares in neither firm $1$ nor firm $2$) in the long run in the cross-sectional economy, and this is consistent with the results in \cite{Basak}. However, in the cross-sectional economy, the shareholders may become extinct in one of two firms (i.e., they hold no stock shares in firm $1$ or firm $2$), for example, Region $5$ in Proposition \ref{th3} (if the region is indeed the optimal region in equilibrium). Therefore, we give some sufficient conditions to identify whether the shareholders would become extinct in either firm.

\begin{definition}\label{def_ex}{\it
For the consumption share $y\in(0,1)$ at the time $t$, the extinction of shareholders in the short run is defined as follows:
\begin{itemize}
  \item[(1)] the minority shareholder is said to be extinct, or not survive, in firm $1$ (resp., firm $2$) in equilibrium, if his equilibrium stock holding of firm $1$ (resp., firm $2$) is zero, i.e., $n^*_{1M}=0$ (resp., $n^*_{2M}=0$);
  \item[(2)] the controlling shareholder is said to be extinct, or not survive, in firm $1$ (resp., firm $2$) in equilibrium, if his equilibrium stock holding of firm $1$ (resp., firm $2$) is zero, i.e., $n^*_{1C}=0$ (resp., $n^*_{2C}=0$).
\end{itemize}
}\end{definition}

Obviously, by the equilibrium conditions (\ref{ec-1}) and (\ref{ec-2}), the minority shareholder becomes extinct in firm $1$ (resp., firm $2$) in equilibrium if and only if $n^*_{1C}=\tau$ (resp., $n^*_{2C}=1$).  The controlling shareholder has full control over firm $1$ (resp., firm $2$) if the minority shareholder becomes extinct in firm $2$ (resp., firm $2$), while the minority shareholder holds all stocks of firm $1$ (resp., firm $2$) if the controlling shareholder becomes extinct in firm $2$ (resp., firm $1$).

Making use of results in Propositions \ref{th1}-\ref{th4}, we give some sufficient conditions on extinction of shareholders in either firm in equilibrium.

\begin{proposition}\label{th5}{\it
Assume $y\in(0,1)$ and $p=1$ hold at the time $t$. The controlling shareholder becomes extinct in firm $1$ if one of the following conditions holds:
\begin{description}
  \item [($\mathcal{C}'1$)] there exists a constant $n\in (0,1]$ such that, by substituting $\overline{\mathbf{n}}^*_{C}=(0,n)^\top$ into Proposition \ref{th3}, we have
        \[n=\frac{\theta_{2C}+\alpha_{2C}}{\xi_{2C}},\quad\frac{\theta_{1C}+\alpha_{1C}}{\xi_{0C}}<n;\]
  \item [($\mathcal{C}'2$)] substituting $\overline{\mathbf{n}}^*_{C}=(0,1)^\top$ into Proposition \ref{th3} gives
        \[\theta_{1C}+{\alpha_{1C}}<\xi_{0C},\quad\xi_{2C}<\theta_{2C}+\alpha_{2C};\]
\end{description}
where related parameters and ratios are all given in Proposition \ref{th3}. The controlling shareholder becomes extinct in firm $2$ if one of the following conditions holds:
\begin{description}
  \item [($\mathcal{C}'3$)] there exists a constant $n\in (0,\tau]$ such that, by substituting $\overline{\mathbf{n}}^*_{C}=(n,0)^\top$ into Proposition \ref{th3}, we have
        \[n=\frac{\theta_{1C}+\alpha_{1C}}{\xi_{1C}},\quad\frac{\theta_{2C}+\alpha_{2C}}{\xi_{0C}}<n;\]
  \item [($\mathcal{C}'4$)] substituting $\overline{\mathbf{n}}^*_{C}=(\tau,0)^\top$ into Proposition \ref{th3} gives
        \[\theta_{2C}+{\alpha_{2C}}<\tau\xi_{0C},\quad \tau\xi_{1C}<\theta_{1C}+\alpha_{1C};\]
\end{description}
where related parameters and ratios are all given in Proposition \ref{th3}.
}\end{proposition}

\begin{proposition}\label{th6}{\it
Assume $y\in(0,1)$ and $p=1$ hold at the time $t$. The minority shareholder becomes extinct in firm $1$ if one of the following conditions holds:
\begin{description}
  \item [($\mathcal{M}'1$)] there exists a constant $n\in [0,1)$ such that, by substituting $\overline{\mathbf{n}}^*_{C}=(\tau,n)^\top$ into Proposition \ref{th3}, we have
        \[n=\frac{\theta_{2C}+\alpha_{2C}-\tau\xi_{0C}}{\xi_{2C}},\quad n<\frac{\theta_{1C}+\alpha_{1C}-\tau\xi_{1C}}{\xi_{0C}};\]
  \item [($\mathcal{M}'2$)] substituting $\overline{\mathbf{n}}^*_{C}=(\tau,0)^\top$ into Proposition \ref{th3} gives
        \[\theta_{2C}+{\alpha_{2C}}<\tau\xi_{0C},\quad \tau\xi_{1C}<\theta_{1C}+\alpha_{1C};\]
\end{description}
where related parameters and ratios are all given in Proposition \ref{th3}. The minority shareholder becomes extinct in firm $2$ if one of the following conditions holds:
\begin{description}
  \item [($\mathcal{M}'3$)] there exists a constant $n\in [0,\tau)$ such that, by substituting $\overline{\mathbf{n}}^*_{C}=(n,1)^\top$ into Proposition \ref{th3}, we have
        \[n=\frac{\theta_{1C}+\alpha_{1C}-\xi_{0C}}{\xi_{1C}},\quad n<\frac{\theta_{2C}+\alpha_{2C}-\xi_{2C}}{\xi_{0C}};\]
  \item [($\mathcal{M}'4$)] substituting $\overline{\mathbf{n}}^*_{C}=(0,1)^\top$ into Proposition \ref{th3} gives
        \[\theta_{1C}+{\alpha_{1C}}<\xi_{0C},\quad\xi_{2C}<\theta_{2C}+\alpha_{2C};\]
\end{description}
where related parameters and ratios are all given in Proposition \ref{th3}.
}\end{proposition}

\begin{proposition}\label{th7}{\it
Assume $y\in(0,1)$ and $0\leq p<1$ hold at the time $t$. The controlling shareholder becomes extinct in firm $1$ if one of the following conditions holds:
\begin{description}
  \item [($\mathcal{C}1$)] there exists a constant $n\in\left(0,\frac{1}{1+(1-p)k}\right)$ such that, by substituting ${\mathbf{n}}^*_{C}=(0,n)^\top$ into Proposition \ref{th4}, we have
        \[n=\frac{\theta_{2C}+(2-p)\alpha_{2C}}{\xi_{2C}+(1-p)[2+(1-p)k]\alpha_{2C}},
\quad\frac{\theta_{1C}+\alpha_{1C}}{\xi_{0C}}<n,\quad \det(\xi_C)>\frac{\alpha_{2C}\xi_{1C}}{k};\]
  \item [($\mathcal{C}2$)] substituting ${\mathbf{n}}^*_{C}=(0,1)^\top$ into Proposition \ref{th4} gives
        \[\theta_{1C}+{\alpha_{1C}}<\xi_{0C},\quad\xi_{2C}<\theta_{2C}+\alpha_{2C},\quad \det(\xi_C)>\frac{\alpha_{2C}\xi_{1C}}{k};\]
  \item [($\mathcal{C}3$)] there exists a constant $n\in\left[\frac{1}{1+(1-p)k},1\right]$ such that, by substituting ${\mathbf{n}}^*_{C}=(0,n)^\top$ into Proposition \ref{th4}, we have
        \[n=\frac{\theta_{2C}+(1-\frac{1}{k})\alpha_{2C}}{\xi_{2C}-\frac{\alpha_{2C}}{k}},
\quad\frac{\theta_{1C}+\alpha_{1C}}{\xi_{0C}}<n,\quad \xi_{2C}-
\frac{\alpha_{2C}}{k}\neq 0,\quad\det(\xi_C)>\frac{\alpha_{2C}\xi_{1C}}{k};\]
\end{description}
where related parameters and ratios are all given in Proposition \ref{th4}. The controlling shareholder becomes extinct in firm $2$ if one of the following conditions holds:
\begin{description}
  \item [($\mathcal{C}4$)] there exists a constant $n\in (0,\tau]$ such that, by substituting ${\mathbf{n}}^*_{C}=(n,0)^\top$ into Proposition \ref{th4}, we have
        \[n=\frac{\theta_{1C}+\alpha_{1C}}{\xi_{1C}},\quad\frac{\theta_{2C}+(2-p)\alpha_{2C}}{\xi_{0C}}<n,
\quad\det(\xi_C)>\frac{\alpha_{2C}\xi_{1C}}{k};\]
  \item [($\mathcal{C}5$)] substituting $\overline{\mathbf{n}}^*_{C}=(\tau,0)^\top$ into Proposition \ref{th4} gives
        \[\theta_{2C}+(2-p){\alpha_{2C}}<\tau\xi_{0C},\quad \tau\xi_{1C}<\theta_{1C}+\alpha_{1C},\quad\det(\xi_C)>\frac{\alpha_{2C}\xi_{1C}}{k};\]
\end{description}
where related parameters and ratios are all given in Proposition \ref{th4}.
}\end{proposition}

\begin{proposition}\label{th8}{\it
Assume $y\in(0,1)$ and $0\leq p<1$ hold at the time $t$. The minority shareholder becomes extinct in firm $1$ if one of the following conditions holds:
\begin{description}
  \item [($\mathcal{M}1$)] there exists a constant $n\in\left[0,\frac{1}{1+(1-p)k}\right)$ such that, by substituting ${\mathbf{n}}^*_{C}=(\tau,n)^\top$ into Proposition \ref{th4}, we have
        \[n=\frac{\theta_{2C}+(2-p)\alpha_{2C}-\tau\xi_{0C}}{\xi_{2C}+(1-p)[2+(1-p)k]\alpha_{2C}},
\quad n<\frac{\theta_{1C}+\alpha_{1C}-\tau\xi_{1C}}{\xi_{0C}},\quad \det(\xi_C)>\frac{\alpha_{2C}\xi_{1C}}{k};\]
  \item [($\mathcal{M}2$)] substituting ${\mathbf{n}}^*_{C}=(\tau,0)^\top$ into Proposition \ref{th4} gives
        \[\theta_{2C}+(2-p){\alpha_{2C}}<\tau\xi_{0C},\quad \tau\xi_{1C}<\theta_{1C}+\alpha_{1C};\]
  \item [($\mathcal{M}3$)] there exists a constant $n\in\left[\frac{1}{1+(1-p)k},1\right)$ such that, by substituting ${\mathbf{n}}^*_{C}=(\tau,n)^\top$ into Proposition \ref{th4}, we have
        \[n=\frac{\theta_{2C}+(1-\frac{1}{k})\alpha_{2C}-\tau\xi_{0C}}
{\xi_{2C}-\frac{\alpha_{2C}}{k}},
\quad n<\frac{\theta_{1C}+\alpha_{1C}-\tau\xi_{1C}}{\xi_{0C}},\quad \xi_{2C}-
\frac{\alpha_{2C}}{k}\neq 0,\quad\det(\xi_C)>\frac{\alpha_{2C}\xi_{1C}}{k};\]
\end{description}
where related parameters and ratios are all given in Proposition \ref{th4}. The minority shareholder becomes extinct in firm $2$ if one of the following conditions holds:
\begin{description}
  \item [($\mathcal{M}4$)] there exists a constant $n\in [0,\tau)$ such that, by substituting ${\mathbf{n}}^*_{C}=(n,1)^\top$ into Proposition \ref{th4}, we have
        \[n=\frac{\theta_{1C}+\alpha_{1C}-\xi_{0C}}{\xi_{1C}},\quad n<\frac{\theta_{2C}+\alpha_{2C}-\xi_{2C}}{\xi_{0C}},\quad \det(\xi_C)>\frac{\alpha_{2C}\xi_{1C}}{k};\]
  \item [($\mathcal{M}5$)] substituting ${\mathbf{n}}^*_{C}=(0,1)^\top$ into Proposition \ref{th4} gives
        \[\theta_{1C}+{\alpha_{1C}}<\xi_{0C},\quad\xi_{2C}<\theta_{2C}+\alpha_{2C},\quad \det(\xi_C)>\frac{\alpha_{2C}\xi_{1C}}{k};\]
\end{description}
where related parameters and ratios are all given in Proposition \ref{th4}.
}\end{proposition}

According to proofs in Appendices, with perfect protection, extinction conditions of the controlling shareholder, i.e., $(\mathcal{C}'1)-(\mathcal{C}'4)$, imply his optimal stock holdings are located in Regions 2, 7, 4 and 8 respectively, and extinction conditions of the minority shareholder, i.e., $(\mathcal{M}'1)-(\mathcal{M}'4)$, imply his optimal stock holdings are located in Regions 3, 8, 5 and 7 respectively. Similarly, from proofs in Appendices, with imperfect protection, extinction conditions of the controlling shareholder, i.e., $(\mathcal{C}1)-(\mathcal{C}5)$, imply his optimal stock holdings are located in Regions 2, 7, 11, 4 and 8 respectively, and extinction conditions of the minority shareholder, i.e., $(\mathcal{M}1)-(\mathcal{M}5)$, imply his optimal stock holdings are located in Regions 3, 8, 12, 5 and 7 respectively. Therefore, the reasons why shareholders become extinct in firm 1 or firm 2 have been shown in Subsection \ref{subsection3.1}.

\section{Numerical results}\label{section4}\noindent
\setcounter{equation}{0}
In this section, we analyze, by numerical results, how investor protection and cross-section setting in the economy affect the equilibrium. For the sake of argument, we choose the fundamental parameters used in \cite{Basak} (see Table \ref{table}).
\begin{table}[htbp]
\centering
\caption{Fundamental parameters in numerical analysis}
\begin{tabular}{cccccccccccc}%l=left, r=right,c=center
\hline
Parameters &$\mu_{1D}$&$\mu_{2D}$ & $\sigma_D$ & $\gamma_C$ & $\gamma_M$ & $\rho$ & $k$
& $l_{1C}$ & $l_{2C}$ & $l_{1M}$ & $l_{2M}$  \\ \hline
Values     &$1.5\%$   &$1.5\%$    & $13\%$     & $3$        & $3.5$      & $0.01$ & $6$
& $0.1$    & $0.1$    & $0.5$    &$0.5$   \\ \hline
\end{tabular}
\label{table}
\end{table}
Moreover, we set $\tau=1,y_1=0.5,\delta_D=10\%$ and vary $y$ from $0$ to $1$ with the step length $0.01$, and other parameters would be given in associated figures. For the sake of simplicity, in the case of $p=1$, we would use, in all following figures, the same notations $n^*_{1C}, n^*_{2C}, x^*, \mu, \sigma$ and $\delta$ to denote $\overline{n}^*_{1C}, \overline{n}^*_{2C}, \overline{x}^*, \overline{\mu}, \overline{\sigma}$ and $\overline{\delta}$, respectively.

\subsection{Stock holdings and diverted output}
We start our numerical analysis in equilibrium with Figure \ref{fig_sx} which demonstrates the effect of investor protection on the controlling shareholder's equilibrium stock holdings and fractions of diverted output in the cross-sectional economy.  To be precise, we explain Figure \ref{fig_sx} as follows: stock holdings $n^*_{1C}$ in panels (a) and (b), stock holdings $n^*_{2C}$ in panels (c) and (d) and fractions of diverted output $x^*$ in panel (e) are obtained by substituting each $y$ and $\tau=1,y_1=0.5,y_2=0.5,\delta_D=10\%$, and parameters in Table \ref{table} into Proposition \ref{th3} in the case of $p=1$ and Proposition \ref{th4} in the cases of $p=0.9$ and $p=0.6$, respectively.

\begin{figure}[htb]
\centering
\subfigure[Stock holdings in firm $1$]{
\includegraphics[ width = 0.31\textwidth]{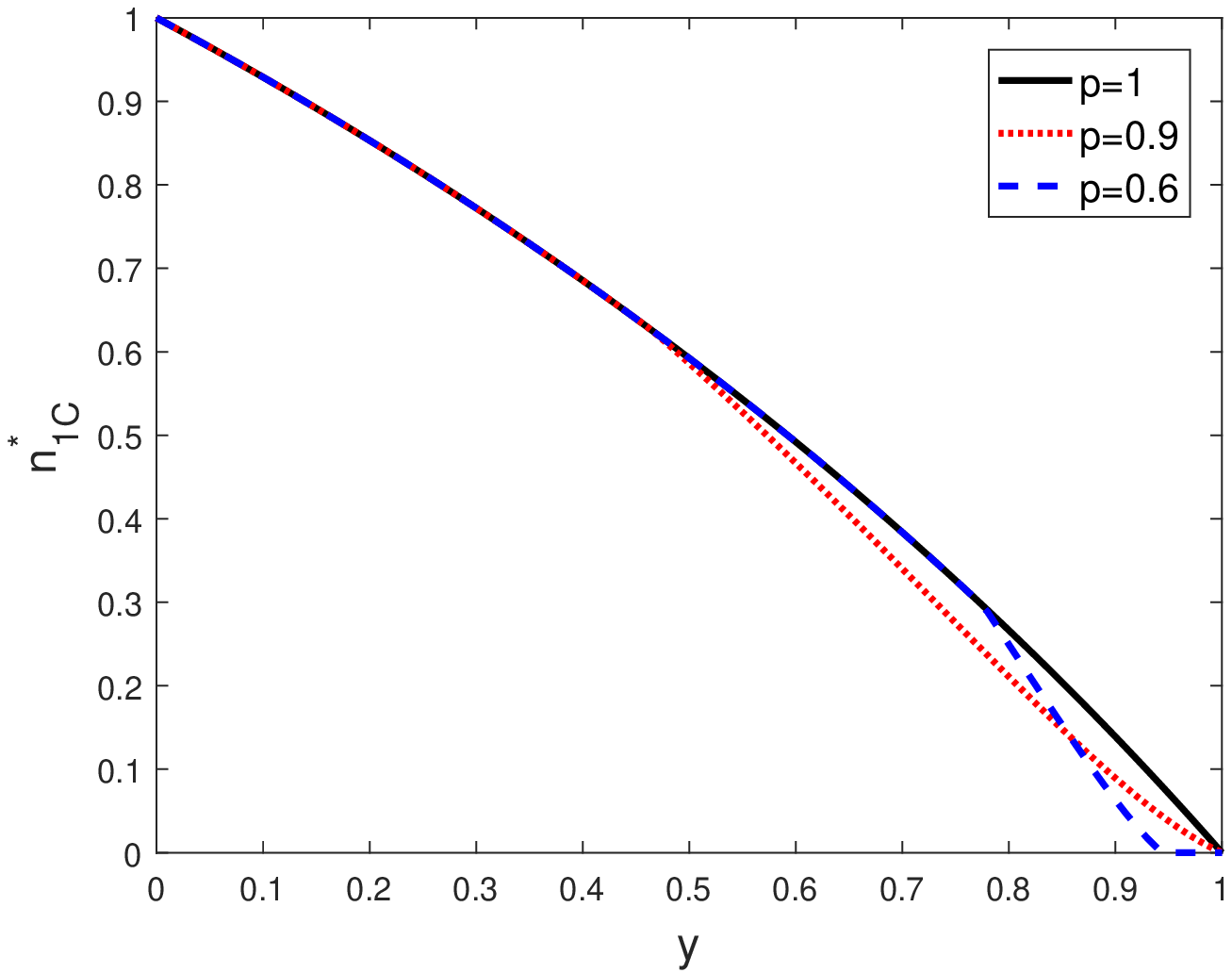}
}
\subfigure[Share difference in firm $1$]{
\includegraphics[ width = 0.31\textwidth]{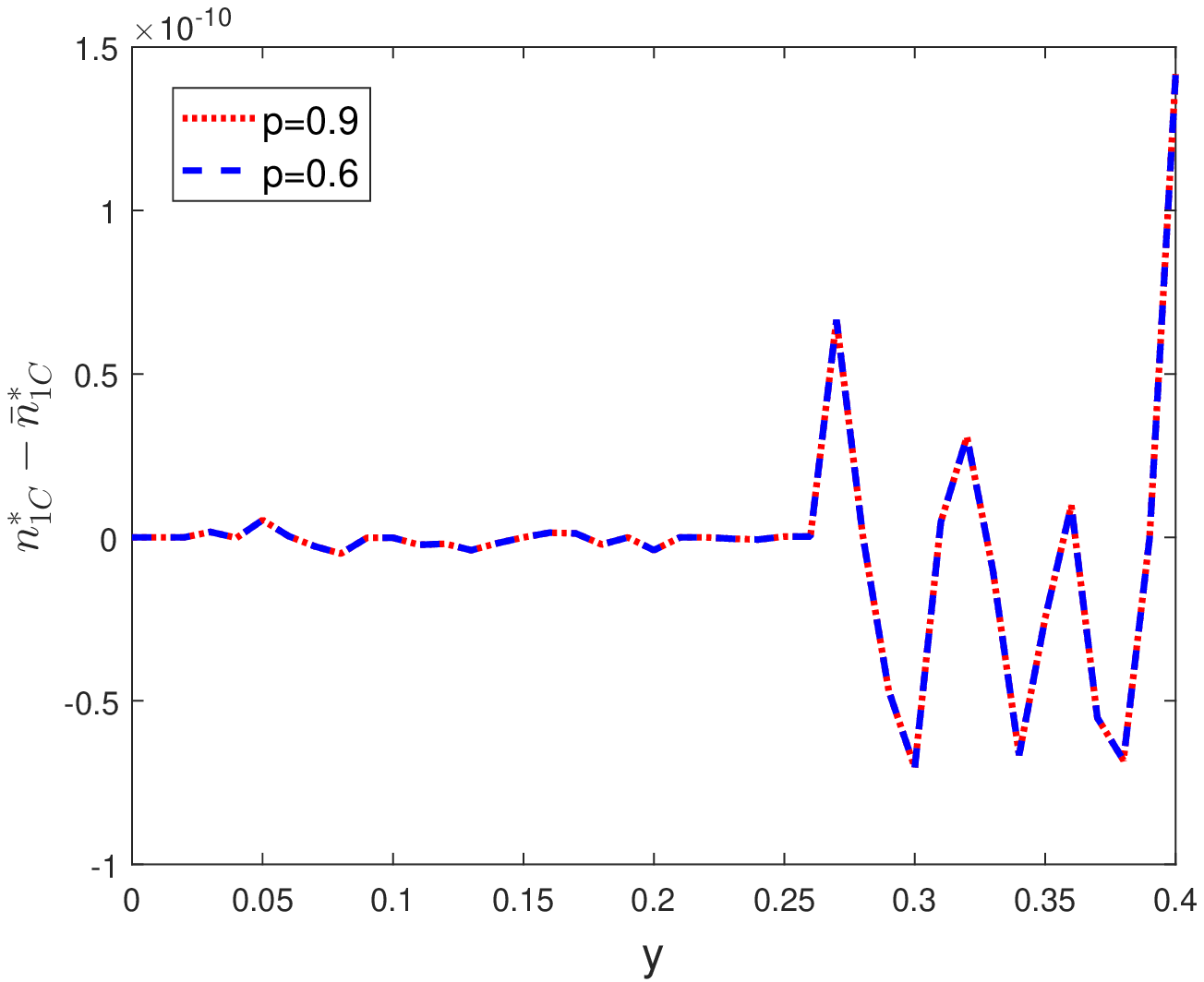}
}\\
\subfigure[Stock holdings in firm $2$]{
\includegraphics[ width = 0.31\textwidth]{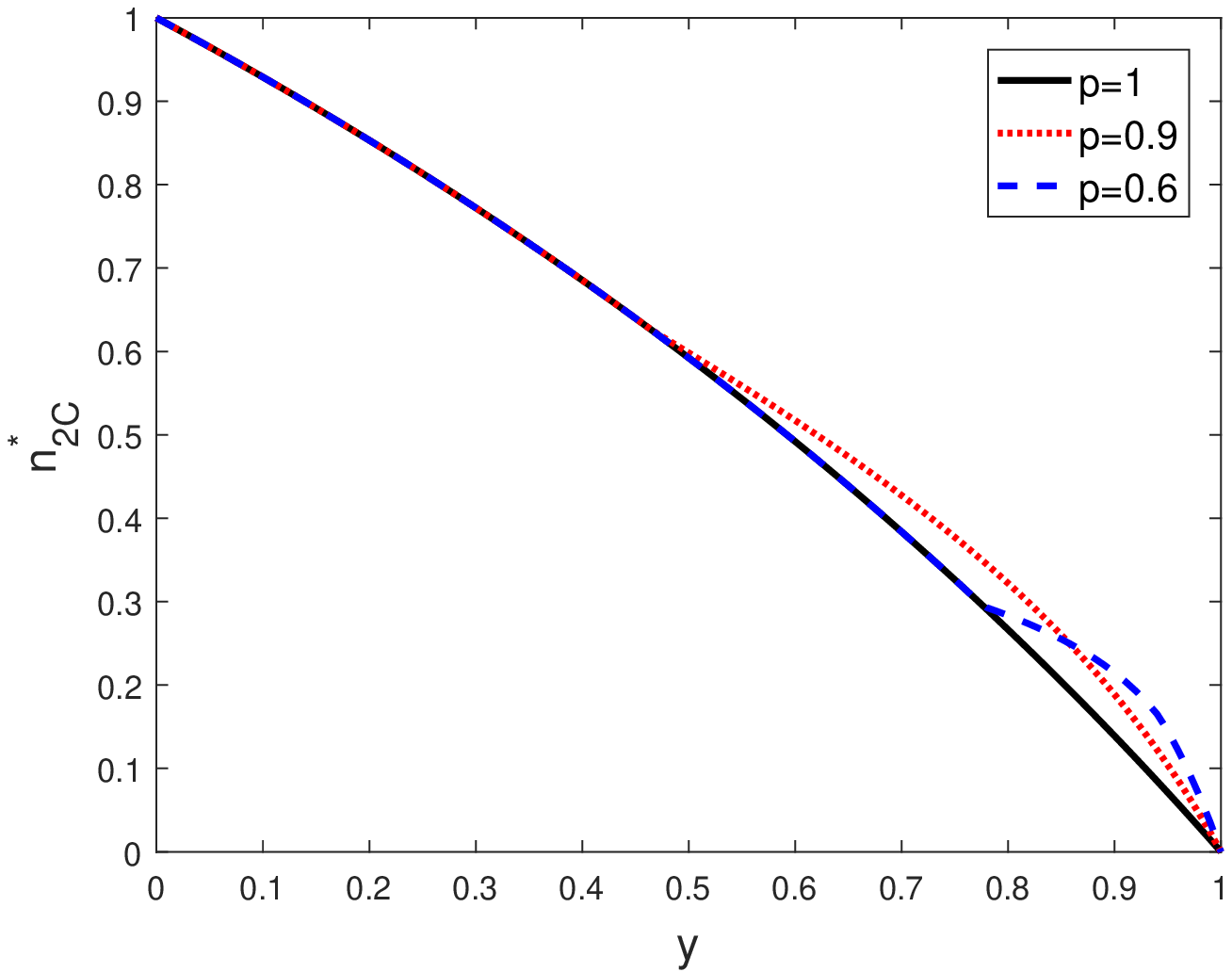}
}
\subfigure[Share difference in firm $2$]{
\includegraphics[ width = 0.31\textwidth]{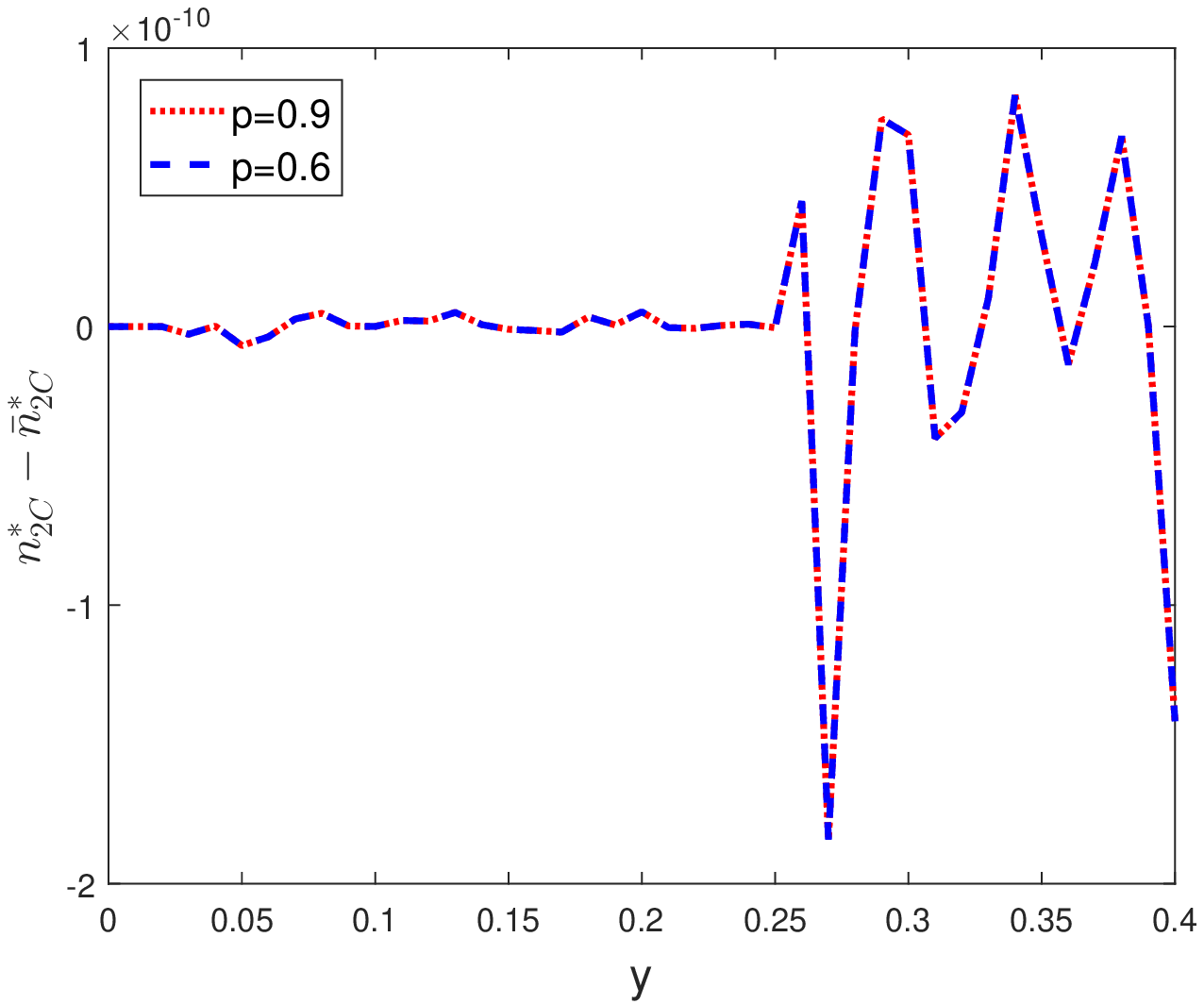}
}
\subfigure[Fractions of of diverted output]{
\includegraphics[ width = 0.31\textwidth]{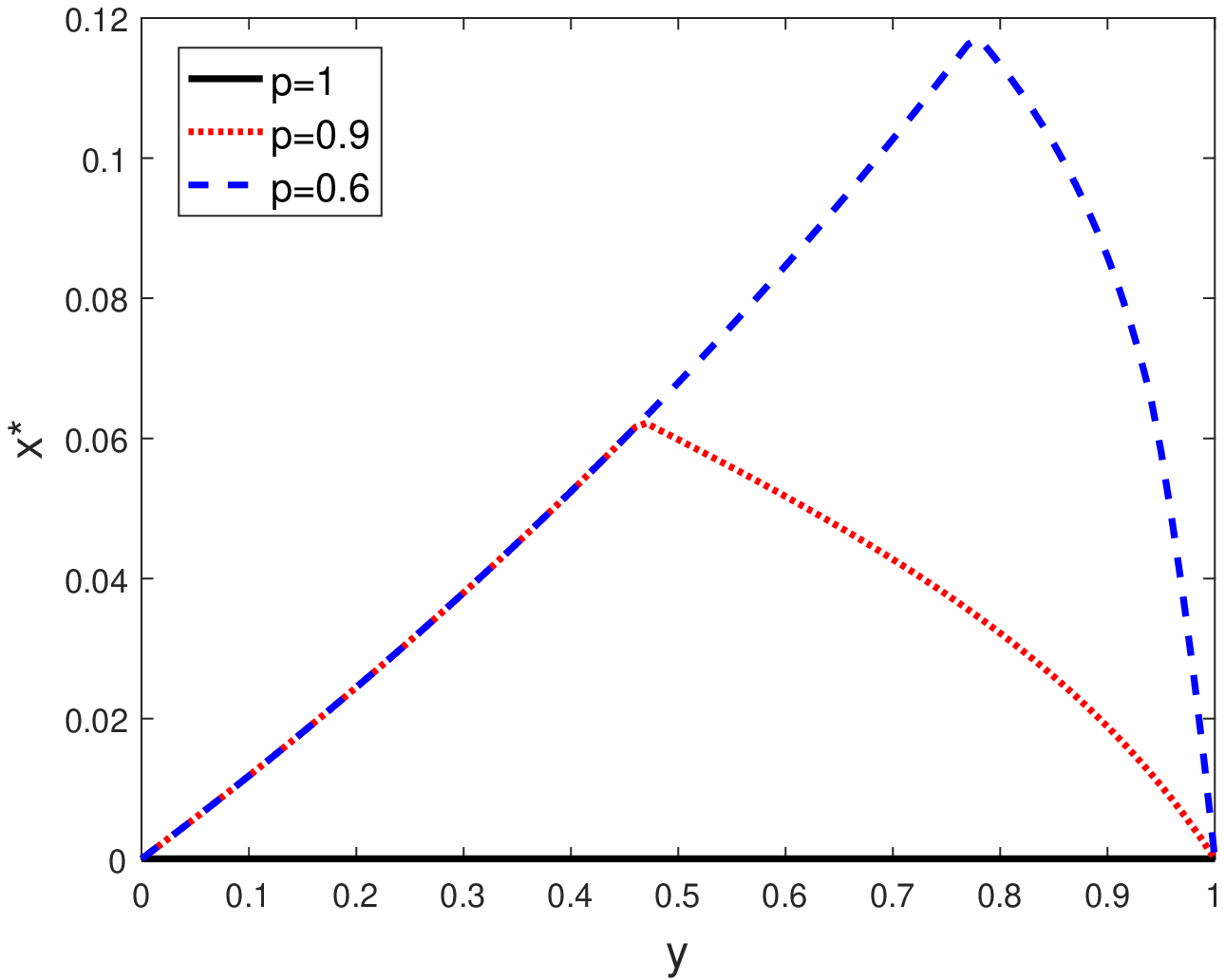}
}
\caption{ The effect of investor protection in the cross-sectional economy on the controlling shareholder's stock holdings and fractions of diverted output}
\label{fig_sx}
\end{figure}

First of all, compared with the case of perfect investor protection ($p=1$), there are kinks in the case of imperfect investor protection ($p=0.9$ and $p=0.6$), especially the controlling shareholder's stock holdings of firm 1 in panel (a) of Figure \ref{fig_sx}, and these kinks imply the change of regions for shareholders' optimal strategies in equilibrium. Actually, along with the decrease of the controlling shareholder's consumption share (i.e., increase of $y$), shareholders' optimal strategies (parameters in Figure \ref{fig_sx}, as well as in Figure \ref{fig_mvr}) in equilibrium are located in different regions for $y=0.01i,\;i\in\{1,2,\cdots,99\}$: in the case of $p=1$, related parameters are all located in Region $1$; in the case of $p=0.9$, related parameters with $1\leq i\leq 46$ are located in Region $10$ and others lie in Region $1$; in the case of $p=0.6$, related parameters with $1\leq i\leq 77$, $78\leq i\leq 94$ and $95\leq i\leq 99$ are respectively located in Regions $10$, $1$ and $2$.

Panels (a) and (c) of Figure \ref{fig_sx} show that poorer protection in firm $2$ tends to decrease the controlling shareholder's stock holding $n^*_{1C}$ in firm $1$ and increase his stock holding $n^*_{2C}$ in firm $2$ relative to the perfect protection in firm $2$, i.e., $n^*_{1C}\leq \overline{n}_{1C}^*,\;n^*_{2C}\geq \overline{n}_{2C}^*$. Such financial phenomenon in the cross-sectional economy is similar to that of home and foreign investors in Proposition 4 of \cite{Giannetti}, and the result in firm 2 is consistent with those in \cite{Basak,Shleifer} indicating the positive relation between poorer investor protection and higher ownership concentration. The reason is twofold. With imperfect investor protection in firm $2$, the controlling shareholder can divert a higher fraction of output in firm $2$ when he owns more shares of firm 2, and this drives him to acquire more shares of firm $2$ and less shares of firm $1$ in equilibrium when his consumption share is low (i.e., $y$ is high).  On the other hand, diverting output in firm $2$ naturally withholds the minority shareholder's stock holding in firm $2$ and drives him to acquire more shares of firm $1$ in equilibrium. Such phenomenon that the controlling shareholder prefers acquiring more shares of firm with imperfect investor protection while the minority shareholder prefers acquiring more shares of firm with perfect investor protection implies that the investor protection constraint in the cross-sectional economy indeed provides the minority shareholder with investor protection. However, when the controlling shareholder's consumption share is high (i.e., $y$ is low), his stock holdings are almost the same as in the benchmark economy, precisely, $n^*_{jC}\approx \overline{n}_{jC}^*$ for $j=1,2$. The investor protection constraint does not bind because owning lots of shares reduces the controlling shareholder's benefits of diverting the output, and hence the controlling shareholder would hold shares based on the benchmark level. This is similar to the economy without the cross-section setting in \cite{Basak}.

Panel (e) of Figure \ref{fig_sx} reveals that better investor protection in the cross-sectional economy considerably reduces the fraction of diverted output in firm 2. This result is what we expect and consistent with those in \cite{Basak,Shleifer}. Since the fraction of diverted output is directly related to the controlling shareholder's stock holding in firm 2, we could discuss it in a similar method used in panel (c) of Figure \ref{fig_sx} which is base on whether the investor protection constraint $x\leq (1-p)n_{2C}$ is binding. When the controlling shareholder's consumption share is low (i.e., $y$ is high), the investor protection constraint is binding and this implies poorer investor protection leads to higher fraction of diverted output. When the controlling shareholder's consumption share is high (i.e., $y$ is low), the investor protection constraint is not binding and the fraction of diverted output is tempered by the cost of stealing $k$ so that the controlling shareholder diverts the same fractions of output in the cross-sectional economies with different levels of imperfect protection.

Furthermore, compared with the case in the economy without cross-section setting, the cross-section in equilibrium can also cause essential difference in the effect of investor protection. For one thing, because of perfect investor protection in firm $1$ and extremely poor protection in firm $2$, the controlling shareholder in Region $2$ would hold no stocks in firm $1$ and acquire more shares of firm $2$ to divert a higher fraction of output in firm $2$. And this means the controlling shareholder becomes extinct in firm 1 (By examining, here condition ($\mathcal{C}1$) indeed holds) as the result of investor protection and cross-section, which never happens in the economy without the cross-section setting (see \cite{Basak}).
For another, contrast with \cite{Basak}, even when the controlling shareholder's consumption share is high (i.e., $y$ is low), the stock holdings, $n_{1C}^*$ and $n_{2C}^*$, are not the same as in the economy with perfect protection (see panels (b) and (d) of Figure \ref{fig_sx}). This is because along with his consumption share, the cross-section setting leads to different returns and volatilities of two stocks, and then this enables the controlling shareholder to adjust his shares between firms $1$ and $2$ so that his stock holdings are not the same as in the economy with perfect protection.

\subsection{Stock return, stock volatility and interest rate}
Figure \ref{fig_mvr} depicts the effect of investor protection in the cross-sectional economy on stock gross returns of firms $1$ and $2$, fundamental stock volatilities, nonfundamental stock volatilities and interest rates. More precisely, we explain Figure \ref{fig_mvr} as follows: stock gross returns in firms $1$ and $2$ are defined by $\mu_1+\frac{D_1}{\tau S_1}$ and $\mu_2+(1-x^*)\frac{D_2}{S_2}$ where $\mu_1$ and $\mu_2$ are the capital gains and $\frac{D_1}{\tau S_1}$ and $(1-x^*)\frac{D_2}{S_2}$ are dividend yields in firms $1$ and $2$ respectively; and gross stock returns of firm $1$ in panel (a), gross stock returns of firm $2$ in panel (b), fundamental stock volatilities $\sigma$ in panel (c), nonfundamental stock volatilities $\delta$ in panel (d) and interest rates $r$ in panel (e) are obtained by substituting each $y$ and $\tau=1,y_1=0.5,y_2=0.5,\delta_D=10\%$, and parameters in Table \ref{table} into Proposition \ref{th3} in the case of $p=1$ and Proposition \ref{th4} in the cases of $p=0.9$ and $p=0.6$, respectively.

\begin{figure}[htb]
\centering
\subfigure[Stock gross returns in firm $1$]{
\includegraphics[ width = 0.31\textwidth]{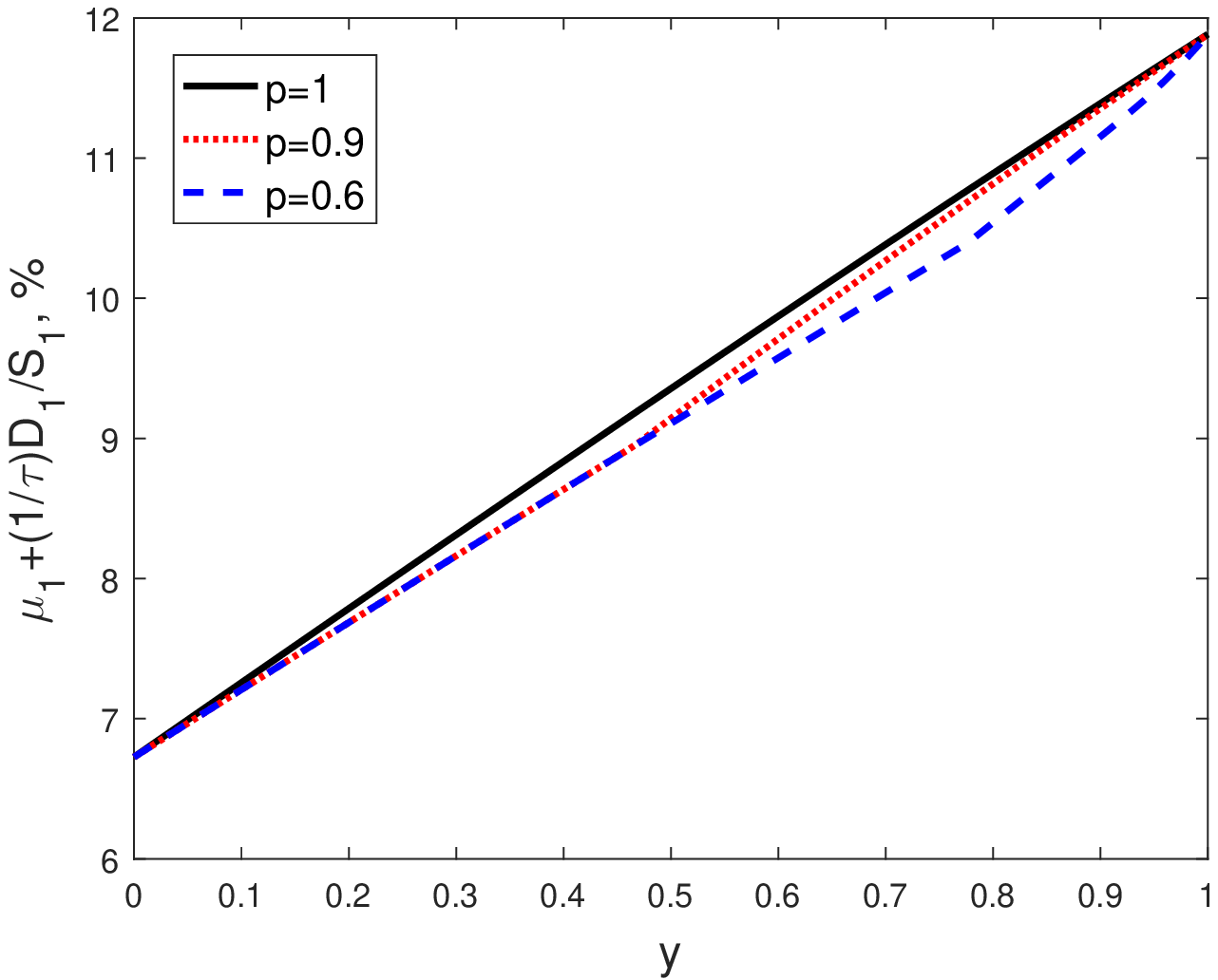}
}
\subfigure[Stock gross returns in firm $2$]{
\includegraphics[ width = 0.31\textwidth]{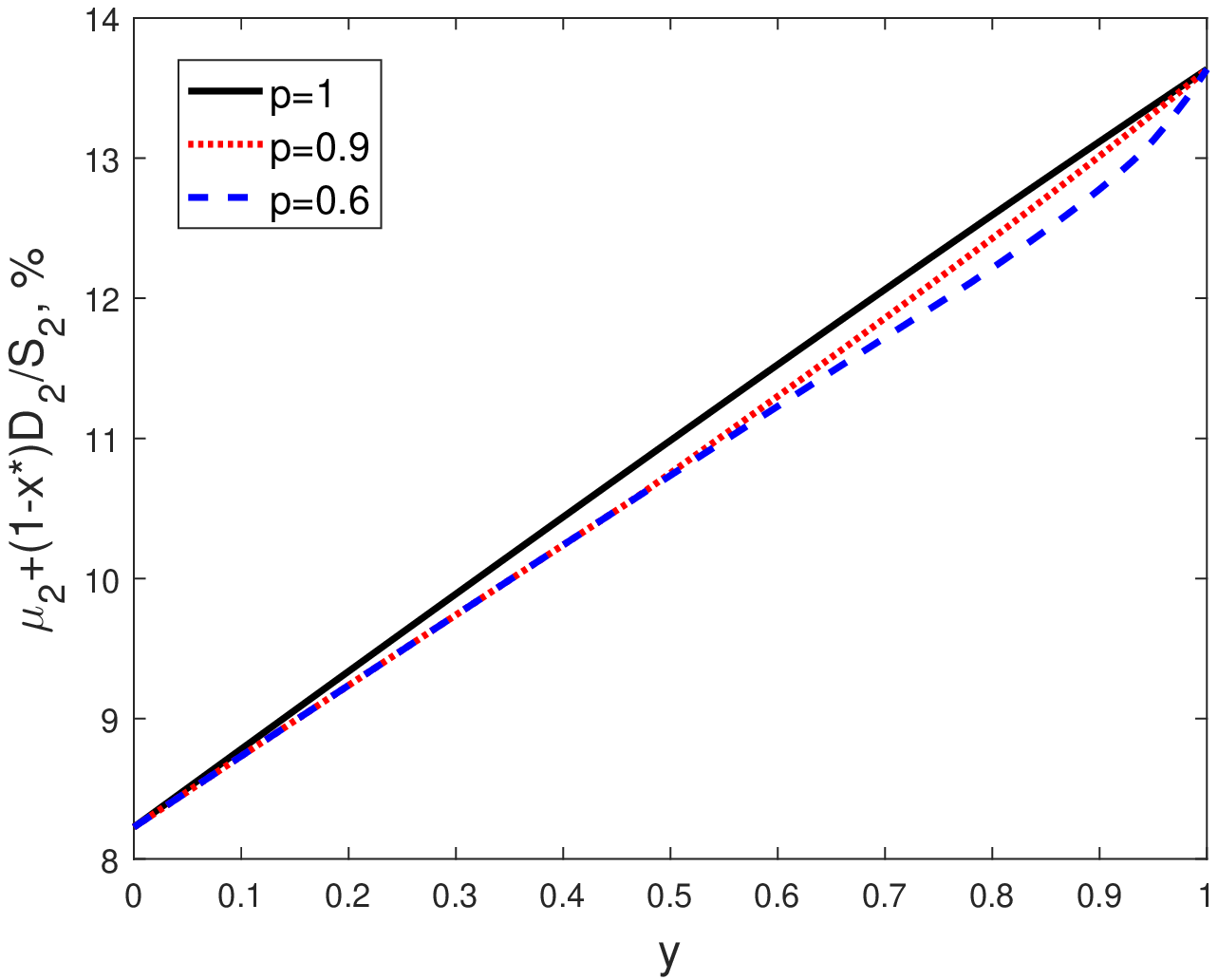}
}\\
\subfigure[Fundamental stock volatilities]{
\includegraphics[ width = 0.31\textwidth]{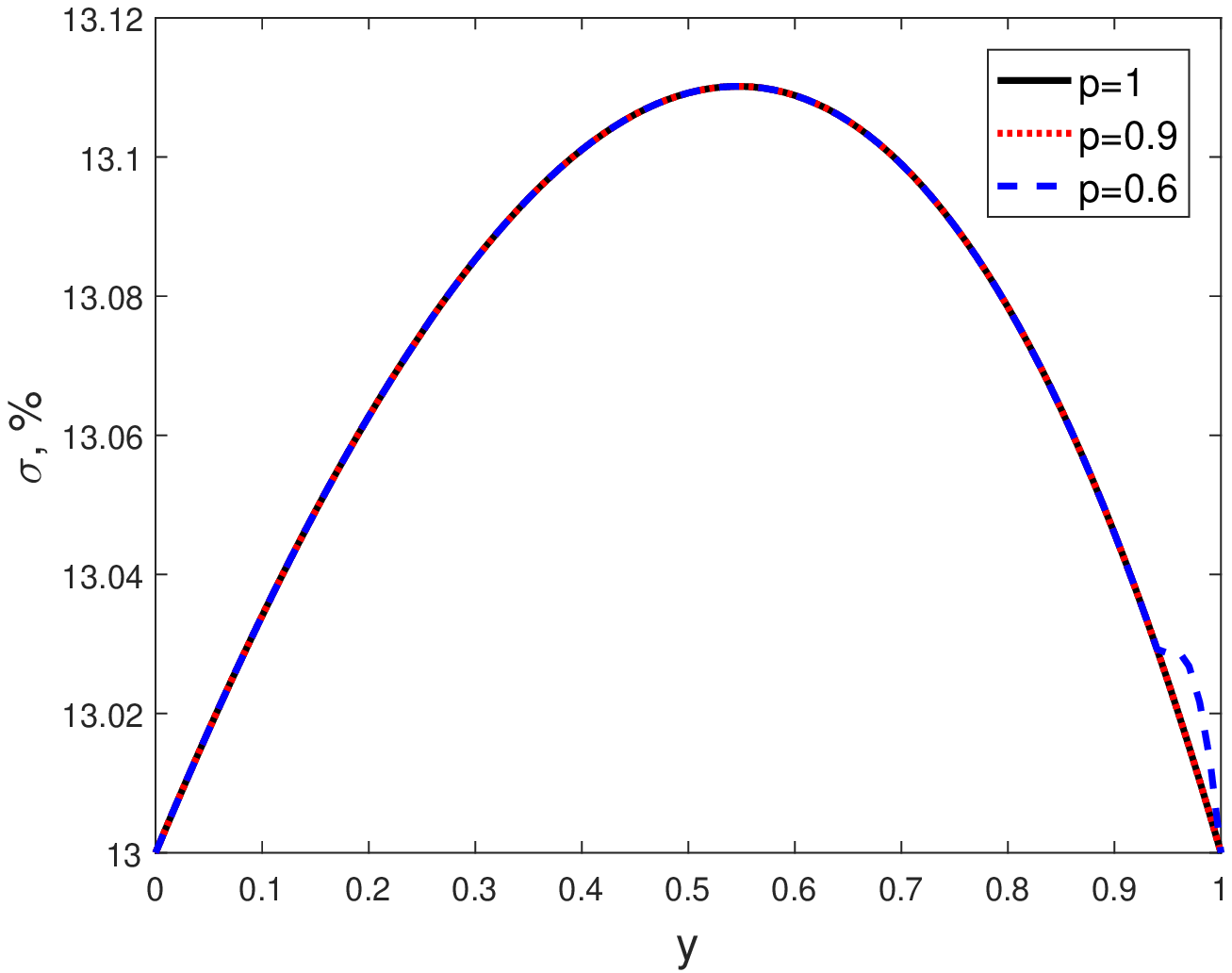}
}
\subfigure[Nonfundamental stock volatilities]{
\includegraphics[ width = 0.31\textwidth]{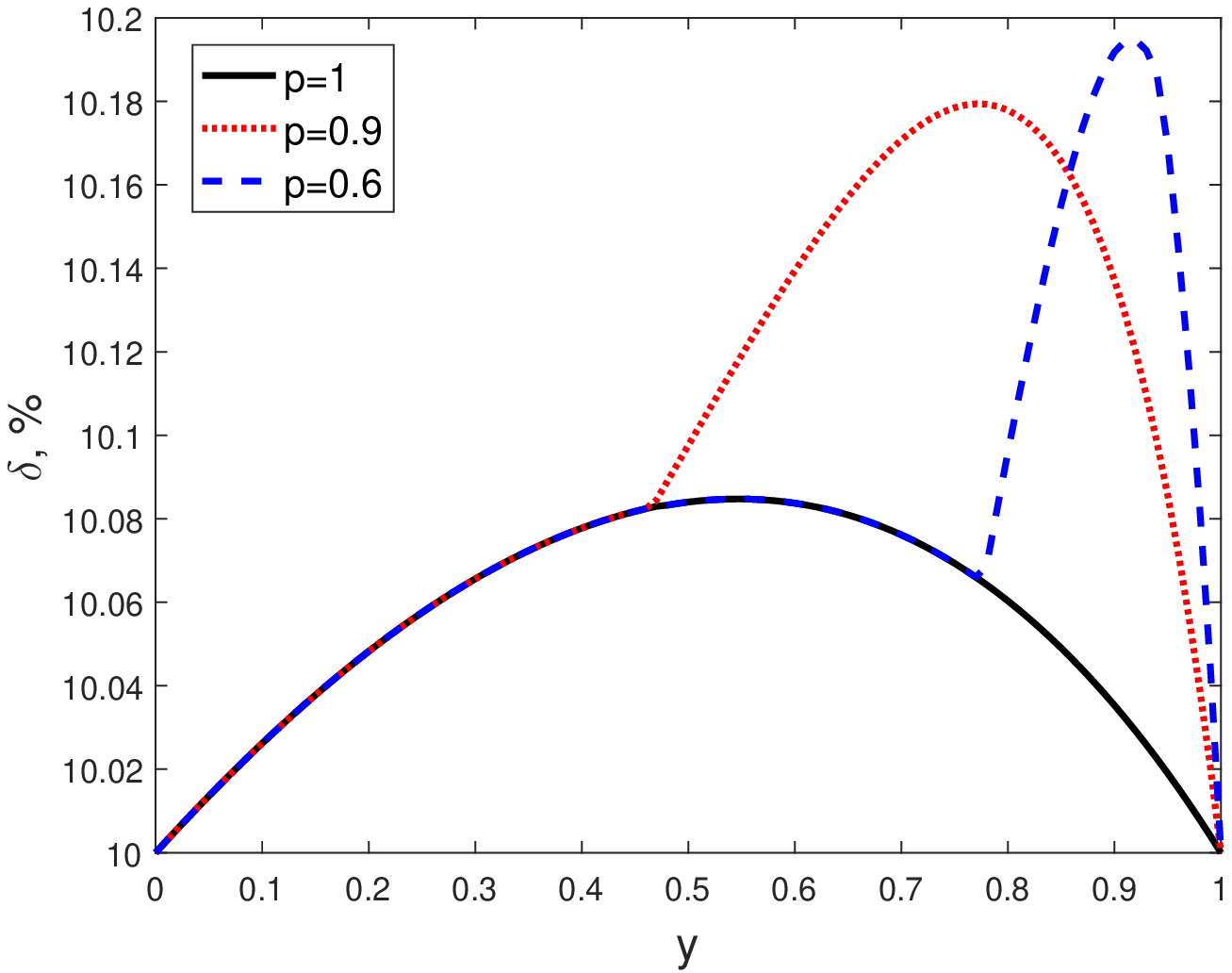}
}
\subfigure[Interest rates]{
\includegraphics[ width = 0.31\textwidth]{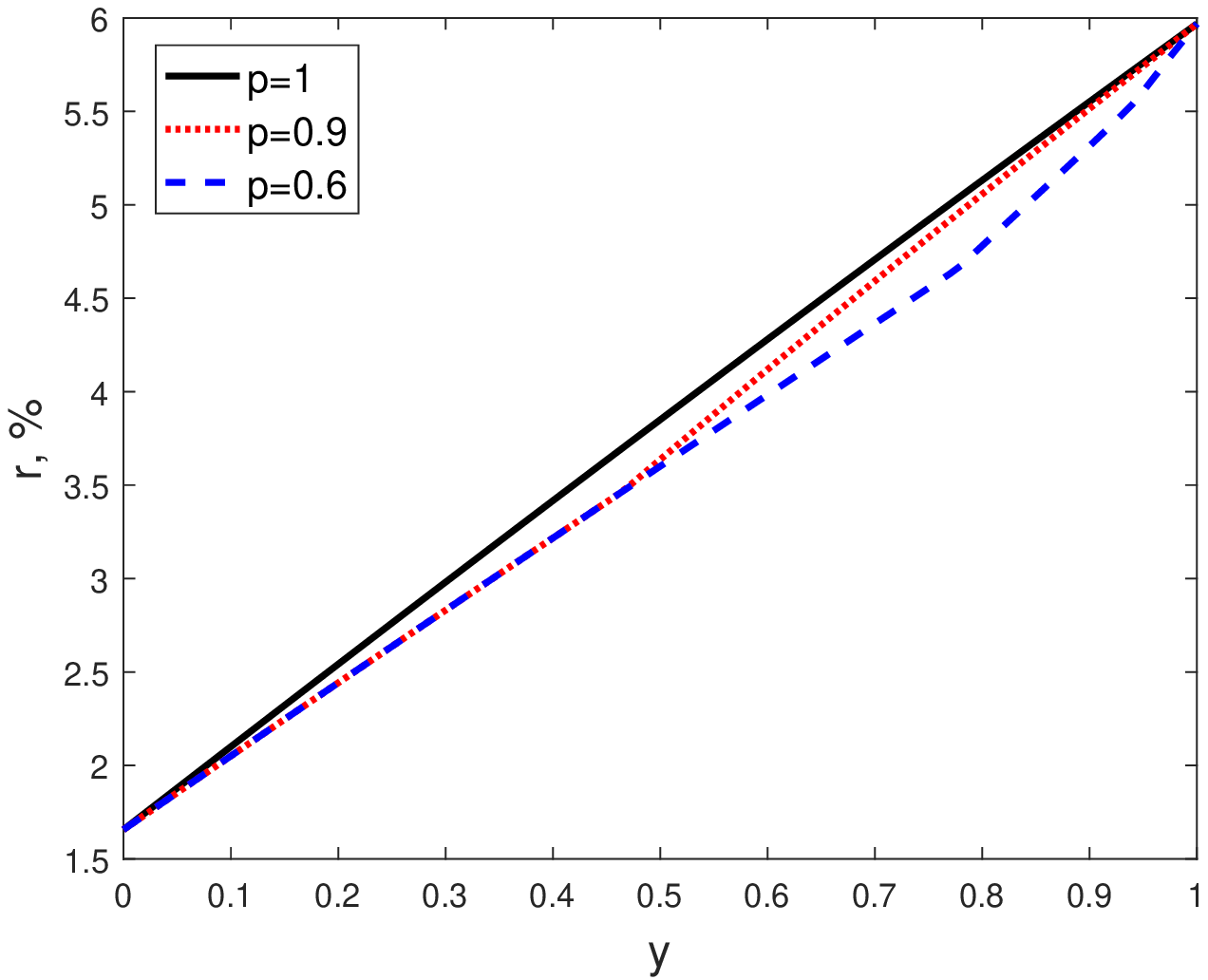}
}
\caption{The effect of investor protection in the cross-sectional economy on stock returns and volatilities, and interest rates }
\label{fig_mvr}
\end{figure}

Panels (a) and (b) of Figure \ref{fig_mvr} demonstrate that better investor protection in firm $2$ leads to higher gross stock returns in both firms $1$ and $2$. This is the same as results in \cite{Basak,Giannetti} documenting the positive relation between the realized stock returns and investor protection, and it is also consistent with the empirical evidence (for example, \cite{Gompers}) on the relation between the realized stock returns and corporate governance. The reason for firm $2$ can be explained similarly as \cite{Basak}. Perfect investor protection makes gross stock returns in firm $2$ be determined by shareholders' risk aversions, and such gross stock returns are sufficiently high to compensate shareholders for risk taking. However, poor investor protection enables the controlling shareholder to divert the output in firm $2$ so that the controlling shareholder can be compensated for excessive risk taking through additional stealing and lower risk premia. Based on discussion of firm $2$, gross stock returns in firm $1$ can be explained. With poorer investor protection in firm $2$, the controlling shareholder would hold less shares of firm $1$ and then take lower excessive risk (including stock volatilities of firm $1$ coming from fundamental risk and additional stock covariance caused by cross-section). Hence, lower risk premia and then lower gross stock returns in firm $1$ are enough to compensate the controlling shareholder for excessive risk taking relative to firm $1$.

Panels (c) and (d) show that, in the economy with imperfect investor protection, both fundamental and nonfundamental stock volatility are higher than in the benchmark economy with perfect protection, i.e., $\sigma\geq\overline{\sigma},\;\delta\geq\overline{\delta}$. And this clearly implies in the economy with imperfect investor protection in firm $2$, stock volatilities in firms $1$ and $2$ are are higher than in the benchmark economy with perfect protection. The result is consistent with that in \cite{Basak}, but the reasons for firms $1$ and $2$ are different. On the one hand, the main reason for higher stock volatility in firm $2$ is that, with imperfect investor protection, the controlling shareholder holds more shares of firm $2$ than in the benchmark economy with perfect protection such that the shares of firm $2$ are under-diversified. Hence, the controlling shareholder's volatile wealth and consumption can translate into the stock volatility of firm $2$ via the market clearing conditions. On the other hand, investor protection actually has few effects on the stock volatility of firm $1$ except for some extreme conditions (see, Region $2$ of the case $p=0.6$ in panel (c)). As is analyzed in panel (a) of Figure \ref{fig_sx}, the shareholders' stock holdings in firm $1$ are almost the same as in the benchmark economy when $y$ is low. Hence, similar diversity of stock holdings leads to almost the same stock volatility. when $y$ is high, higher covariance between two firms (caused by higher volatility of firm $2$ through cross-section) neutralizes higher diversity of stock holdings in firm $1$, leading to similar stock volatility in firm $1$. However, in Region $2$ of the case $p=0.6$ in panel (c), the controlling shareholder in firm $1$ is extinct and then all risk of stock covariance is taken by the minority shareholder, which leads to higher stock volatility of firm $1$. Furthermore, better investor protection can be regarded as tighter constraints on the controlling shareholder's portfolios, and such tighter constraints would reduce stock volatilities (see \cite{Chabakauri,Hardouvelis}), which is in agreement with our results.

Panel (e) reveals that better investor protection leads to higher interest rate in the cross-sectional economy. This is consistent with the results of Figure 8 in \cite{Basak} and Proposition 7 in \cite{Shleifer}. With poor investor protection in firm $2$, not only does the minority shareholder face lower risk premia (with no diverted output for him) in firm $2$, but he also endures lower risk premia in firm $1$. Then for the minority shareholder, investing stocks in firms $1$ and $2$ is less attractive and turning to investing bonds is a better choice. Therefore, the minority shareholder is willing to invest more in bonds and provide cheaper credit, which results in lower interest rate in the bond market.

\subsection{The effect of cross-section: comparison to the economy with the single relative firm}

Compared with the economy with the single relative firm (i.e., the firm has the same economy parameters: the output return and volatilities, the fraction of the output paid to the shareholders as labor incomes, the time-preference parameter and pecuniary cost from diverting output), Figure \ref{fig_compare1} (resp., Figure \ref{fig_compare2}) demonstrates the effect of cross-section on the controlling shareholder's  optimal strategies and asset prices in firm $1$ (resp., firm $2$). Figures \ref{fig_compare1} and \ref{fig_compare2} are further explained as follows:
in Figure \ref{fig_compare1}, parameters $n_{1C}^*,\mu_1,\sigma$ and $r$ are obtained in the same way as Figures \ref{fig_sx} and \ref{fig_mvr}, and parameters $n_{1C}^{s*},\mu_1^s,\sigma^s_1$ and $r^s_1$ are obtained by substituting each $y$ and parameters $\mu_D=0.015,\sigma_D=0.13,\gamma_C=3,\gamma_M=3.5,\rho=0.01,k=6,l_{C}=0.1,l_{M}=0.5$ into Proposition 2 (the perfect protection case) in \cite{Basak};
in Figure \ref{fig_compare2}, stock volatilities $\sigma_2=\sqrt{\sigma^2+\delta^2}$ and parameters $n_{2C}^*,x^*,\mu_2,\sigma,\delta$ and $r$ are obtained in the same way as Figures \ref{fig_sx} and \ref{fig_mvr}, and parameters $n_{2C}^{s*},x^{s*},\mu_2^s,\sigma^s_2$ and $r^s_2$ are obtained by substituting  each $y$ and parameters $\mu_D=0.015,\sigma_D=\sqrt{0.13^2+0.1^2},\gamma_C=3,\gamma_M=3.5,\rho=0.01,k=6,l_{C}=0.1,l_{M}=0.5$ into Proposition 2 of \cite{Basak} in the case of $p=1$ and  Proposition 3 of \cite{Basak} in the cases of $p=0.9$ and $p=0.6$, respectively.

\begin{figure}[htb]
\centering
\subfigure[Stock holdings]{
\includegraphics[ width = 0.31\textwidth]{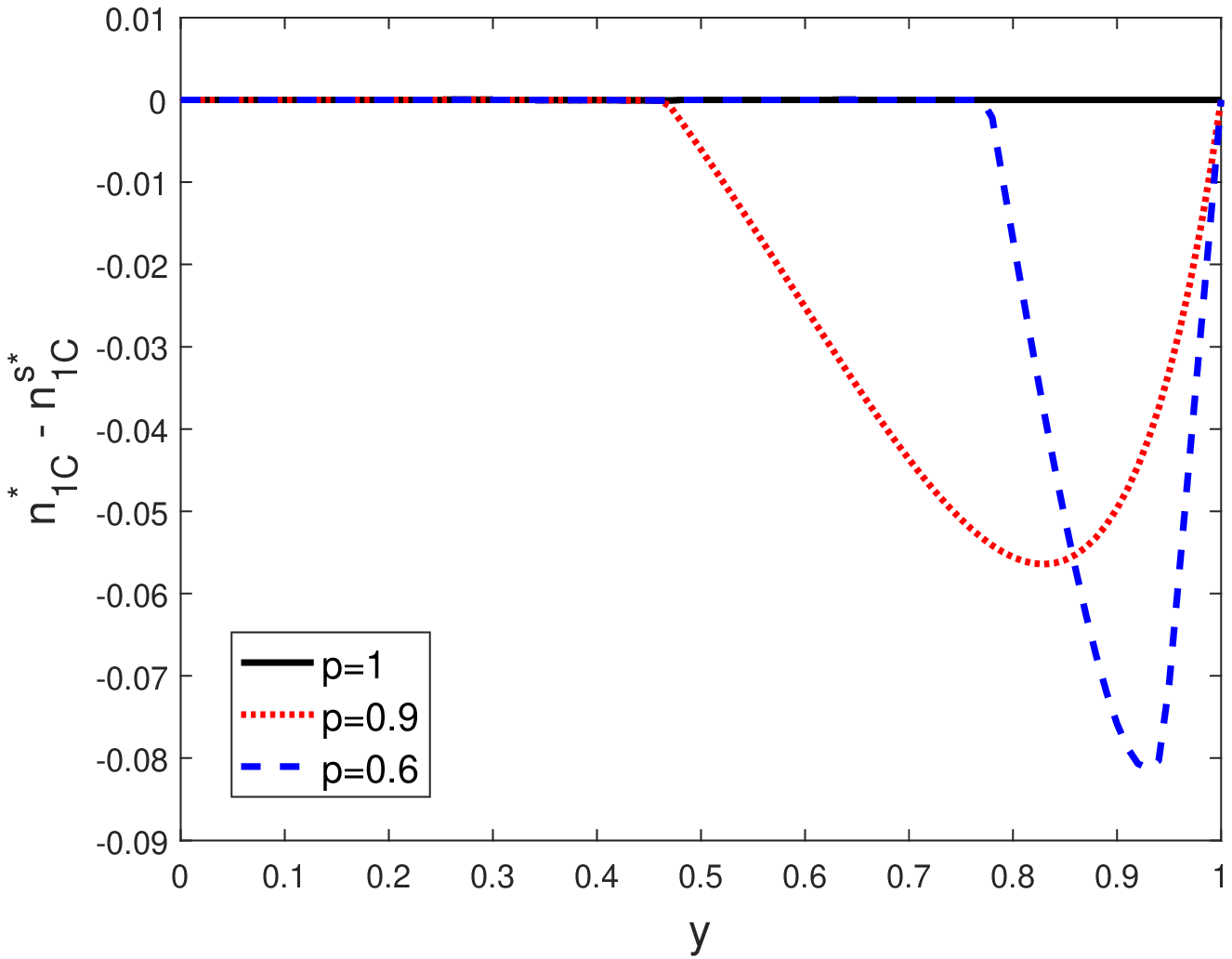}
}
\subfigure[Stock returns]{
\includegraphics[ width = 0.31\textwidth]{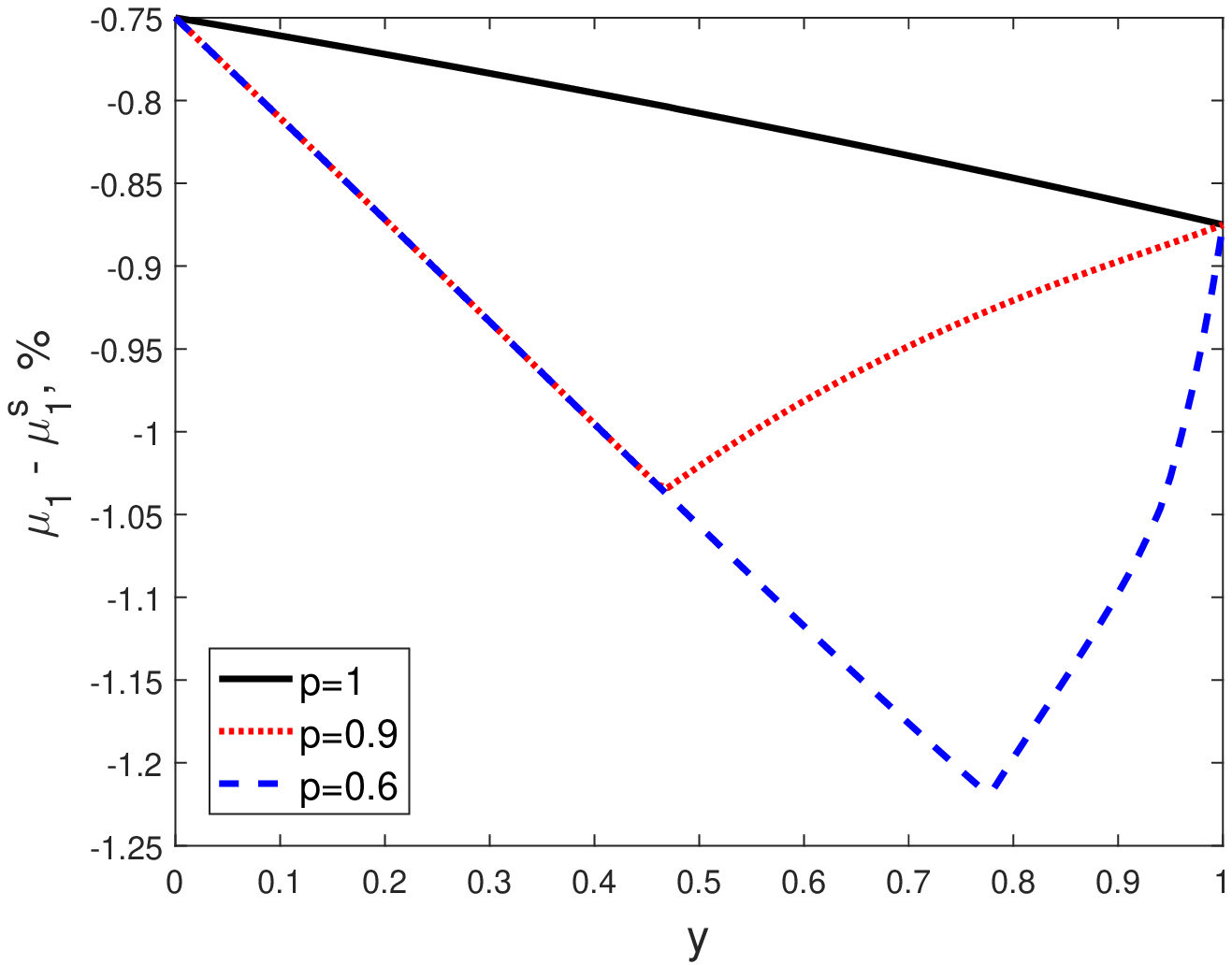}
}\\
\subfigure[Stock volatilities]{
\includegraphics[ width = 0.31\textwidth]{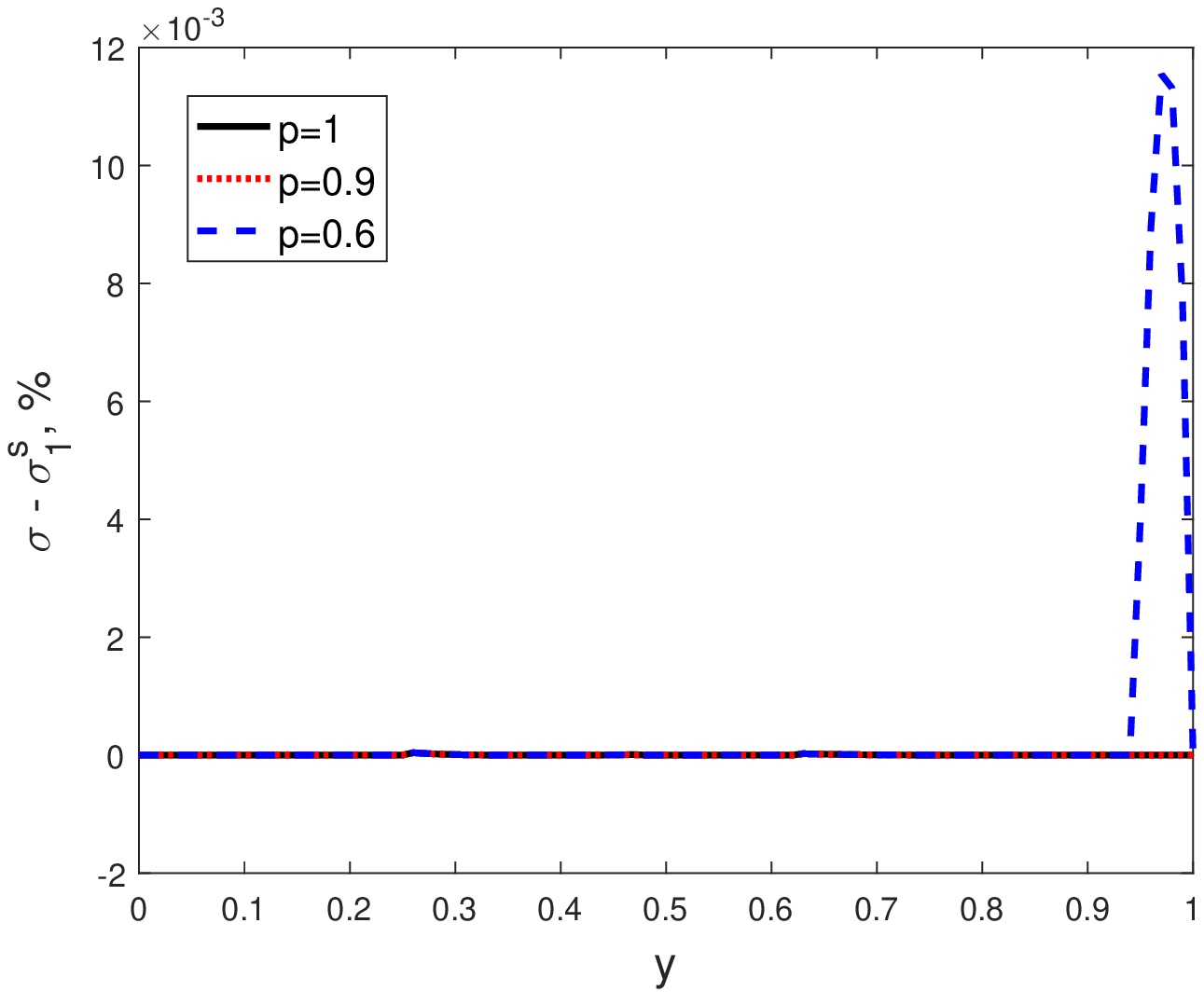}
}
\subfigure[Interest rates]{
\includegraphics[ width = 0.31\textwidth]{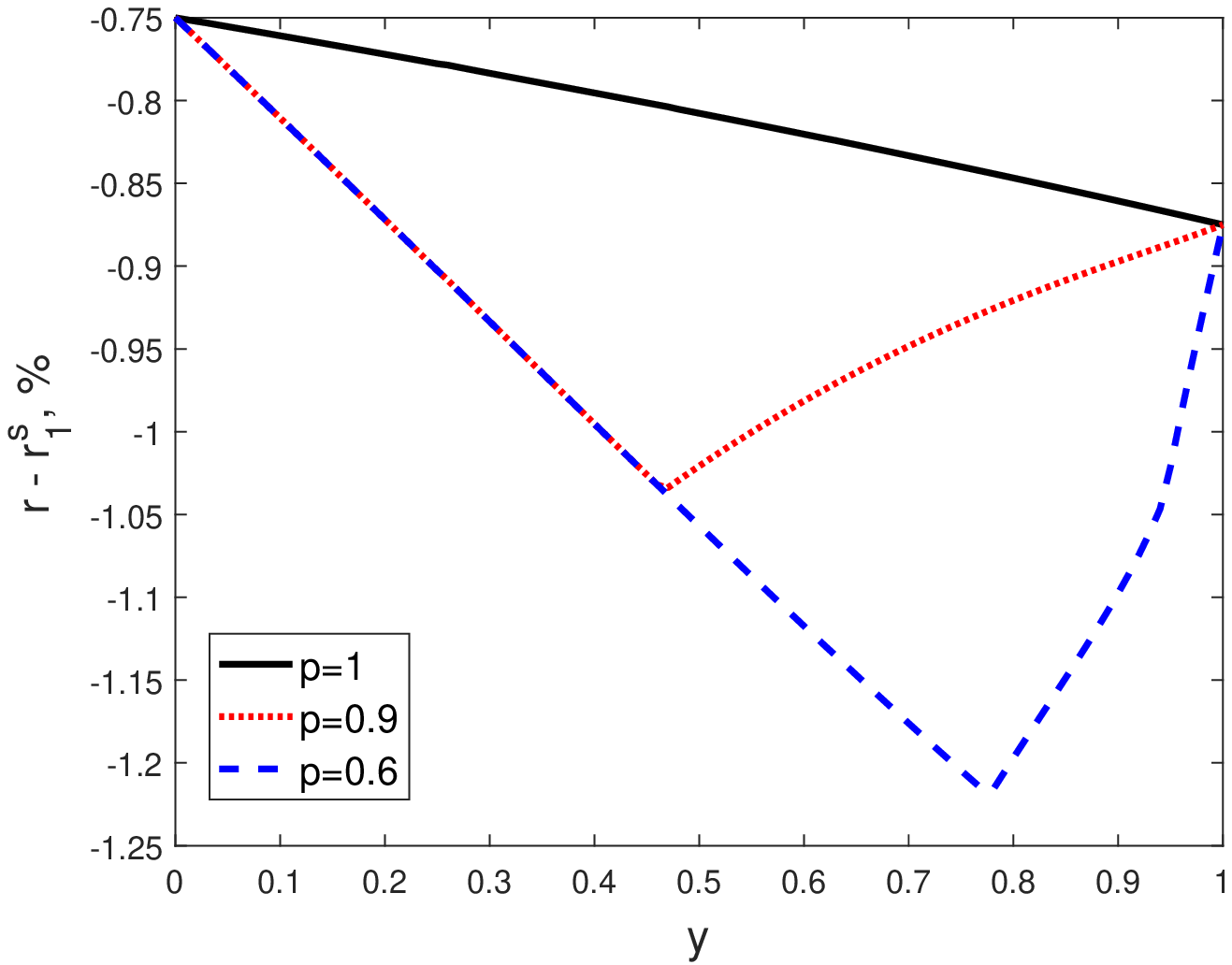}
}
\caption{Comparison of the equilibrium in firm $1$ between the cross-sectional economy and the economy with the single relative firm }
\label{fig_compare1}
\end{figure}

\begin{figure}[htb]
\centering
\subfigure[Stock holdings]{
\includegraphics[ width = 0.31\textwidth]{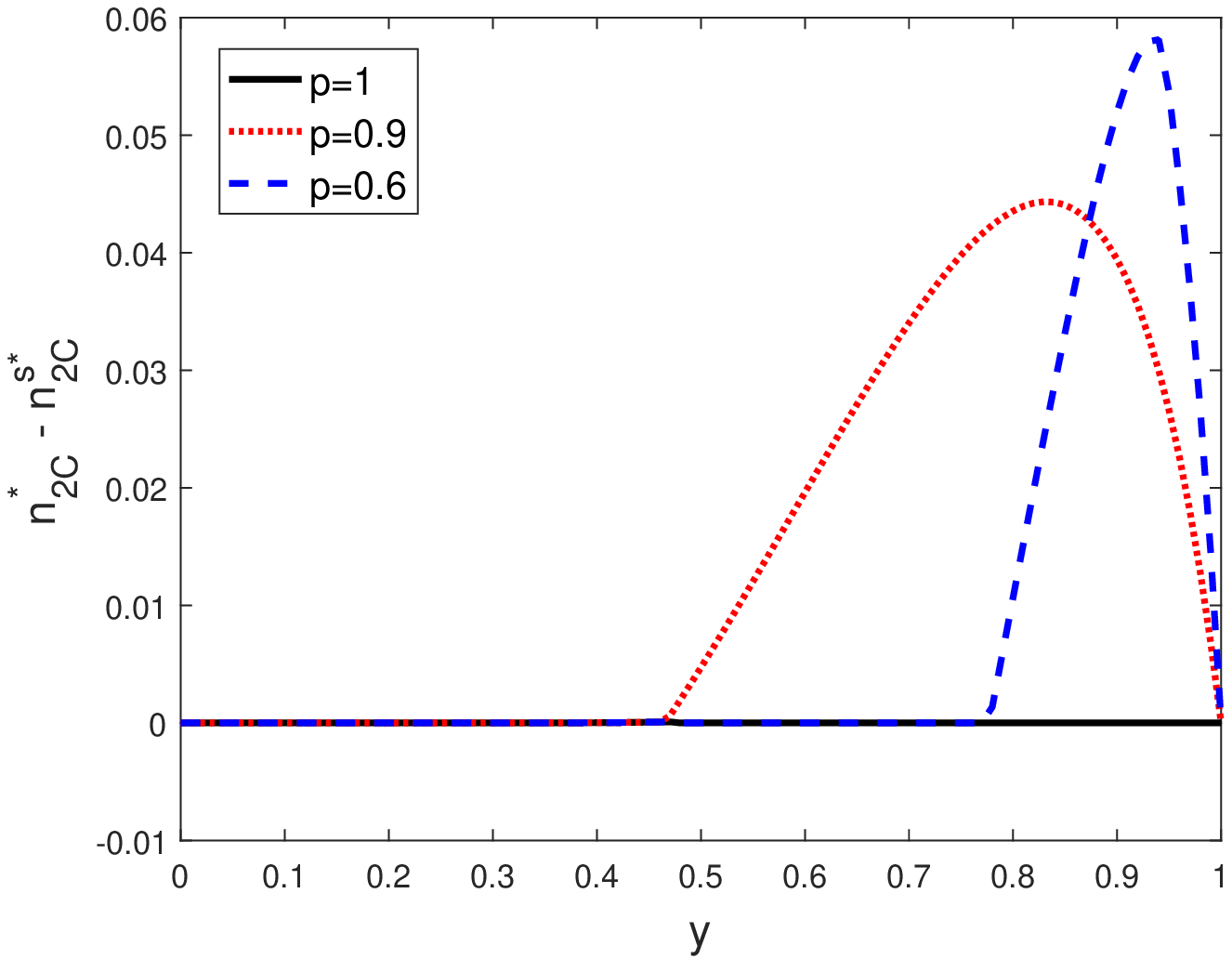}
}
\subfigure[Fractions of diverted output]{
\includegraphics[ width = 0.31\textwidth]{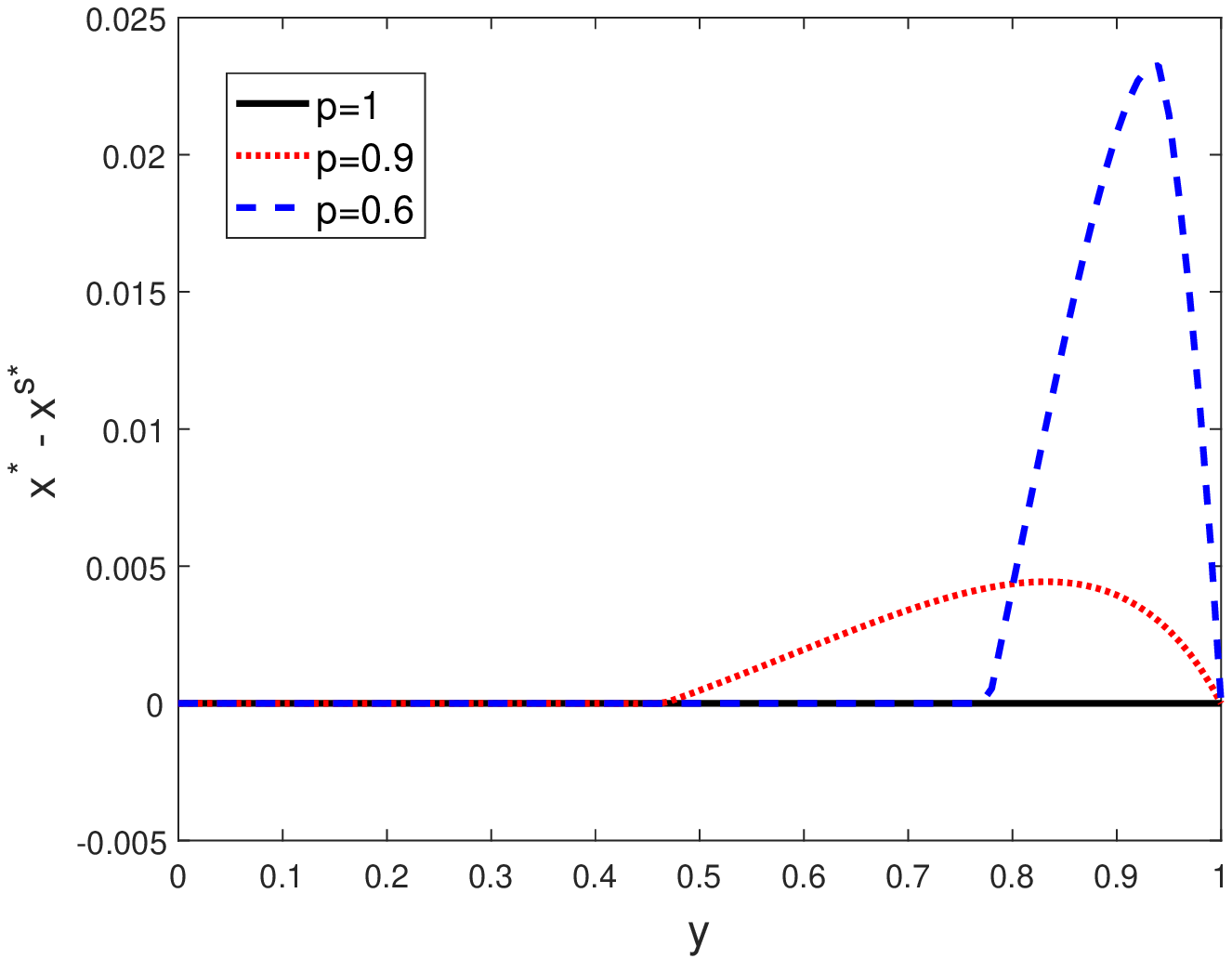}
}\\
\subfigure[Stock returns]{
\includegraphics[ width = 0.31\textwidth]{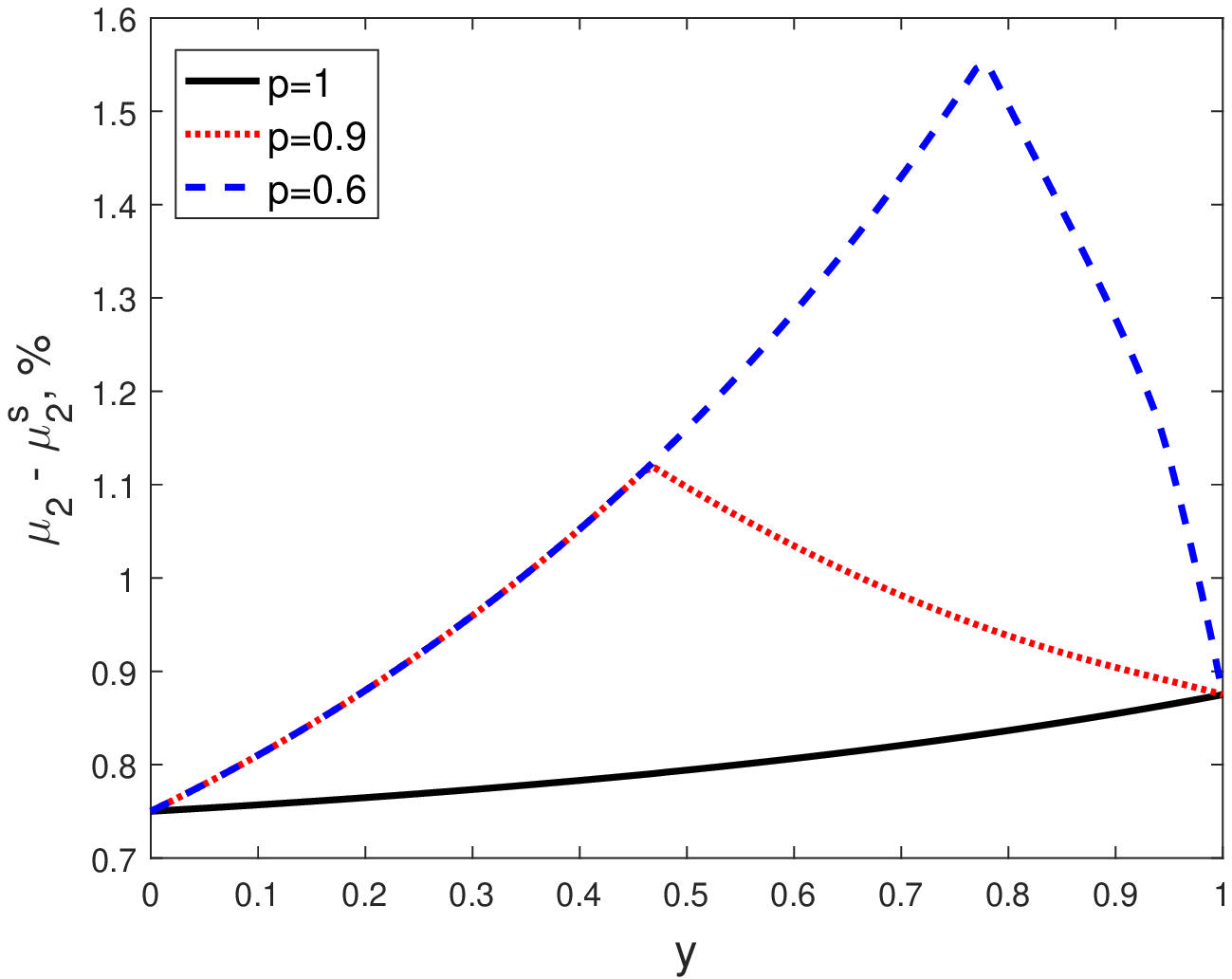}
}
\subfigure[Stock volatilities]{
\includegraphics[ width = 0.31\textwidth]{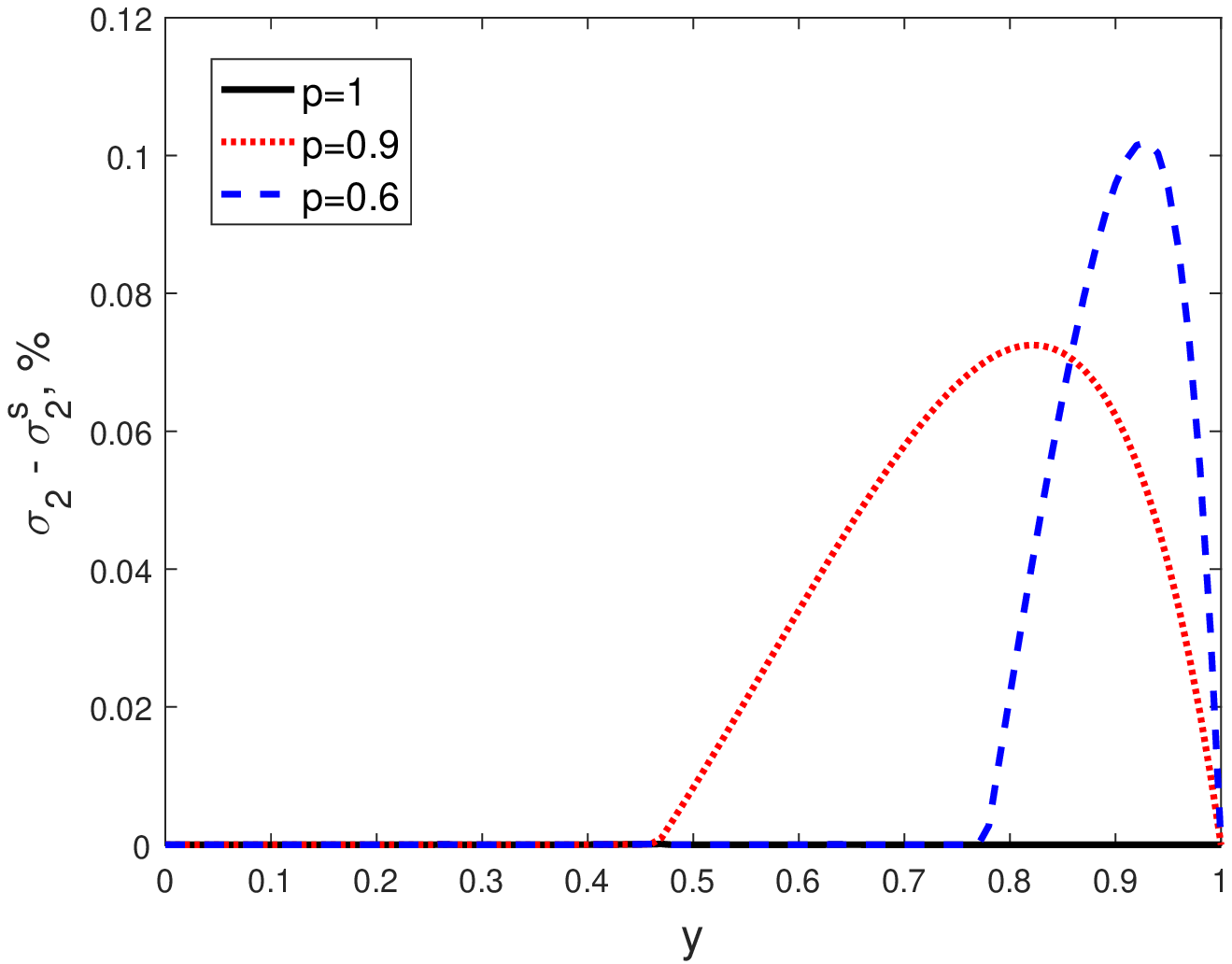}
}
\subfigure[Interest rates]{
\includegraphics[ width = 0.31\textwidth]{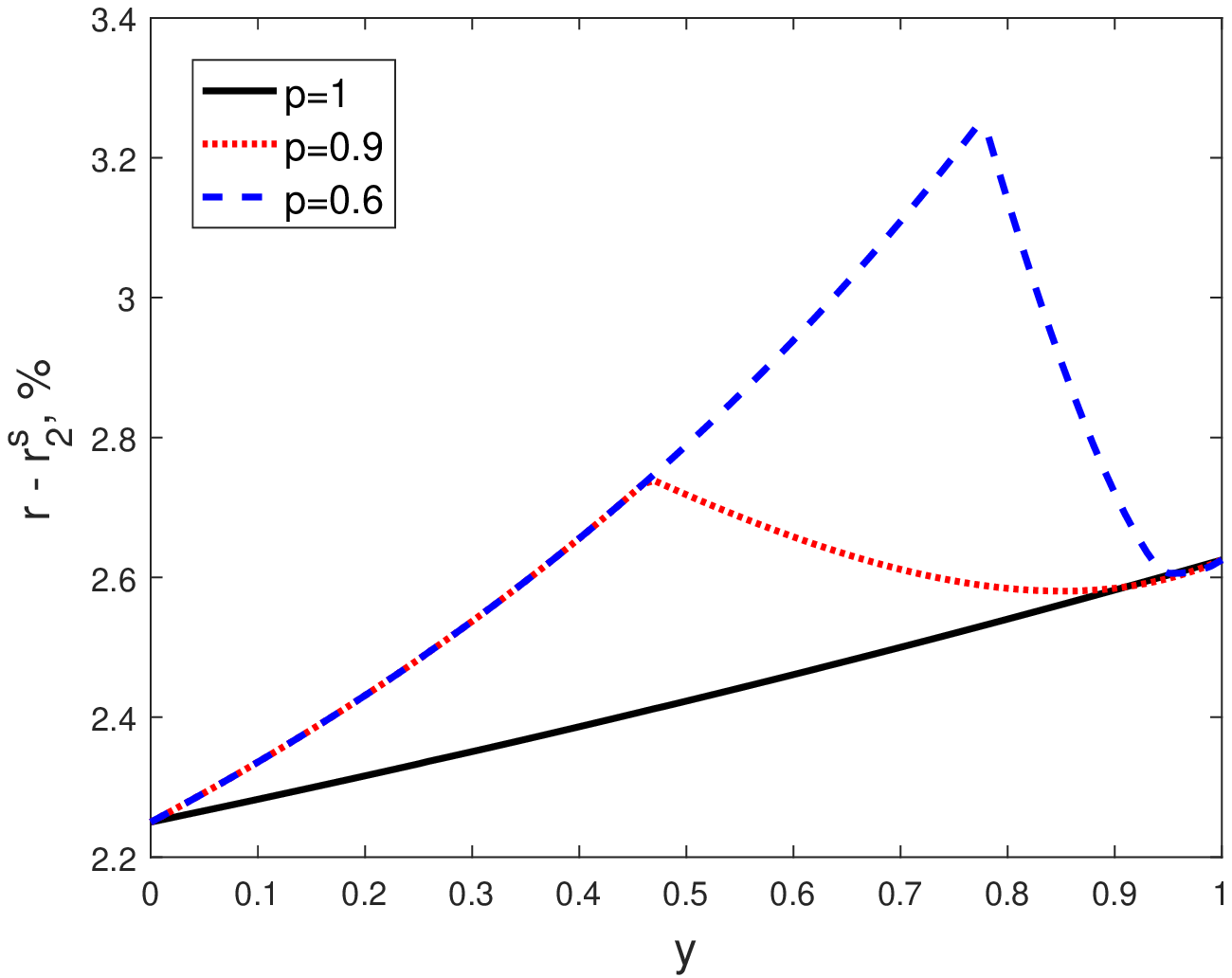}
}
\caption{ Comparison of the equilibrium in firm $2$ between the cross-sectional economy and the economy with the relative single firm }
\label{fig_compare2}
\end{figure}

Figure \ref{fig_compare1} demonstrates that compared to the single relative firm with perfect protection, the cross-section setting in equilibrium with poorer protection in firm $2$ decreases the controlling shareholder's stock holdings of firm $1$, leads to lower stock returns and higher stock volatilities of firm $1$, and reduces interest rates in the cross-sectional economy. Panel (a) shows that, in the cross-sectional economy with perfect protection for both firms the controlling shareholder holds the same shares of firm $1$ as the economy with the single relative firm, while in the cross-sectional economy with imperfect protection for firm $2$ he would hold less shares of firm $1$ than the economy with the single relative firm. The reasons are expressed as follows. For the economy with perfect protection for both firms, there is no difference between the controlling shareholder and the minority shareholder to determine optimal strategies, and hence the cross-section setting in equilibrium has few effects on the controlling shareholder's stock holdings of firm $1$. However, in the cross-sectional economy with imperfect protection for firm $2$, diverting output of firm $2$ drives the controlling shareholder to acquire more shares of firm $2$ and makes, through the cross-section setting in equilibrium, stocks of firm $1$ without diverting output less attractive to the controlling shareholder. Panel (c) implies that the cross-section setting in equilibrium has few effects on stock volatilities of firm $1$. This is because stock volatilities originate from uncertainty of the output in firm $1$ (or equivalently, the fundamental risk in the economy) and the output volatility is same in both economies. Panel (d) reveals that no matter whether investor protection of firm $2$ is perfect in the cross-sectional economy, interest rates are lower than those in the economy with the single relative firm. The reason is twofold. For one thing, due to the cross-section setting in equilibrium, there exists extra risk of covariance and this makes shareholders acquire more bonds and provide cheaper credit, which leads to lower interest rates even in the cross-sectional economy with perfect protection for firm $2$. For another, when there exists diverting output in firm $2$, the minority shareholder would invest more in bonds with cheaper credit, which, to a larger extent, lowers interest rates. As for stock returns, almost same stock volatilities should cause almost same excessive risk and then stock returns are mainly determined by interest rates. Hence, Panel (b) verifies that, just like interest rates, no matter whether investor protection of firm $2$ is perfect in the cross-sectional economy, stock returns are lower than those in the economy with the single relative firm.

Figure \ref{fig_compare2} demonstrates that compared to the single relative firm with the same level of investor protection, the cross-section setting in equilibrium with poorer protection in firm $2$ leads to higher stock holdings of firm $2$ for the controlling shareholder, higher fractions of diverted output in firm $2$, higher stock returns and volatilities of firm $2$, and higher interest rates in the cross-sectional economy. Panel (a) shows that, in the cross-sectional economy with perfect protection for both firms the controlling shareholder holds the same shares of firm $2$ as the economy with the single relative firm (which has perfect protection), while in the cross-sectional economy with imperfect protection for firm $2$ he would hold more shares of firm $2$ than the economy with the single relative firm (which has the same level of investor protection). The reasons are clear. Similar to discussion in Figure \ref{fig_compare1} for the economy with perfect protection for both firms, the cross-section setting in equilibrium has few effects on the controlling shareholder's stock holdings of firm $2$. However, in the cross-sectional economy with imperfect protection for firm $2$, the cross-section setting in equilibrium provides the minority shareholder with alternative choice---investing more in stocks of firm $1$ with perfect protection, and consequently, though with same level of investor protection, the controlling shareholder is able to hold more shares in firm $2$ than those in the economy with single relative firm.
Panel (b) depicts that imperfect protection in firm $2$ enables the controlling shareholder to divert more output of firm $2$ in the cross-sectional economy than that in the economy with the single relative firm. This is because, due to the cross-section setting, more shares of firm $2$ offer the controlling shareholder looser constraints of diverting output of firm $2$. Panel (d) implies that, compared with the single relative firm with the same level of investor protection, the cross-section setting in equilibrium has few effects on stock volatilities of firm $2$ in the cross-sectional economy with perfect protection for both firms and increases stock volatilities of firm $2$ in the cross-sectional economy with imperfect protection for firm $2$. The reason for the former is that with perfect protection for both firms, the same output volatility leads to the same stock volatility. And the reason for the later is that the fact that the controlling shareholder holds more shares of firm $2$ makes stocks more under-diversified and then stocks in firm $2$ is more volatile. Panel (e) reveals that no matter whether investor protection of firm $2$ is perfect in the cross-sectional economy, interests rates are higher than those in the economy with the single relative firm. The main reason can be expressed as follows. Though there is extra covariance risk, the cross-section setting in equilibrium provides shareholders with alternative choices---investing in stocks of firm $1$ with lower stock volatility and poorer protection in firm $2$ reinforces such choice for the minority shareholder, while shareholders in the economy with the single relative firm can only invest in stocks of firm $2$ with higher stock volatility. Hence less bonds are required in the cross-sectional economy than those in the economy with the single relative firm, and this leads to higher interest rates.
As for stock returns, by noticing that the difference of stock volatilities between two economies is small, similar to Figure \ref{fig_compare1}, stock returns are mainly determined by interest rates and then an approximate level of stock volatilities should cause an approximate level of excessive risk. Hence, panel (c) verifies that, just like interest rates, stock returns in the cross-sectional economy are higher than those in the economy with single relative firm.

From the point of competition in the case of imperfect protection $p<1$, though cross-section introduces fiercer competition into firms, it has opposite effects on firm stock returns.  On the one hand, competition between firms barely change the stock volatilities of firm 1 but it hurts development and growth of firm 1 leading to lower stock returns. This is similar to negative relation between competition and growth in product markets (see, for example, \cite{Aghion,Grossman}). On the other hand, stock returns of firm 2 are mainly determined by its stock volatilities instead of competition between firms, i.e., although fiercer competition between firm may decrease stock returns of firm 2, diverting output causes higher stock volatilities resulting in higher stock returns. The result in firm 2 is similar to ``Darwinian" view of competition (see, for example, \cite{Nickell,Porter}) pointing at a positive correlation between competition and growth within a firm.

Summarizing, compared to the economy with the single relative firm, the cross-section setting in equilibrium provides shareholders with more choice of optimal strategies and reinforces the importance of better investor protection for the minority shareholder in the cross-sectional economy. Therefore, the cross-sectional economy plays a different and important role in asset prices and investor protection.

\subsection{The effect of stock-value ratio $y_2$}
Figure \ref{fig_y2} depicts the effect of stock-value ratio $y_2$ on the controlling shareholder's optimal strategies and parameters of asset prices in equilibrium. We note that from our setting of numerical analysis, lower stock-value ratio $y_2$ equivalently means lower stock price and lower market capitalization of firm $1$, and higher stock price and higher market capitalization of firm $2$. Figure \ref{fig_y2} is further explained as follows: all parameters in panels (a)-(h) are obtained by substituting each $y$ and $p=0.6,\tau=1,y_1=0.5,\delta_D=10\%$, and parameters in Table \ref{table} into Proposition \ref{th4} in the cases of $y_2=0.9$, $y_2=0.5$ and $y_2=0.1$, respectively. Based on the kinks in panels (a) and (b), with the increase of $y$, regions of parameters change in the following forms: Region $10$ and then Region $1$  in the case of $y_2=0.9$; Region $10$, then Region $1$ and finally Region $2$ in the cases of $y_2=0.5$ and $y_2=0.1$.

\begin{figure}[htb]
\centering
\subfigure[Stock holdings in firm $1$]{
\includegraphics[ width = 0.31\textwidth]{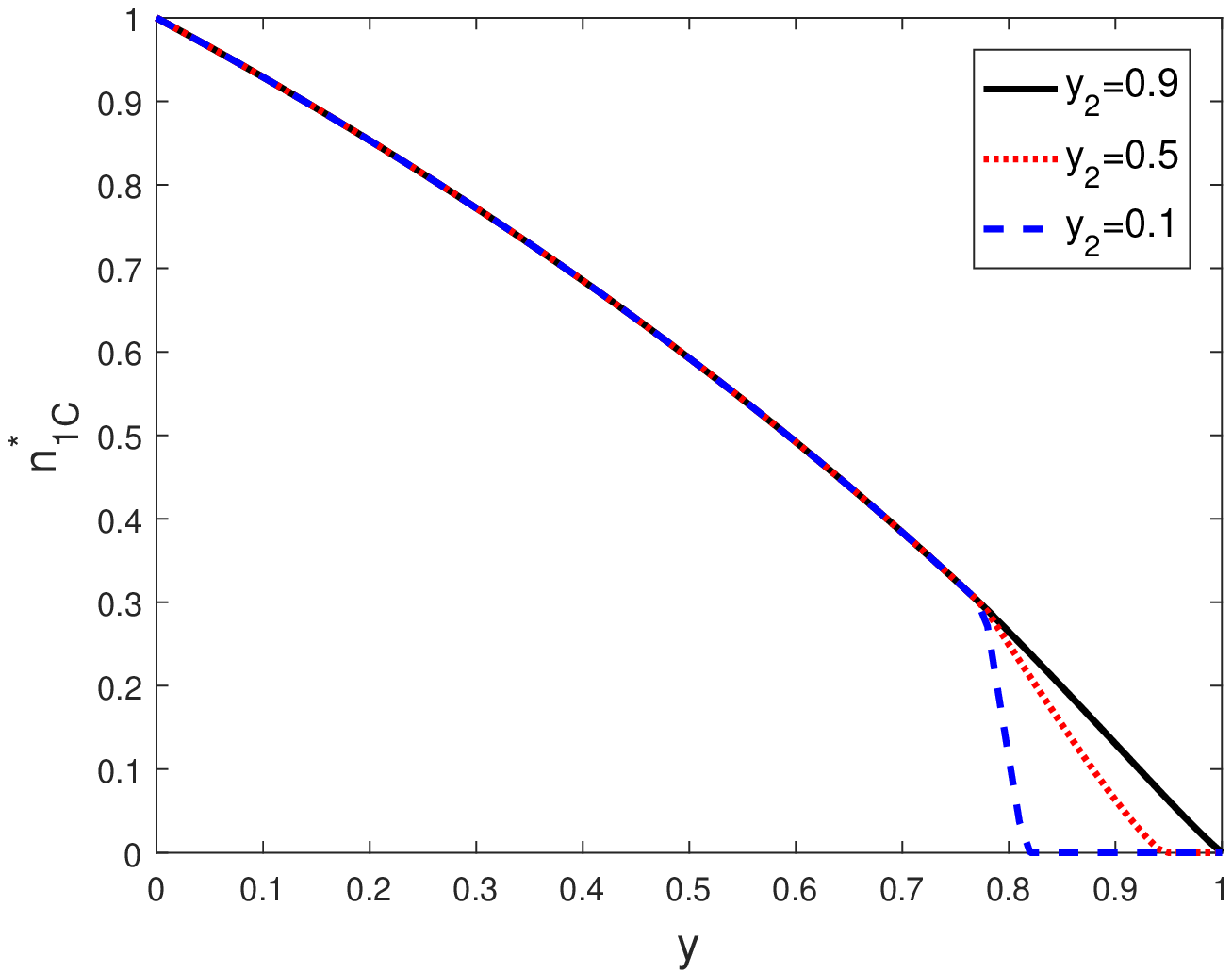}
}
\subfigure[Stock holdings in firm $2$]{
\includegraphics[ width = 0.31\textwidth]{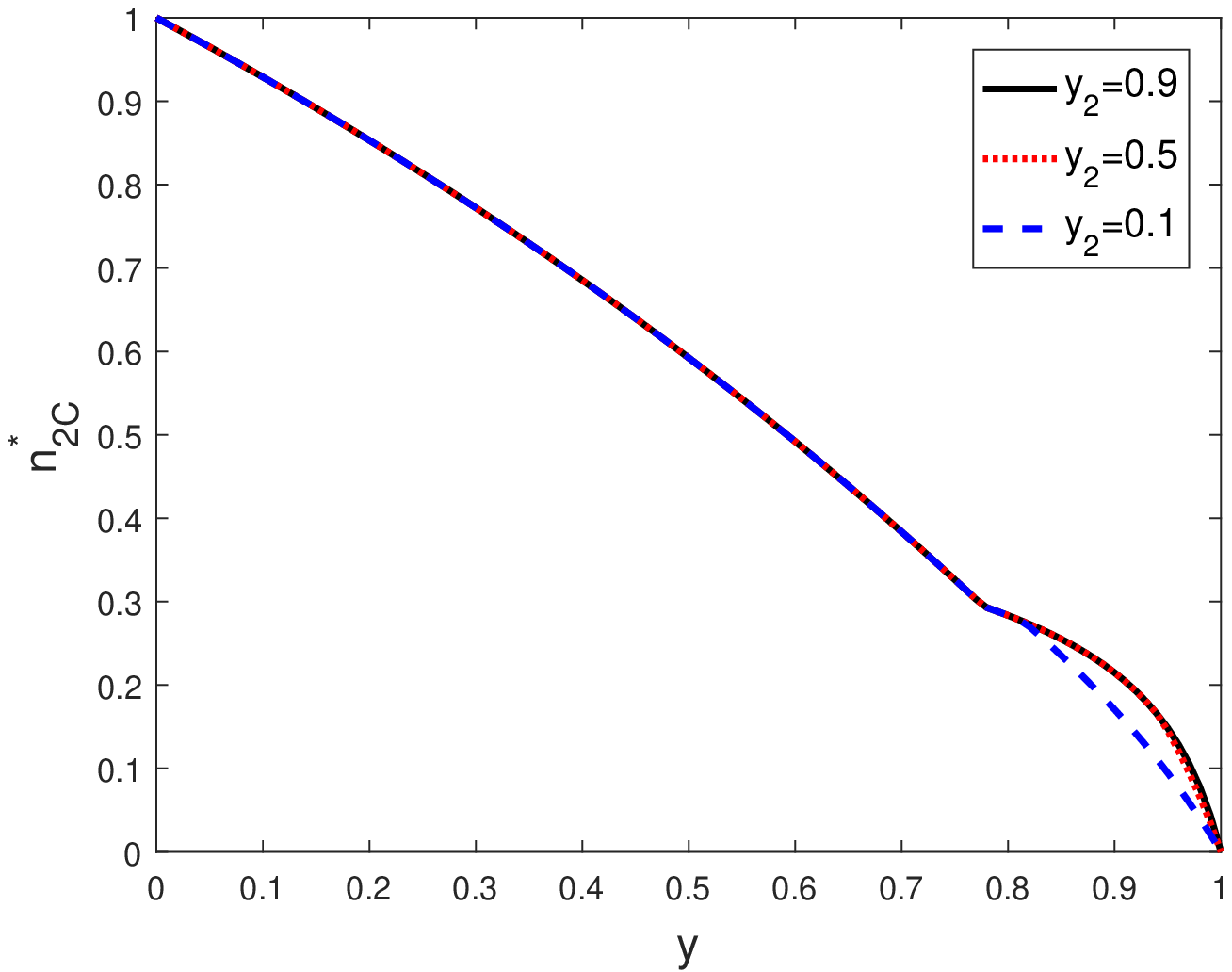}
}
\subfigure[Fractions of diverted output]{
\includegraphics[ width = 0.31\textwidth]{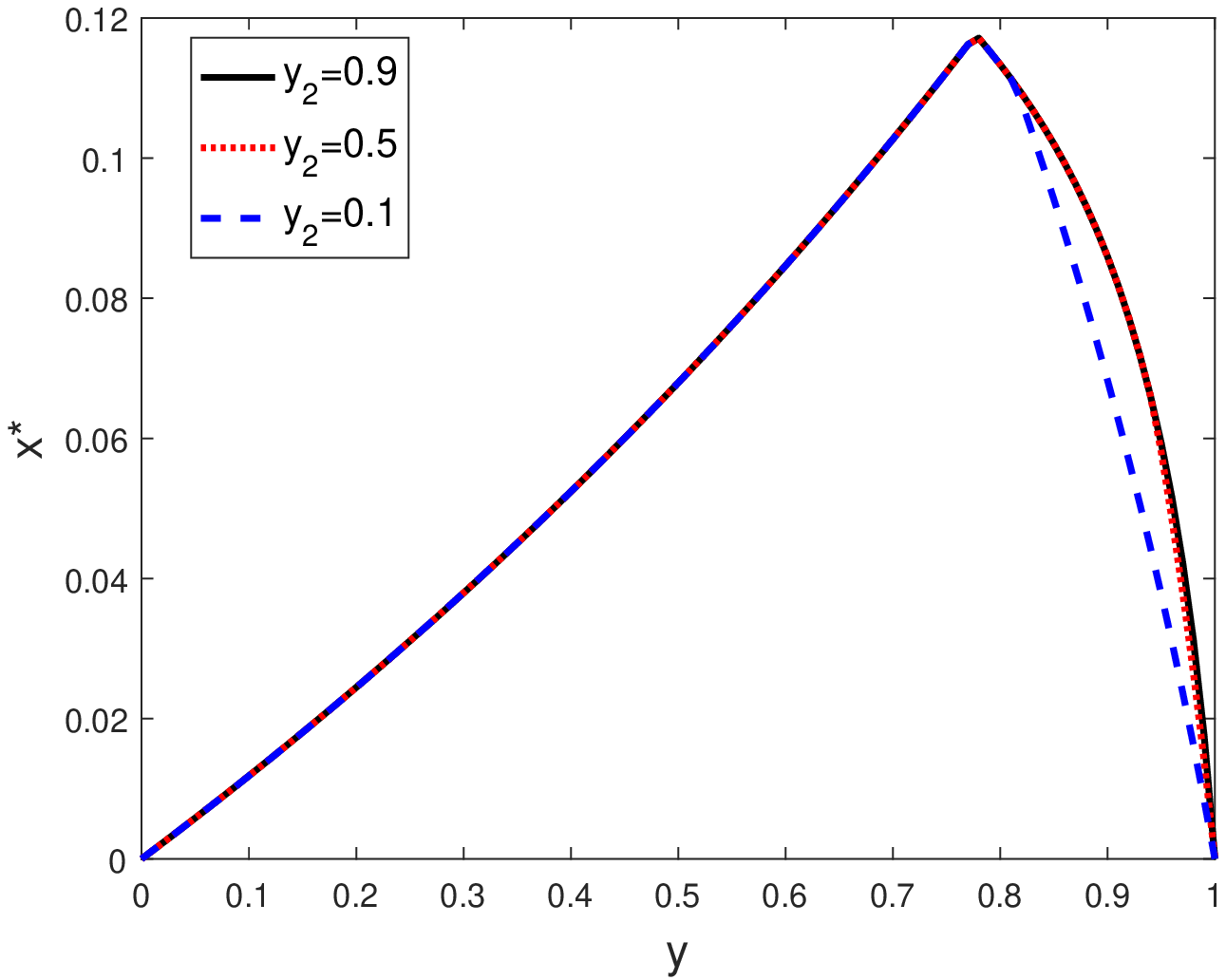}
}
\subfigure[Stock gross returns in firm $1$]{
\includegraphics[ width = 0.31\textwidth]{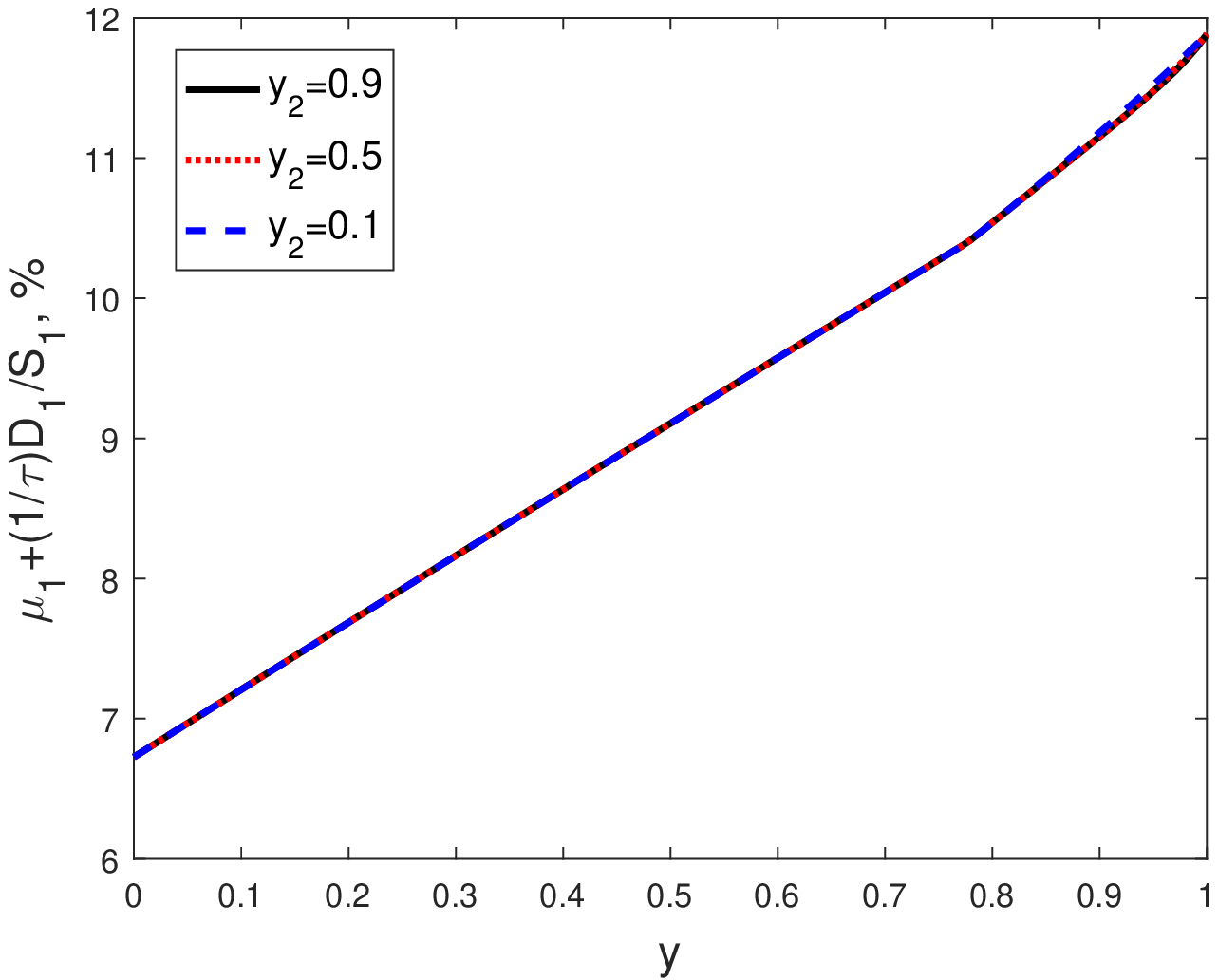}
}
\subfigure[Stock gross returns in firm $2$]{
\includegraphics[ width = 0.31\textwidth]{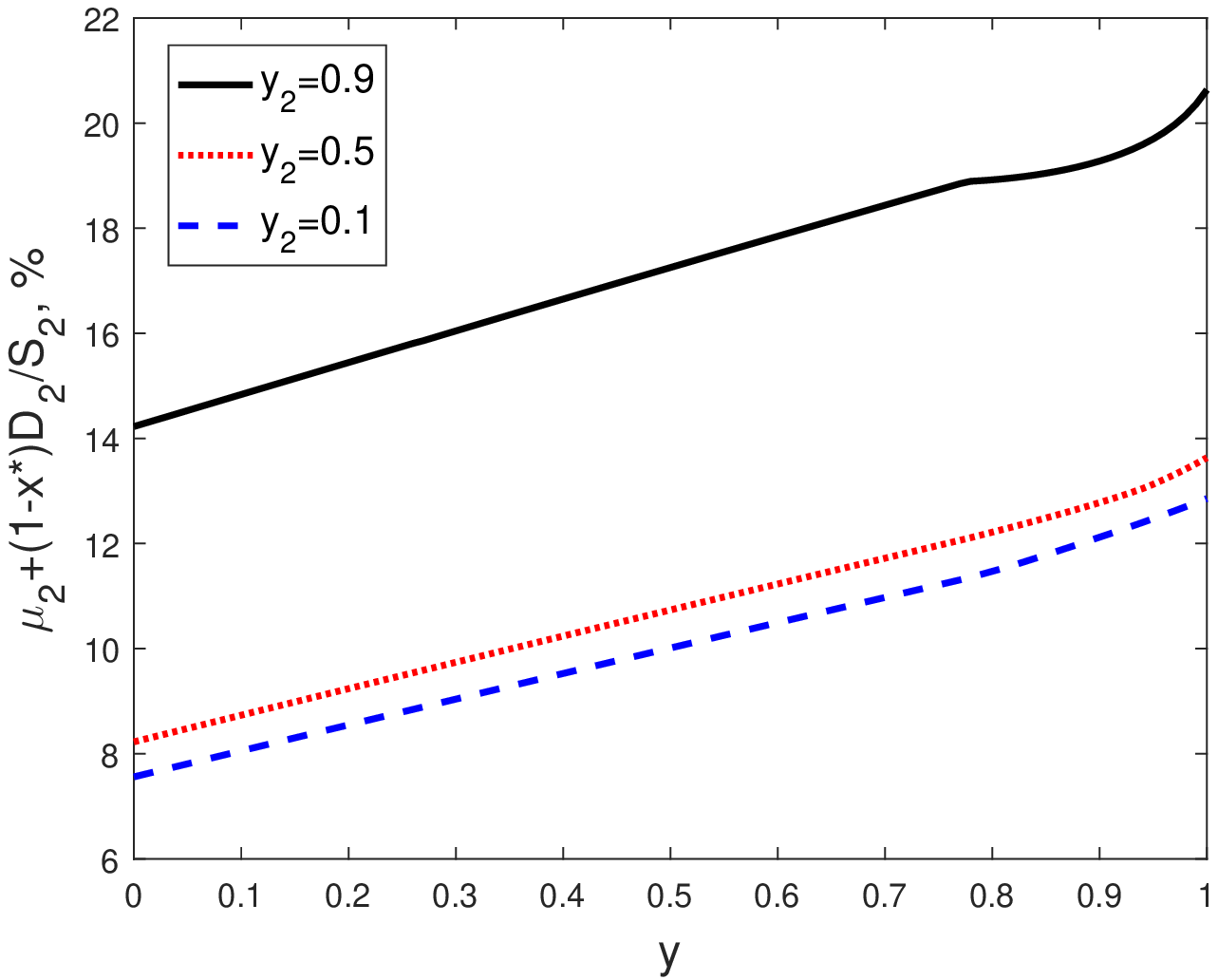}
}
\subfigure[Fundamental stock volatilities]{
\includegraphics[ width = 0.31\textwidth]{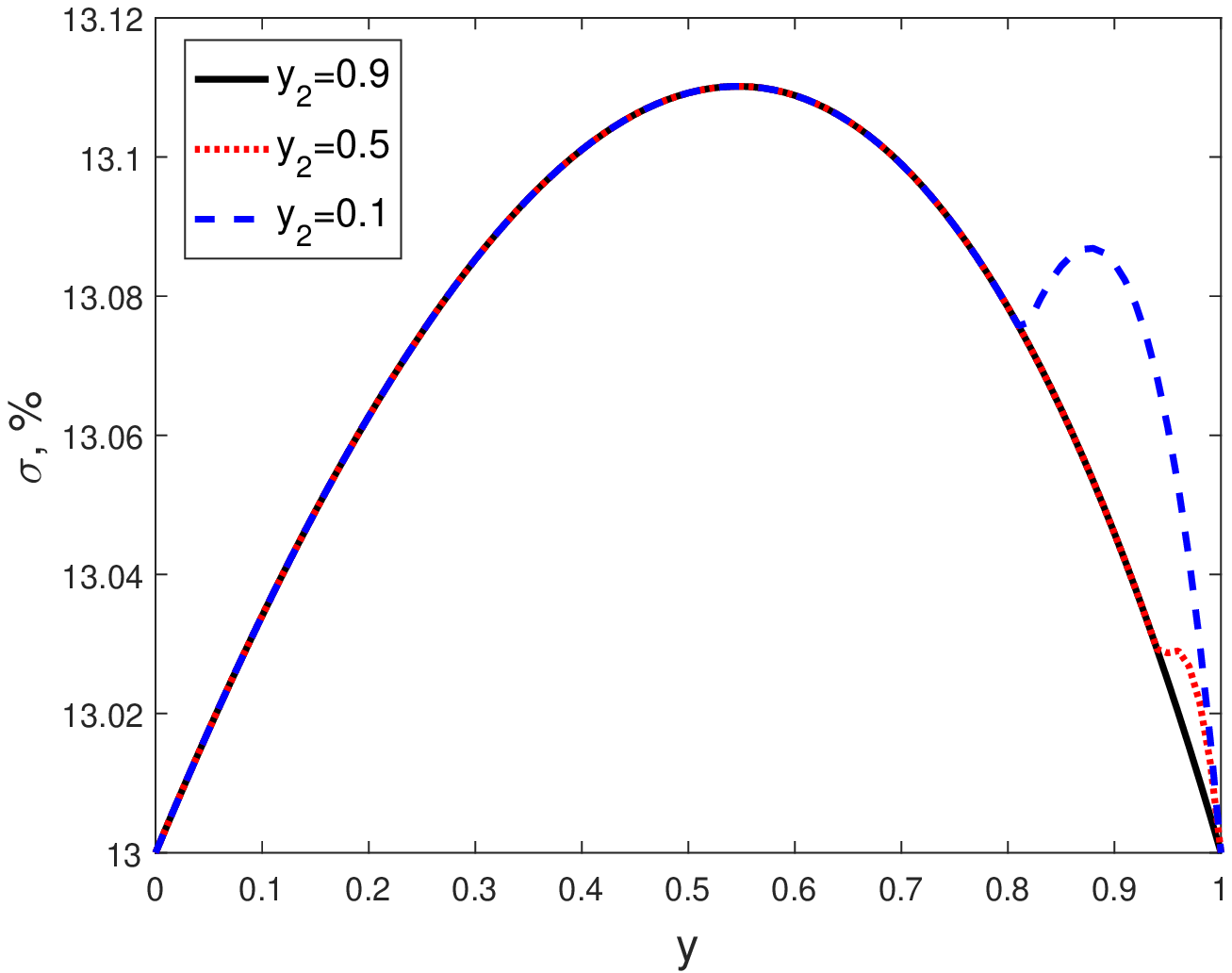}
}
\subfigure[Nonfundamental stock volatilities]{
\includegraphics[ width = 0.31\textwidth]{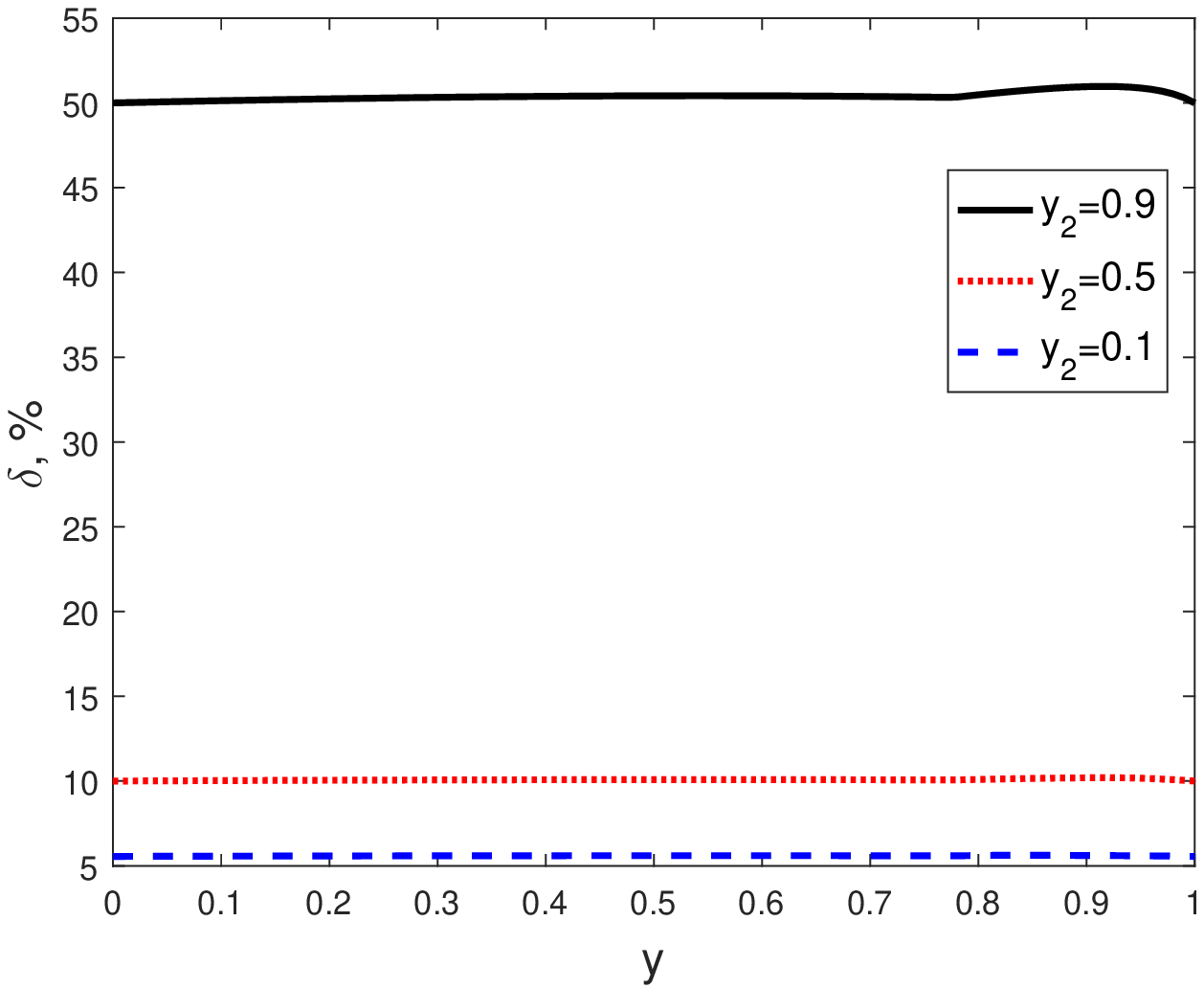}
}
\subfigure[Interest rates]{
\includegraphics[ width = 0.31\textwidth]{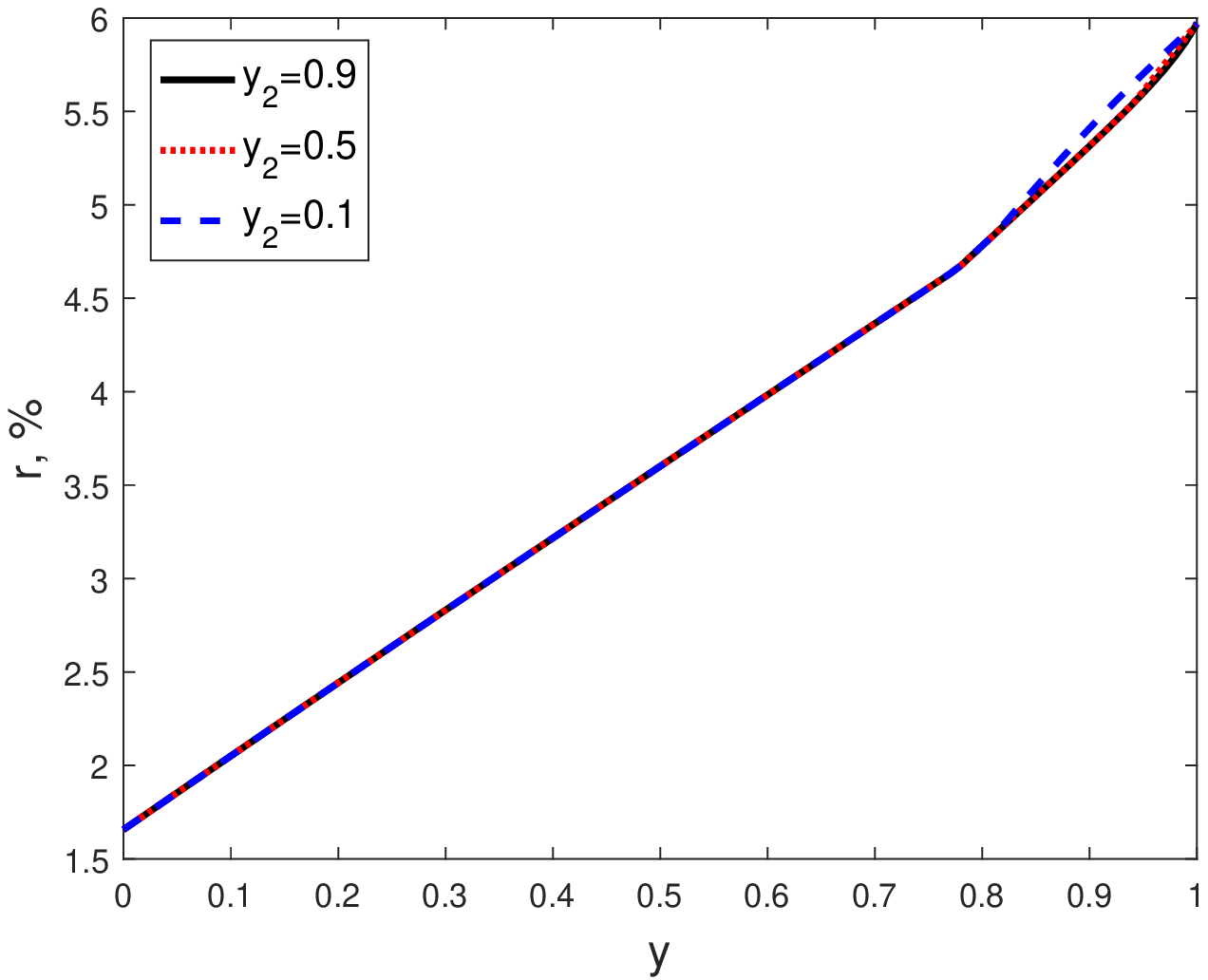}
}
\caption{The effect of the stock-value ratio in the cross-sectional economy}
\label{fig_y2}
\end{figure}

Panels (a), (d) and (f) demonstrate that in firm $1$, lower stock-value ratio $y_2$ leads to lower stock holdings of the controlling shareholder, slightly higher stock gross returns and higher stock volatilities. The main reason is twofold. On the one hand, when $y$ is low (Region 10), the investor protection constraint is not binding such that the controlling shareholder's consumption-portfolio choice of firm $1$ is determined by his risk aversion instead of stock value. And hence lower stock price and lower market capitalization of firm $1$ have few effects on the controlling shareholder's stock holding. The same stock holdings of the shareholders result in the same stock volatilities and then the same risk premia (stock gross returns) in firm $1$. On the other hand, when $y$ is low (Regions $1$ and $2$), the investor protection constraint is binding, and due to imperfect investor protection, the controlling shareholder would focus more on the stocks of firm $2$. Then based on risk aversion, lower stock price and lower market capitalization of firm $1$ make stocks of firm $1$ less attractive to the controlling shareholder, which leads to lower stock holdings of the controlling shareholder in firm $1$. As for stock volatility, since the minority shareholder would prefer the stocks of firm $1$ and hold much higher shares of firm $1$ (especially, in Region 2, the condition ($\mathcal{C}1$) holds and the controlling shareholder is extinct), higher covariance risk between two firms increases stock volatility of firm $2$. And finally, higher stock gross return compensates the shareholders for higher risk taking.

Panels (b), (e) and (g) (together with (f)) show that in firm $2$, lower stock-value ratio $y_2$ leads to lower stock holdings of the controlling shareholder, lower stock gross returns and lower stock volatilities. The reason for lower stock holdings can be explained as follows. For stock holdings with lower $y$ (Region 10), as is discussed in Figure \ref{fig_sx}, the shareholders keep almost the same shares as the case of perfect protection and this indicates that investor protection has few effects on the shareholders' stock holdings in firm $2$. Then similar to panel (a) of Figure \ref{fig_y2}, the same stock holdings of the shareholders are mainly determined by risk aversion instead of stock value. For stock holdings with higher $y$ (Regions 1 and 2), the investor protection constraint makes the controlling shareholder acquire more shares of firm 2, while higher stock price makes it harder for the controlling shareholder to hold more shares of firm 2 (because more wealth is needed to buy one unit of stock). Hence, higher stock price forces the controlling shareholder to keep less shares of firm 2. The reason for lower stock gross returns and lower stock volatilities is based on risk premium. Indeed, higher stock price decreases the risk premium relative to the output of firm 2 (i.e., the ratio $\frac{D_2}{S_2}$), which naturally decreases the stock gross return. And then lower stock gross return, in general, leads to lower stock volatility because of no free lunch with vanishing risk.

Panel (c) reveals that lower stock-value ratio $y_2$ would decrease the fractions of diverted output in firm $2$. Since the fractions of diverted output directly depend on the controlling shareholder's stock holdings of firm $2$, the reason is clear: with lower stock-value ratio $y_2$, the same stock holdings of the controlling shareholder in firm $2$ lead to the same fractions of diverted output ($x^*=\frac{1-n^*_{2C}}{k}$) in Region $10$ and lower stock holdings of the controlling shareholder in firm $2$ result in lower fractions of diverted output ($x^*=(1-p)n^*_{2C}$) in Regions $1$ and $2$.

Panel (h) implies that lower stock-value ratio $y_2$ would increase interest rates. In region $10$, the same diversity of stocks in firms $1$ and $2$ naturally leads to the same interest rates irrelative to stock price. With lower stock-value ratio in Regions $1$ and $2$, the minority shareholder holds more stocks in both firms $1$ and $2$, or equivalently, stocks in the economy are more attractive to him. Therefore, the minority shareholder would invest less in bonds and is not willing to provide cheaper credit, which cause higher interest rates in Regions $1$ and $2$.

When the stock-value ratio is small (i.e., $y_2=0.1$), stocks in firm 1 are called small stocks and stocks in firm 2 are called big stocks (see, \cite{Fama}). Conversely, when the stock-value ratio is big (i.e., $y_2=0.9$), stocks in firm 1 are called big stocks and stocks in firm 2 are called small stocks. Hence, the effect of stock-value ratio $y_2$ can be viewed as ``size effect". The size effect is of little difference in stock returns of firm 1, while it has a significant effect on stock returns of firm 2. The result of stock returns of firm 2 is consistent with empirical evidence in \cite{Banz} finding that smaller firms have higher risk adjusted returns, on average, than larger firms; it is also in agreement with \cite{Fama} showing that spreads in average momentum returns decrease from smaller to bigger stocks.

\section{An extension: investor protection in both firms}\label{section5}\noindent
\setcounter{equation}{0}
In this section, we extend the investor protection constraint in firm $2$ (i.e., (\ref{protectC}) in Section \ref{section3}) by assuming that the controlling shareholder can also divert output in firm $1$ with a similar investor protection constraint.

Providing that the economy settings (\ref{D1})-(\ref{W}) hold and that the controlling shareholder can divert the output of both firms $1$ and $2$ satisfying the investor protection constraints---(\ref{protectC}) in firm $2$ and
\begin{equation}\label{protectC1}
0\leq x'(t)\leq (1-p')\frac{n_{1C}(t)}{\tau}
\end{equation}
in firm $1$, then shareholders' wealth for $i=C,M$ can be expressed as follows:
\begin{align}
dX_i(t)=&\bigg[X_i(t)r(t)+l_{1i}\widehat{D}_1(t)+l_{2i}\widehat{D}_2(t)-c_i(t)
+n_{1i}(t)\left(S_1(t)(\mu_{1}(t)-r(t))+(1-x'(t))\frac{D_1(t)}{\tau}\right)\nonumber\\
&+n_{2i}(t)\left(S_2(t)(\mu_{2}(t)-r(t))+(1-x(t))D_2(t)\right)+\mathbf{1}_{\{i=M\}}(f'(x'(t),D_1(t))+f(x(t),D_2(t)))\nonumber\\
&+\mathbf{1}_{\{i=C\}}(x'(t)D_1(t)-f'(x'(t),D_1(t))+x(t)D_2(t)-f(x(t),D_2(t)))\bigg]dt\nonumber\\
&+(n_{1i}(t)S_1(t)\sigma(t)+n_{2i}(t)S_2(t)\sigma(t))dW(t)
+n_{2i}(t)S_2(t)\delta(t)dB(t),\label{Xx}
\end{align}
where $f'(x',D_1)=\frac{k'(x')^2D_1}{2}$ is the pecuniary cost of firm $1$ from diverting output with a constant $k'$ and the parameter $k'$ captures the magnitude of the cost of firm $1$, $x'$ is the fraction of diverted output in firm 1 with the parameter $p'\in[0,1]$ (interpreted as the protection of the minority shareholder in firm $1$), and other parameters are the same as those in (\ref{X}). Here, the pecuniary function $f'(x',D_1)=\frac{k'(x')^2D_1}{2}$ is equivalent to the form $\frac{k'(x')^2}{2}\frac{D_1}{\tau}$, and we adopt the former to keep the similar expression as that in firm 2. It follows that the optimization problem of the controlling shareholder is (\ref{problemC}) with additional constraint (\ref{protectC1}) and the optimization problem of the minority shareholder remains to be (\ref{problemM}), i.e.,
\begin{align}
\max\limits_{c_{C}(t),\mathbf{x}(t),\mathbf{n}_{C}(t)}V'_C(c_{C}(t),X_{C}(t),X_{C}(t+dt)),\label{proC}\\
\max\limits_{c_{M}(t),\mathbf{n}_{M}(t)}V'_M(c_{M}(t),X_{M}(t),X_{M}(t+dt)),\label{proM}
\end{align}
where $\mathbf{x}=(x',x)^\top$ and $V'_C$ and $V'_M$ are still given by (\ref{wealth}), subject to (\ref{Xx}), the investor protection constraints (\ref{protectC1}) and (\ref{protectC}), and share constraints $\mathbf{n}_{C}(t)\in[0,\tau]\times[0,1], \mathbf{n}_{M}(t)\in \mathbb{R}\times\mathbb{R}$.

It is clear that when the investor protect is perfect in firm $1$, i.e., $p'=1$, above economy is the same as that in Section \ref{section2} and related results are derived in Propositions \ref{th3} and \ref{th4}. Hence, we just consider the case of imperfect protection in both firms $1$ and $2$, i.e., $0\leq p'<1,0\leq p<1$. With the similar method in Section \ref{section3}, consumption problems for shareholders remain to be (\ref{problemCc}) and (\ref{problemMc}) and then in order to solve problems (\ref{proC}) and (\ref{proM}), it suffices to study the following optimization problems with $0\leq p'<1,0\leq p<1$:
for the controlling shareholder,
\begin{equation}\label{pC1}
\left\{
\begin{aligned}
\min \;&-J'_C(\mathbf{x},\mathbf{n}_{C})\\
s.t.\;&h_1(\mathbf{x},\mathbf{n}_{C})\equiv -x'\leq 0;;\\
& h_2(\mathbf{x},\mathbf{n}_{C})\equiv x'-(1-p')\frac{n_{1C}}{\tau}\leq 0;\\
&h_3(\mathbf{x},\mathbf{n}_{C})\equiv -x\leq 0;\\
& h_4(\mathbf{x},\mathbf{n}_{C})\equiv x-(1-p)n_{2C}\leq 0;\\
&h_5(\mathbf{x},\mathbf{n}_{C})\equiv n_{1C}-\tau\leq 0;\\
&h_6(\mathbf{x},\mathbf{n}_{C})\equiv n_{2C}- 1\leq 0,
\end{aligned}
\right.
\end{equation}
and for the minority shareholder $M$,
\begin{equation}\label{pM1}
\min\limits_{\mathbf{n}_{M}\in \mathbb{R}^2} -J'_M(\mathbf{n}_{M}),
\end{equation}
where functions $J'_C$ and $J'_M$ are given by
\begin{align*}
J'_C(\mathbf{x},\mathbf{n}_{C})=&\mathbf{n}_C^\top\theta_C-\frac{1}{2}\mathbf{n}_C^\top\xi_C\mathbf{n}_C
+{n_{1C}}(1-x')\alpha_{1C}+\tau\left(x'-\frac{k'(x')^2}{2}\right)\alpha_{1C}\\
&+n_{2C}(1-x)\alpha_{2C}+\left(x-\frac{kx^2}{2}\right)\alpha_{2C};\\
J'_M(\mathbf{n}_{M})=&\mathbf{n}_M^\top\theta_M-\frac{1}{2}\mathbf{n}_M^\top\xi_M\mathbf{n}_M
+{n_{1M}}(1-x')\alpha_{1M}+n_{2M}(1-x)\alpha_{2M}.
\end{align*}

Solving optimization problems (\ref{problemCc}), (\ref{problemMc}), (\ref{pC1}) and (\ref{pM1}) under the partial equilibrium setting, the shareholders' optimal strategies are obtained as the following proposition.

\begin{proposition}\label{th9}{\it
Provided that there exists investor protection in both firms with $0\leq p'<1, 0\leq p<1$ in the cross-sectional economy, then the optimal consumptions $c_{i}^*$, the fractions of diverted output $(x')^*, x^*$ and the optimal stock holdings $\mathbf{n}_{i}^*$ for $i=C,M$ are given by
\begin{align*}
c_{i}^*=&\rho^{\frac{1}{\gamma_i}}X_{i},\quad i=C,M;\\
(x')^*=&(x')^*(n_{1C}^*)=\min\left\{\frac{1-n_{1C}^*/\tau}{k'},(1-p')\frac{n_{1C}^*}{\tau}\right\};\\
x^*=&x^*(n_{2C}^*)=\min\left\{\frac{1-n_{2C}^*}{k},(1-p)n_{2C}^*\right\};\\
\mathbf{n}_{C}^*=&\eta_{v}^*=(\eta_{1,v}^*,\eta_{2,v}^*)^\top,\quad (\mathrm{Region\;}v) \quad\text{if}\; (J_C')^*=J'_C(x^*,\eta_{v}^*), v=1,\cdots,16;\\
\mathbf{n}_M^*=&(\xi_M)^{-1}\cdot\left[\theta_M+(1-(x')^*){\alpha_{1M}}\mathbf{e}_1+(1-x^*)\alpha_{2M}\mathbf{e}_2\right],
\end{align*}
where Regions $1$-$16$ are given by
\begin{description}
%region 1
  \item [] {$\mathrm{Region\; 1}:\; \eta_1^*=\left(\xi_C+(1-p')(2+k'(1-p'))\frac{\alpha_{1C}}{\tau}\mathbf{e}_{11}
+(1-p)(2+k(1-p))\alpha_{2C}\mathbf{e}_{22}\right)^{-1}$\\
\text{\qquad\qquad}\; $\cdot\left[\theta_C+
(2-p'){\alpha_{1C}}\mathbf{e}_1+\left(2-p\right)\alpha_{2C}\mathbf{e}_{2}\right],$\text{\quad if}\; $0\leq\eta_{1,1}^*< \frac{\tau}{1+k'(1-p')},\; 0\leq\eta_{2,1}^*<\frac{1}{1+(1-p)k}$;}
%region 2
  \item [] {$\mathrm{Region\; 2}:\;
 \eta_2^*=\left(0,\frac{\theta_{2C}+(2-p)\alpha_{2C}}{\xi_{2C}+(1-p)[2+(1-p)k]\alpha_{2C}}
\right)^\top$,\text{\quad if}\; $0\leq\eta_{2,2}^*\leq 1,\;\frac{\theta_{1C}+(2-p'){\alpha_{1C}}}{\xi_{0C}}<\eta_{2,2}^*<\frac{1}{1+(1-p)k}$;}
%region 3
  \item [] {$\mathrm{Region\; 3}:\;
 \eta_3^*=\left(\tau,\frac{\theta_{2C}+(2-p)\alpha_{2C}-\tau\xi_{0C}}{\xi_{2C}+(1-p)[2+(1-p)k]\alpha_{2C}}
\right)^\top$,\text{\quad if}\;
$0\leq\eta_{2,3}^*< \frac{1}{1+(1-p)k},\;\eta_{2,3}^*<\frac{\theta_{1C}+{\alpha_{1C}}-\tau\xi_{1C} }{\xi_{0C}};$}
%region 4
  \item [] {$\mathrm{Region\; 4}:\;
  \eta_4^*=\left(\frac{\theta_{1C}+(2-p'){\alpha_{1C}}}{\xi_{1C}+(1-p')(2+k'(1-p'))\frac{\alpha_{1C}}{\tau}},
0\right)^\top$,\text{\quad if}\;
$0\leq\eta_{1,4}^*,\;\frac{\theta_{2C}+(2-p){\alpha_{2C}}}{\xi_{0C}}<\eta_{1,4}^*<\frac{\tau}{1+k'(1-p')}$;}
%region 5
  \item [] {$\mathrm{Region\; 5}:\;
  \eta_5^*=\left(\frac{\theta_{1C}+(2-p'){\alpha_{1C}}-\xi_{0C}}
{\xi_{1C}+(1-p')(2+k'(1-p'))\frac{\alpha_{1C}}{\tau}},1\right)^\top$,\text{\quad if}\;
$\eta_{1,5}^*<\frac{\theta_{2C}+\alpha_{2C}-\xi_{2C}}{\xi_{0C}},\;0\leq\eta_{1,5}^*< \frac{\tau}{1+k'(1-p')}$;}
%region 6
  \item [] {$\mathrm{Region\; 6}:\;
  \eta_{6}^*=\left(0,0\right)^\top$,\text{\quad  if}\; $\theta_{1C}+(2-p')\alpha_{1C}<0,\;\theta_{2C}+(2-p)\alpha_{2C}<0$;}
  %region 7
  \item [] {$\mathrm{Region\; 7}:\;
  \eta_{7}^*=\left(0,1\right)^\top$,\text{\quad if}\;
  $\theta_{1C}+(2-p'){\alpha_{1C}}<\xi_{0C},\;\xi_{2C}<\theta_{2C}+\alpha_{2C}$;}
  %region 8
  \item [] {$\mathrm{Region\; 8}:\;
  \eta_{8}^*=\left(\tau,0\right)^\top$,\text{\quad if}\;
  $\tau\xi_{1C}<\theta_{1C}+{\alpha_{1C}},\;\theta_{2C}+(2-p)\alpha_{2C}<\tau\xi_{0C}$;}
  %region 9
  \item [] {$\mathrm{Region\; 9}:\;
  \eta_{9}^*=\left(\tau,1\right)^\top$,\text{\quad if}\;
  $\xi_{0C}+\tau\xi_{1C}<\theta_{1C}+{\alpha_{1C}},\;\tau\xi_{0C}+\xi_{2C}<\alpha_{2C}+\theta_{2C}$;}
  %region 10
  \item [] {$\mathrm{Region\; 10}:\;
  \eta_{10}^*=\left(\xi_C+(1-p')(2+k'(1-p'))\frac{\alpha_{1C}}{\tau}\mathbf{e}_{11}-\frac{\alpha_{2C}}{k}
\mathbf{e}_{22}\right)^{-1}\cdot\big[\theta_C
+(2-p'){\alpha_{1C}}\mathbf{e}_1$\;\\
\text{\qquad\qquad}\;$+\left(1-\frac{1}{k}\right)\alpha_{2C}\mathbf{e}_{2}\big]$,\text{\quad if}\;
$0\leq\eta_{1,10}^*< \frac{\tau}{1+k'(1-p')},\;\frac{1}{1+(1-p)k}\leq \eta_{2,10}^*\leq 1$,\\
\text{\qquad\qquad}$\det\left(\xi_C+(1-p')(2+k'(1-p'))\frac{\alpha_{1C}}{\tau}\mathbf{e}_{11}-\frac{\alpha_{2C}}{k}
\mathbf{e}_{22}\right)\neq 0$;\\
%region 10a
  \text{\quad or}\quad$\eta_{10}^*=(\eta_{1,10}^*,\eta_{2,10}^*)^\top$,\text{\quad if}\;$0\leq\eta_{1,10}^*< \frac{\tau}{1+k'(1-p')},\;\frac{1}{1+(1-p)k}\leq \eta_{2,10}^*\leq 1$,\\
\text{\qquad\qquad}\;$\xi_{0C}\eta_{1,10}^*+\left(\xi_{2C}-\frac{\alpha_{2C}}{k}\right)\eta_{2,10}^*
=\theta_{2C}+\left(1-\frac{1}{k}\right){\alpha_{2C}},\;
\left(\xi_{2C}-\frac{\alpha_{2C}}{k}\right)(\theta_{1C}+(2-p'){\alpha_{1C}})=$\\
\text{\qquad\qquad}\;$\xi_{0C}\left(\theta_{2C}+\left(1-\frac{1}{k}\right)\alpha_{2C}\right),\;\det\left(\xi_C+(1-p')(2+k'(1-p'))
\frac{\alpha_{1C}}{\tau}\mathbf{e}_{11}-\frac{\alpha_{2C}}{k}
\mathbf{e}_{22}\right)=0$;}
 %region 11
  \item [] {$\mathrm{Region\; 11}:\;
  \eta_{11}^*=\left(0,\frac{\theta_{2C}+(1-\frac{1}{k})\alpha_{2C}}{\xi_{2C}-
\frac{\alpha_{2C}}{k}}\right)^\top
$,\text{\quad if}\;
$\frac{1}{1+(1-p)k}\leq\eta_{2,11}^*\leq 1,\;\eta_{2,11}^*>\frac{\theta_{1C}+(2-p'){\alpha_{1C}}}{\xi_{0C}},\;\xi_{2C}-
\frac{\alpha_{2C}}{k}\neq 0$;\\
%region 11a
 \text{\quad or}\quad$\eta_{11}^*=(0,\eta_{2,11}^*)^\top$,\text{\quad if}\;
$\frac{1}{1+(1-p)k}\leq\eta_{2,11}^*\leq 1,\;\eta_{2,11}^*>\frac{\theta_{1C}+(2-p'){\alpha_{1C}}}{\xi_{0C}},\;\xi_{2C}-
\frac{\alpha_{2C}}{k}=0$,\\
\text{\qquad\qquad}\;$\theta_{2C}+(1-\frac{1}{k})\alpha_{2C}=0$;}
%region 12
  \item [] {$\mathrm{Region\; 12}:\;
  \eta_{12}^*=\left(\tau,\frac{\theta_{2C}+(1-\frac{1}{k})\alpha_{2C}-\tau\xi_{0C}}
{\xi_{2C}-\frac{\alpha_{2C}}{k}}\right)^\top$,\text{\quad if}\;$\frac{1}{1+(1-p)k}\leq\eta_{2,12}^*\leq 1,\;\eta_{2,12}^*<\frac{\theta_{1C}+{\alpha_{1C}}-\tau\xi_{1C}}{\xi_{0C}}$,\\
\text{\qquad\qquad}\;$\xi_{2C}-
\frac{\alpha_{2C}}{k}\neq 0$;\\
%region 12a
   \text{\quad or}\quad$\eta_{12}^*=(\tau,\eta_{2,12}^*)^\top$,\text{\quad if}\;$\frac{1}{1+(1-p)k}\leq\eta_{2,12}^*\leq 1,\;\eta_{2,12}^*<\frac{\theta_{1C}+{\alpha_{1C}}-\tau\xi_{1C}}{\xi_{0C}},\xi_{2C}-
\frac{\alpha_{2C}}{k}= 0$,\\
\text{\qquad\qquad}\;$\theta_{2C}+(1-\frac{1}{k})\alpha_{2C}-\tau\xi_{0C}=0$;}
 %region 13
  \item [] {$\mathrm{Region\; 13}:\;
  \eta_{13}^*=\left(\xi_C-\frac{\alpha_{1C}}{k'\tau}\mathbf{e}_{11}-\frac{\alpha_{2C}}{k}
\mathbf{e}_{22}\right)^{-1}\cdot\big[\theta_C
+\left(1-\frac{1}{k'}\right)\alpha_{1C}\mathbf{e}_1+\left(1-\frac{1}{k}\right)\alpha_{2C}\mathbf{e}_{2}\big]$,\;\\
\text{\qquad\qquad}\text{ if}\;
$\frac{\tau}{1+k'(1-p')}\leq\eta_{1,13}^*\leq \tau ,\;\frac{1}{1+(1-p)k}\leq \eta_{2,13}^*\leq 1,\;\det\left(\xi_C-\frac{\alpha_{1C}}{k'\tau}\mathbf{e}_{11}-\frac{\alpha_{2C}}{k}
\mathbf{e}_{22}\right)\neq 0$;\\
%region 13a
   \text{\quad or}\quad$\eta_{13}^*=(\eta_{1,13}^*,\eta_{2,13}^*)^\top$,\text{\quad if}\;$\frac{\tau}{1+k'(1-p')}\leq\eta_{1,13}^*\leq \tau ,\;\frac{1}{1+(1-p)k}\leq \eta_{2,13}^*\leq 1$,\\
\text{\qquad\qquad}\;$\left(\xi_{2C}-\frac{\alpha_{2C}}{k}\right)(\theta_{1C}+\left(1-\frac{1}{k'}\right)\alpha_{1C})
=\xi_{0C}\left(\theta_{2C}+\left(1-\frac{1}{k}\right)\alpha_{2C}\right)$,\\
\text{\qquad\qquad}\;$\;\det\left(\xi_C-\frac{\alpha_{1C}}{k'\tau}\mathbf{e}_{11}-\frac{\alpha_{2C}}{k}
\mathbf{e}_{22}\right)= 0,\;\xi_{0C}\eta_{1,13}^*+\left(\xi_{2C}-\frac{\alpha_{2C}}{k}\right)\eta_{2,13}^*
=\theta_{2C}+\left(1-\frac{1}{k}\right){\alpha_{2C}};$}
%region 14
  \item [] {$\mathrm{Region\; 14:}\;
  \eta_{14}^*=\left(\xi_C-\frac{\alpha_{1C}}{k'\tau}
\mathbf{e}_{11}+(1-p)(2+k(1-p))\alpha_{2C}\mathbf{e}_{22}\right)^{-1}\cdot\big[\theta_C
+\left(1-\frac{1}{k'}\right)\alpha_{1C}\mathbf{e}_{1}+(2-p){\alpha_{2C}}\mathbf{e}_2\big],$\;\\
\text{\qquad\qquad}\;\text{if}\;
$\frac{\tau}{1+k'(1-p')}\leq\eta_{1,14}^*\leq\tau,\;0\leq \eta_{2,14}^*<\frac{1}{1+(1-p)k}$,\\
\text{\qquad\qquad}$\det\left(\xi_C-\frac{\alpha_{1C}}{k'\tau}
\mathbf{e}_{11}+(1-p)(2+k(1-p))\alpha_{2C}\mathbf{e}_{22}\right)\neq 0$;\\
%region 14a
   \text{\quad or}\quad$\eta_{14}^*=(\eta_{1,14}^*,\eta_{2,14}^*)^\top$,\text{\quad if}\;$\frac{\tau}{1+k'(1-p')}\leq\eta_{1,14}^*\leq\tau,\;0\leq \eta_{2,14}^*<\frac{1}{1+(1-p)k}$,\\
\text{\qquad\qquad}\;$\left(\xi_{1C}-\frac{\alpha_{1C}}{k'\tau}\right)\eta_{1,14}^*+\xi_{0C}\eta_{2,14}^*
=\theta_{1C}+\left(1-\frac{1}{k'}\right){\alpha_{1C}},\;
\left(\xi_{1C}-\frac{\alpha_{1C}}{k'\tau}\right)(\theta_{2C}+(2-p){\alpha_{2C}})=$\\
\text{\qquad\qquad}\;$\xi_{0C}\left(\theta_{1C}+\left(1-\frac{1}{k'}\right)\alpha_{1C}\right),\;
\det\left(\xi_C-\frac{\alpha_{1C}}{k'\tau}
\mathbf{e}_{11}+(1-p)(2+k(1-p))\alpha_{2C}\mathbf{e}_{22}\right)=0$;}
%region 15
  \item [] {$\mathrm{Region\; 15}:\;
  \eta_{15}^*=\left(\frac{\theta_{1C}+(1-\frac{1}{k'})\alpha_{1C}-\xi_{0C}}
{\xi_{1C}-\frac{\alpha_{1C}}{k'\tau}},1\right)^\top$,\text{\quad if}\;$\frac{\tau}{1+(1-p')k'}\leq\eta_{1,15}^*\leq \tau,\;\eta_{1,15}^*<\frac{\theta_{2C}+{\alpha_{2C}}-\xi_{2C}}{\xi_{0C}}$,\\
\text{\qquad\qquad}\;$\xi_{1C}-
\frac{\alpha_{1C}}{k'\tau}\neq 0$;\\
%region 15a
   \text{\quad or}\quad$\eta_{15}^*=(\eta_{1,15}^*,1)^\top$,\text{\quad if}\;$\frac{\tau}{1+(1-p')k'}\leq\eta_{1,15}^*\leq \tau,\;\eta_{1,15}^*<\frac{\theta_{2C}+{\alpha_{2C}}-\xi_{2C}}{\xi_{0C}},\xi_{1C}-
\frac{\alpha_{1C}}{k'\tau}= 0$,\\
\text{\qquad\qquad}\;$\theta_{1C}+(1-\frac{1}{k'})\alpha_{1C}-\tau\xi_{0C}=0$;}
%region 16
  \item [] {$\mathrm{Region\; 16}:\;
  \eta_{16}^*=\left(\frac{\theta_{1C}+(1-\frac{1}{k'})\alpha_{1C}}
{\xi_{1C}-\frac{\alpha_{1C}}{k'\tau}},0\right)^\top$,\text{\quad if}\;$\frac{\tau}{1+(1-p')k'}\leq\eta_{1,16}^*\leq \tau,\;\eta_{1,16}^*>\frac{\theta_{2C}+(2-p){\alpha_{2C}}}{\xi_{0C}}$,\\
\text{\qquad\qquad}\;$\xi_{1C}-
\frac{\alpha_{1C}}{k'\tau}\neq 0$;\\
%region 16a
   \text{\quad or}\quad$\eta_{16}^*=(\eta_{1,16}^*,0)^\top$,\text{\quad if}\;$\frac{\tau}{1+(1-p')k'}\leq\eta_{1,16}^*\leq \tau,\;\eta_{1,16}^*>\frac{\theta_{2C}+(2-p){\alpha_{2C}}}{\xi_{0C}},\xi_{1C}-
\frac{\alpha_{1C}}{k'\tau}= 0$,\\
\text{\qquad\qquad}\;$\theta_{1C}+(1-\frac{1}{k'})\alpha_{1C}=0$;}
\end{description}
and in Region $v$ we have $J'_C(x^*_v,\eta_{v}^*)=-\infty$ if the related conditions do not hold, and $(J'_{C})^*$ is defined by $(J'_{C})^*=\max\{J'_{C}(x^*_v,\eta_{v}^*):v=1,\cdots,16\}$.
}\end{proposition}

In the cross-sectional economy with investor protection in both firms, we adopt the same equilibrium in Definition \ref{equilibrium} and then derive parameters of asset prices.

\begin{proposition}\label{th10}{\it
Provide that there exists investor protection in both firms with $0\leq p'<1, 0\leq p<1$ in the cross-sectional economy. Then under the equilibrium in Definition \ref{equilibrium} and the condition $y\in (0,1)$, the shareholders' optimal consumptions $c_{i}^*$ ($i=C,M$) are given by (\ref{ps_c}), and the stock growth rate $\mu$ and the interest rate $r$ are given by:
\begin{align}
%\mu
\mu=&r(\mathbf{e}_1+\mathbf{e}_2)+\frac{\frac{\tau y}{y_2}\mathbf{e}_{11}+\frac{y}{1-y_2}\mathbf{e}_{22}}
{\rho^{\frac{1}{\gamma_M}}\Gamma_0}[\xi_M(\tau\mathbf{e}_1+\mathbf{e}_2-\mathbf{n}_C^*)
-(1-(x')^*)\alpha_{1M}\mathbf{e}_1-(1-x^*)\alpha_{2M}\mathbf{e}_2];\label{uu}\\
%r
r=&y_1\mu_{1D}+(1-y_1)\mu_{2D}-\frac{\Gamma_1}{\tau}\left[y_2\Gamma_0(\mu_1-r)
+(1-(x')^*)(1-l_{1M}-l_{1C})y_1\right]\nonumber\\
&-\Gamma_2[(1-y_2)\Gamma_0(\mu_2-r)+(1-x^*)(1-l_{2M}-l_{2C})(1-y_1)]\nonumber\\
&-y_1\left(\rho^{\frac{1}{\gamma_C}}l_{1C}+\rho^{\frac{1}{\gamma_M}}l_{1M}\right)
-(1-y_1)\left(\rho^{\frac{1}{\gamma_C}}l_{2C}+\rho^{\frac{1}{\gamma_M}}l_{2M}\right)
+\rho^{\frac{1}{\gamma_C}}(1-y)+\rho^{\frac{1}{\gamma_M}}y\nonumber\\
&-(1-l_{1M}-l_{1C})y_1\left[\rho^{\frac{1}{\gamma_C}}\left((x')^*-\frac{k'}{2}(x'^*)^2\right)
+\rho^{\frac{1}{\gamma_M}}\frac{k'}{2}(x'^*)^2\right]\nonumber\\
&-(1-l_{2M}-l_{2C})(1-y_1)\left[\rho^{\frac{1}{\gamma_C}}\left(x^*-\frac{k}{2}(x^*)^2\right)
+\rho^{\frac{1}{\gamma_M}}\frac{k}{2}(x^*)^2\right];\label{rr}
\end{align}
and the parameter $\mu_y$ in (\ref{dy}) are given by
\begin{align*}
\mu_y=&yr-yy_1\mu_{1D}-y(1-y_1)\mu_{2D}-\sigma_y\sigma_D-(1-y_1)\delta_D\delta_y+\rho^{\frac{1}{\gamma_M}}
[l_{1M}y_1+l_{2M}(1-y_1)-y]\\
&+\rho^{\frac{1}{\gamma_M}}\left[\frac{k'}{2}(x'^*)^2(1-l_{1M}-l_{1C})y_1
+\frac{k}{2}(x^*)^2(1-l_{2M}-l_{2C})(1-y_1)\right]\\
&+\rho^{\frac{1}{\gamma_M}}\frac{\tau-n_{1C}^*}{\tau}
\left[{y_2}\Gamma_0(\mu_1-r)+(1-(x')^*)(1-l_{1M}-l_{1C})y_1\right]\\
&+\rho^{\frac{1}{\gamma_M}}(1-n_{2C}^*)[(1-y_2)\Gamma_0(\mu_2-r)+(1-x^*)(1-l_{2M}-l_{2C})(1-y_1)];
\end{align*}
and the stock volatilities $\sigma$ and $\delta$, parameters $\sigma_y$ and $\delta_y$ in (\ref{dy}), parameters $\mu_{1y}$, $\sigma_{1y}$ and $\delta_{1y}$ in (\ref{dy1}), and parameters $\mu_{2y},\sigma_{2y}$ and $\delta_{2y}$ in (\ref{dy2}) are respectively given by the same expressions of (\ref{sigma}), (\ref{delta}), (\ref{y}), (\ref{y1}) and (\ref{y2}) in Proposition \ref{th4}. The minority shareholder's optimal stock holding $\mathbf{n}_{M}^*$ is given by
\begin{equation*}
n_{1M}^*=\tau-n_{1C}^*,\quad n_{2M}^*=1-n_{2C}^*,
\end{equation*}
and the controlling shareholder's optimal stock holding $\mathbf{n}_{C}^*$ can be obtained by solving fixed-point equation
\begin{align*}
\mathbf{n}_C^*=&\mathop{\mathrm{argmax}}\limits_{\mathbf{n}_C\in[0,\tau]\times[0,1]}
\left\{\sum_{j=1}^2n_{jC}(\mu_{j}-r)\frac{S_j}{X_C}
+{n_{1C}}(1-x'^*(n_{1C}))\frac{D_1}{\tau X_C}+n_{2C}(1-x^*(n_{2C}))\frac{D_2}{X_C}\right.\\
&+\left(x'^*(n_{1C})-\frac{k'x'^*(n_{1C})^2}{2}\right)\frac{D_1}{X_C}
+\left(x^*(n_{2C})-\frac{kx^*(n_{2C})^2}{2}\right)\frac{D_2}{X_C}\\
&\left.-\frac{\gamma_C}{2}\left[ \left(\frac{n_{1C}S_1\sigma}{X_C}\right)^2+\frac{2n_{1C}n_{2C}S_1S_2\sigma^2}{X_C^2} +\left(\frac{n_{2C}S_2}{X_C}\right)^2(\sigma^2+\delta^2)\right]\right\}
\end{align*}
with ratios (\ref{ratio}), $x'^*(n_{2C})=\min\left\{\frac{1-n_{1C}/\tau}{k'},(1-p')\frac{n_{1C}}{\tau}\right\}$, $x^*(n_{2C})=\min\left\{\frac{1-n_{2C}}{k},(1-p)n_{2C}\right\}$ and parameters ${\mu},{r},{\sigma},{\delta}$ given by (\ref{uu}), (\ref{rr}), (\ref{sigma}) and (\ref{delta}).
}\end{proposition}

The major difference between Proposition \ref{th10} and Proposition \ref{th4} is investor protection in firm 1 instead of cross-section in the economy. Hence, in the following part of this section, we discuss, by numerical methods, the effect of investor protection in the cross-sectional economy with investor protection in both firms. Figures \ref{fig_k6} and \ref{fig_k15} demonstrate the effect of investor protection of both firms $1$ and $2$ in the cross-sectional economy where parameters related to investor protection are same in each figure, i.e., $p'=p$ and $k'=k$. To be precise, we further explain Figures \ref{fig_k6} and \ref{fig_k15} as follows:
in Figure \ref{fig_k6}, all parameters are obtained by substituting each $y$ and $\tau=1,y_1=1,y_2=0.5,k'=k=6,\delta_D=10\%$, and parameters in Table \ref{table} into Proposition \ref{th3} in the case of $p'=p=1$ and Proposition \ref{th10} in the cases of $p'=p=0.9$ and $p'=p=0.6$, respectively;
in Figure \ref{fig_k15}, all parameters are obtained by substituting each $y$ and $\tau=1,y_1=1,y_2=0.5,k'=k=15,\delta_D=10\%$, and parameters in Table \ref{table} except for $k$ into Proposition \ref{th3} in the case of $p'=p=1$ and Proposition \ref{th10} in the cases of $p'=p=0.9$ and $p'=p=0.6$, respectively.

\begin{figure}[htb]
\centering
\subfigure[Stock holdings in firm $1$]{
\includegraphics[ width = 0.31\textwidth]{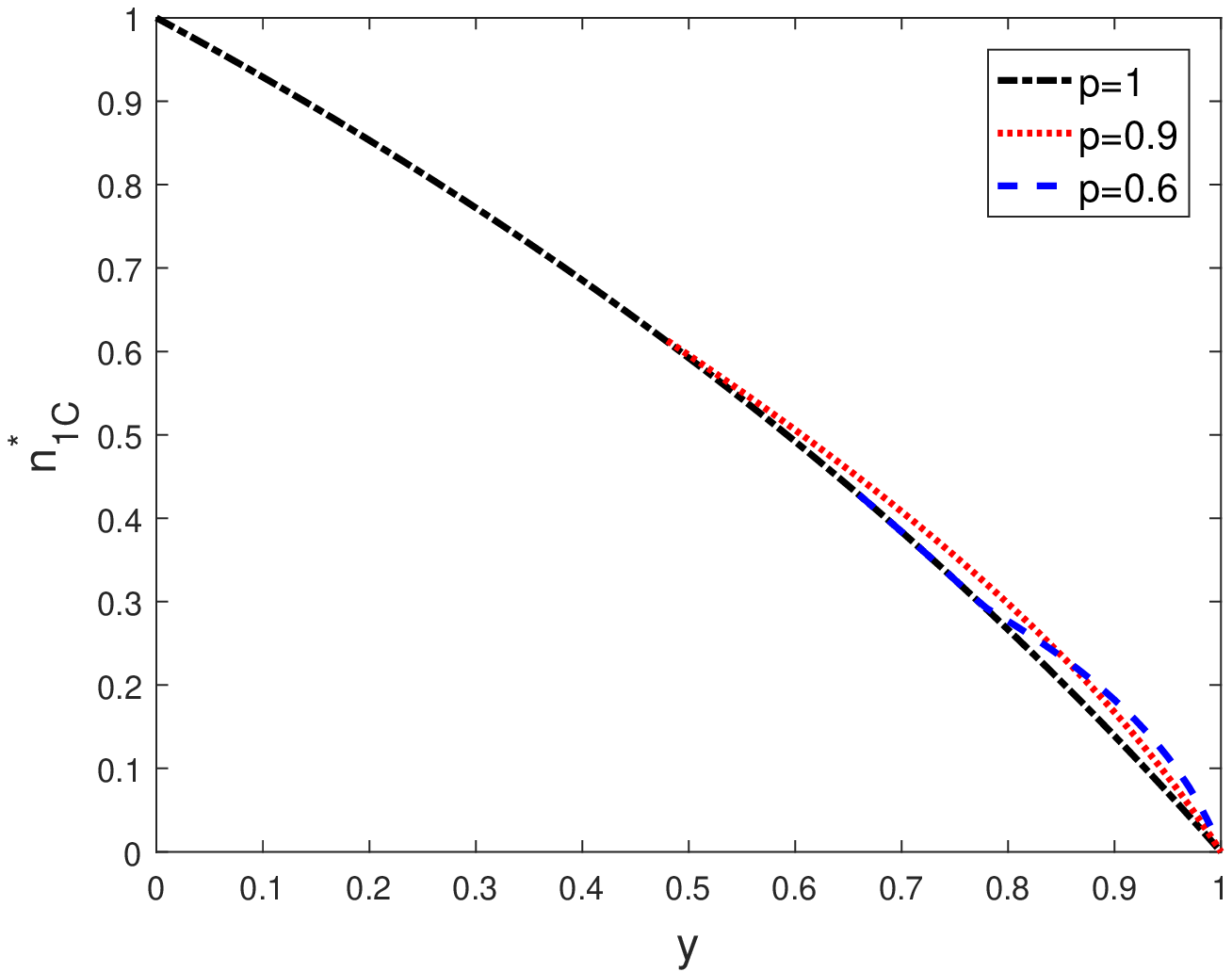}
}
\subfigure[Stock holdings in firm $2$]{
\includegraphics[ width = 0.31\textwidth]{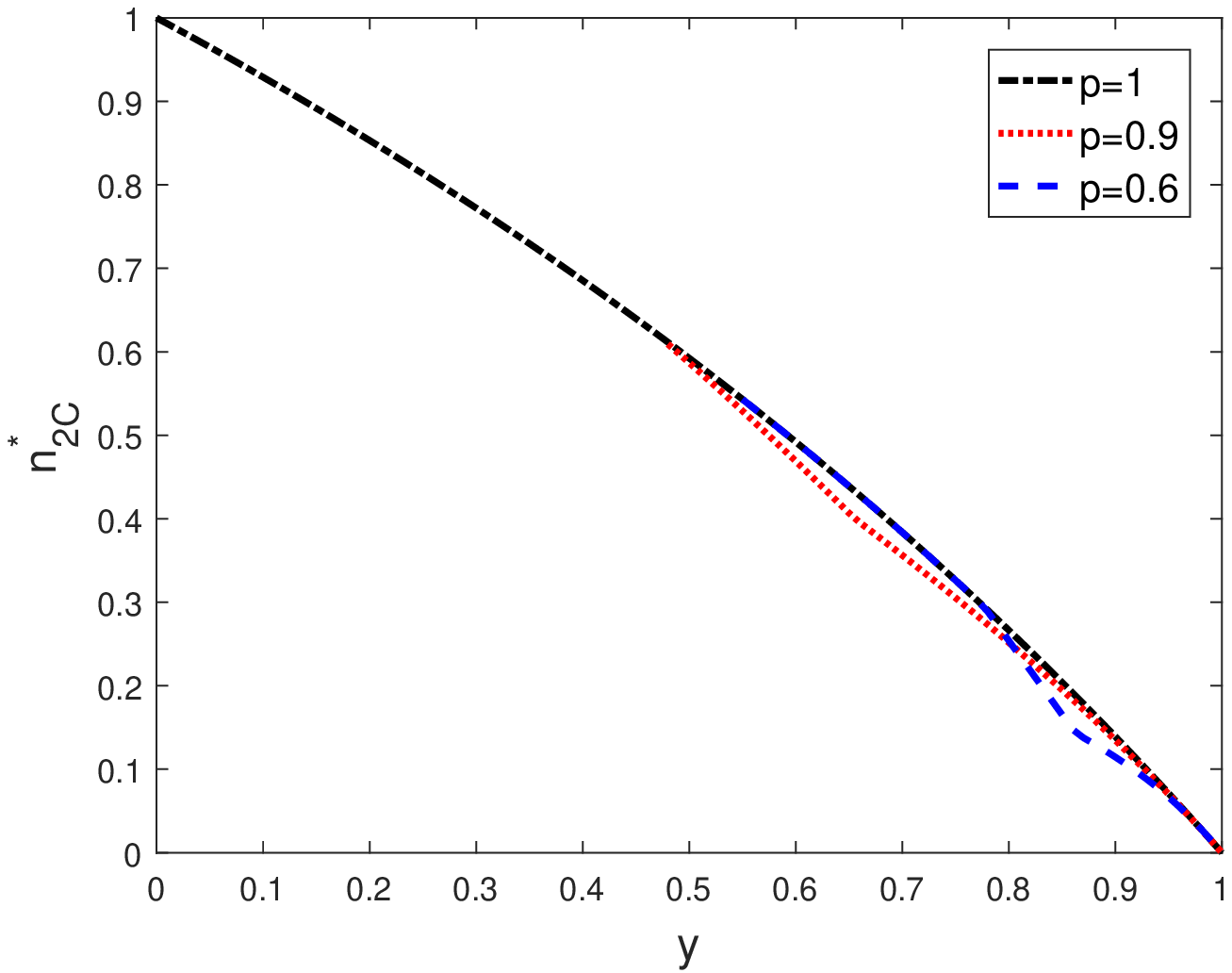}
}
\caption{ The effect of investor protection in both firms ($k=6$)}
\label{fig_k6}
\end{figure}

\begin{figure}[htb]
\centering
\subfigure[Stock holdings in firm $1$]{
\includegraphics[ width = 0.31\textwidth]{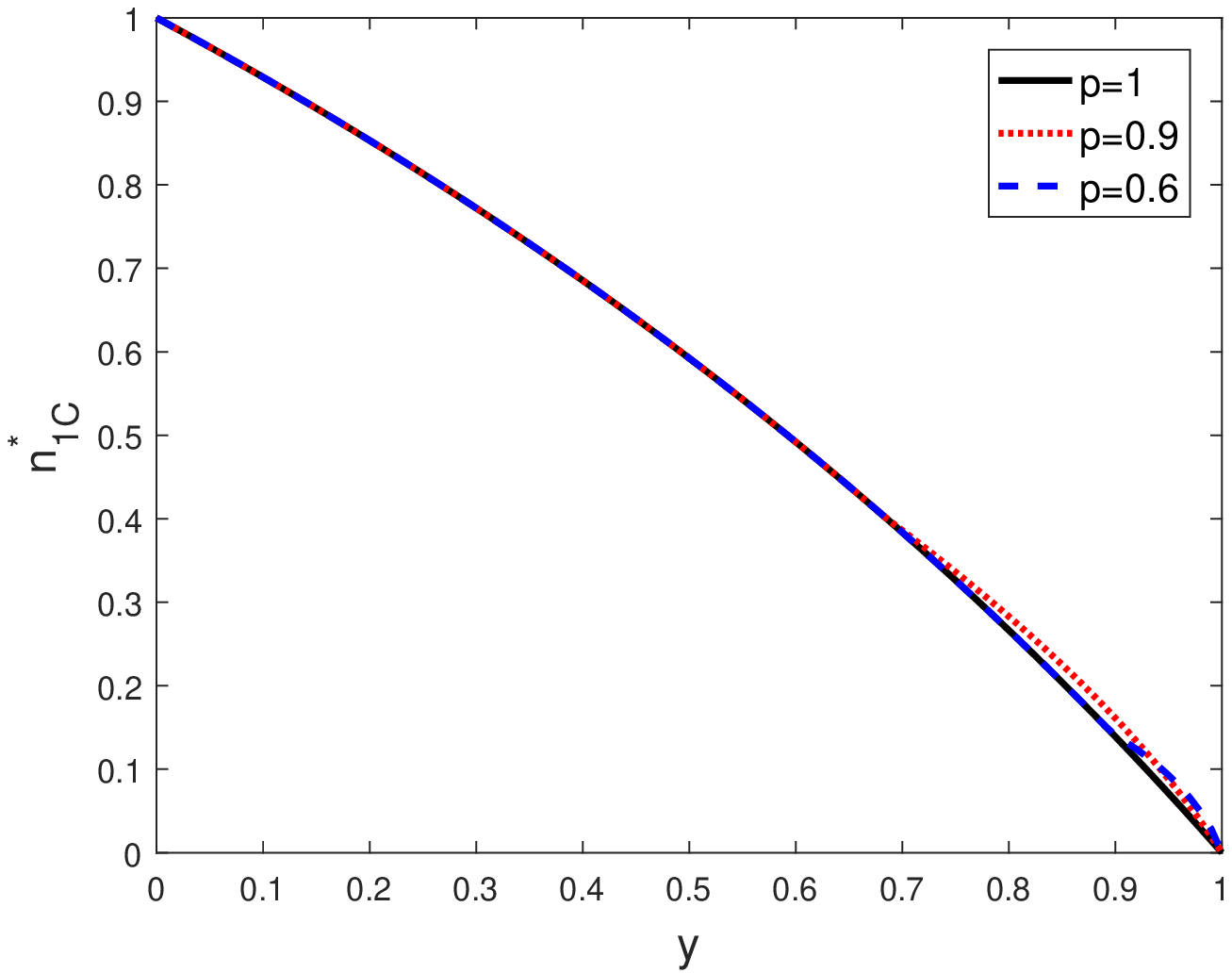}
}
\subfigure[Share difference in firm $1$]{
\includegraphics[ width = 0.31\textwidth]{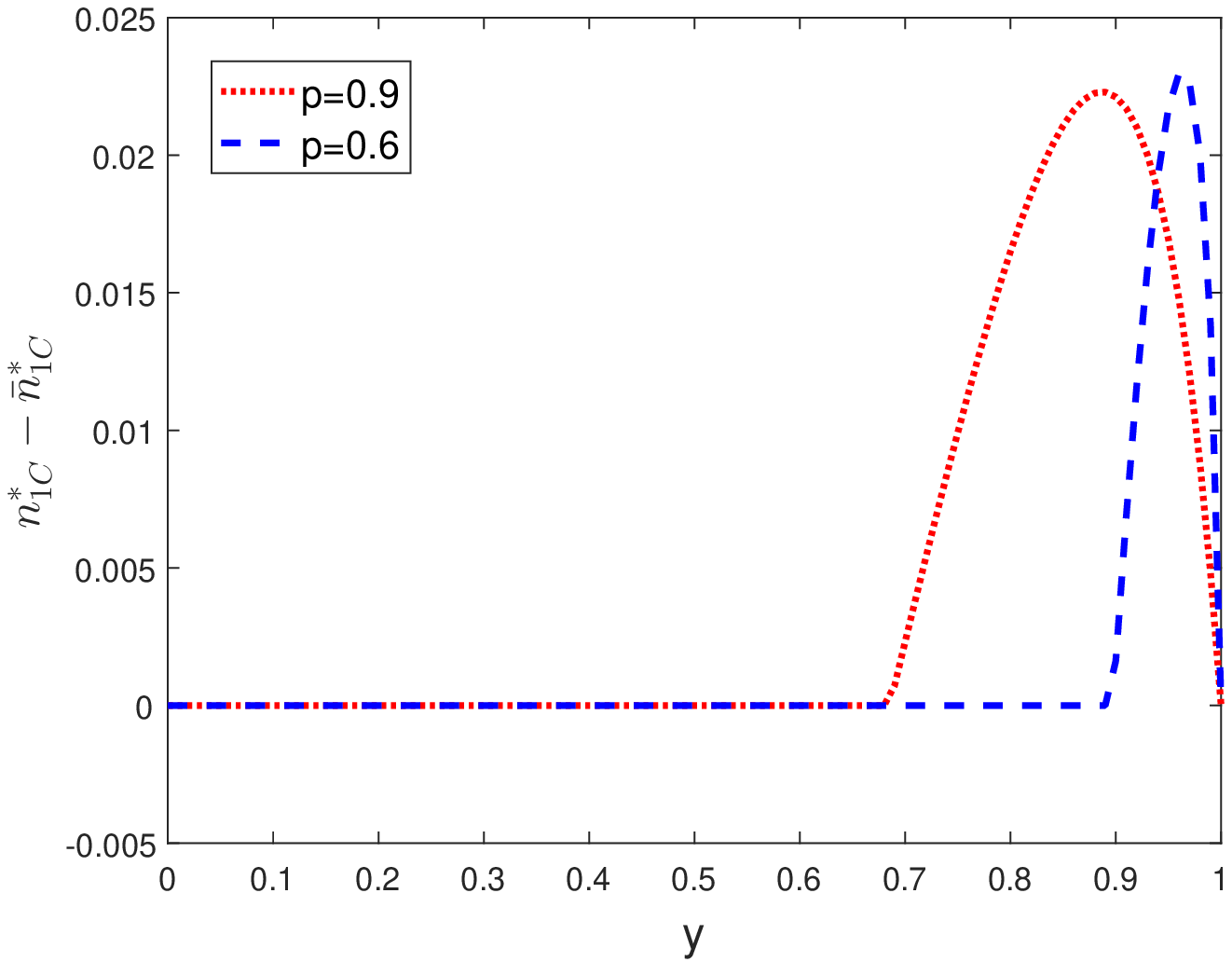}
}
\subfigure[Fractions of diverted output of firm 1]{
\includegraphics[ width = 0.31\textwidth]{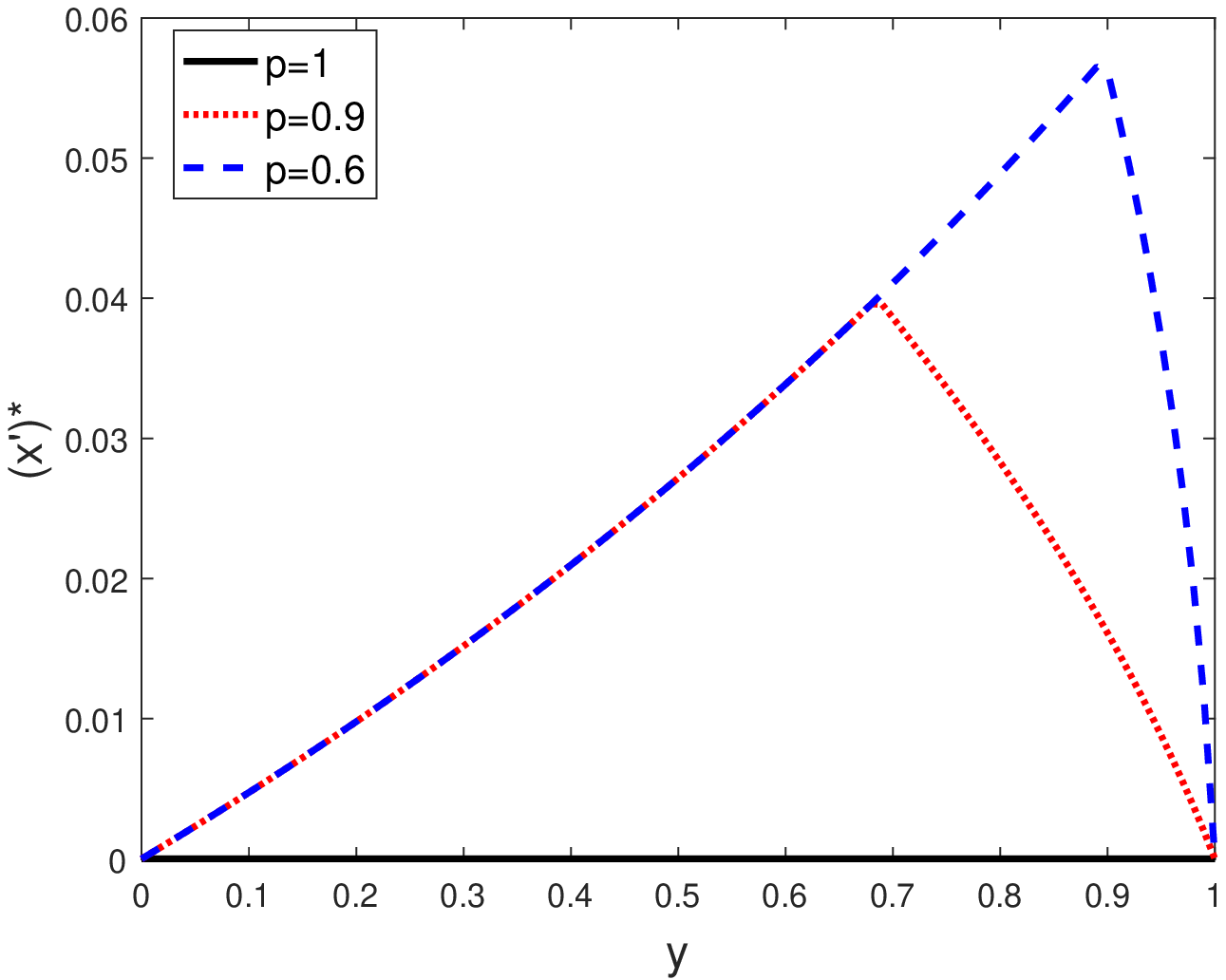}
}
\subfigure[Stock holdings in firm $2$]{
\includegraphics[ width = 0.31\textwidth]{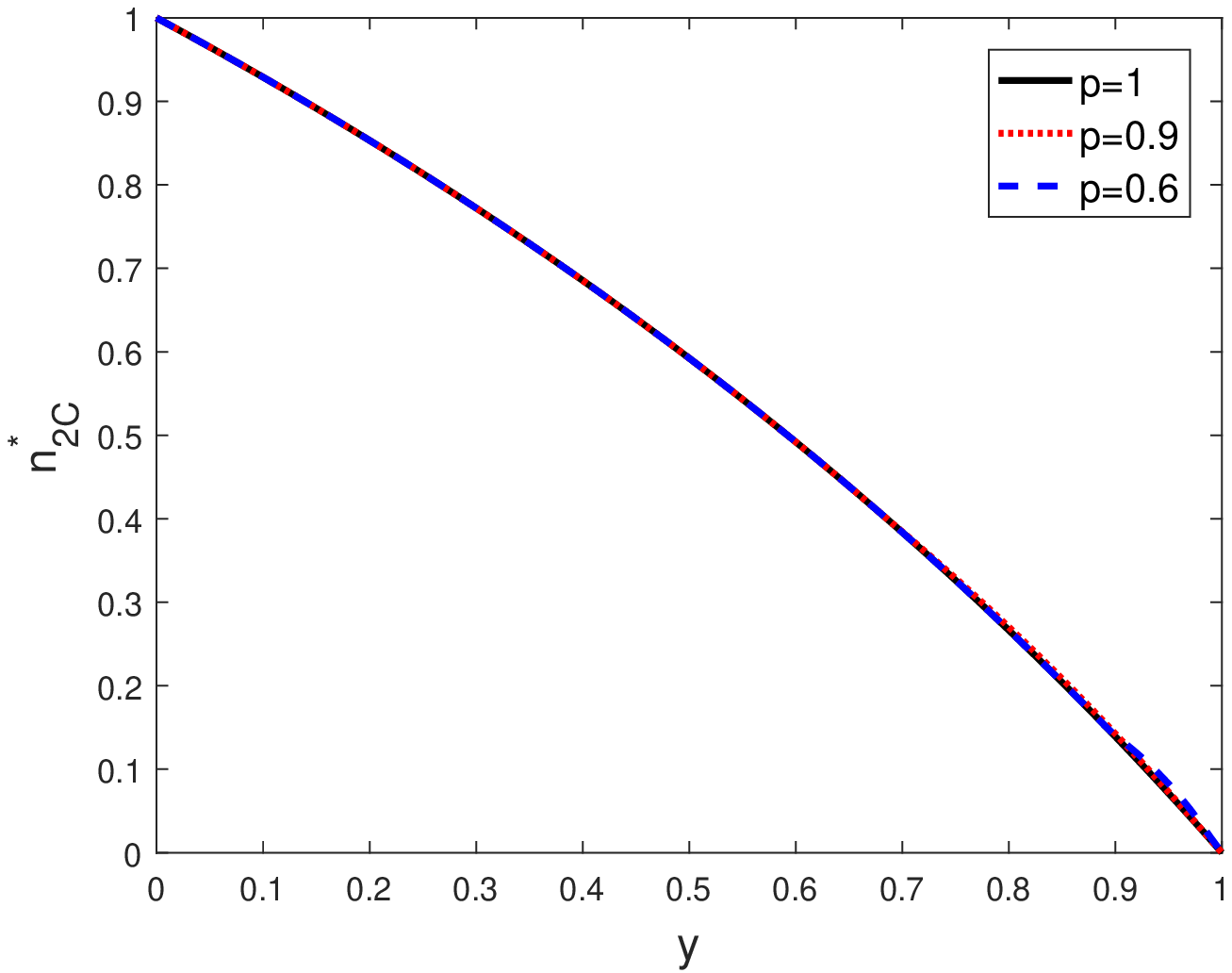}
}
\subfigure[Share difference in firm $2$]{
\includegraphics[ width = 0.31\textwidth]{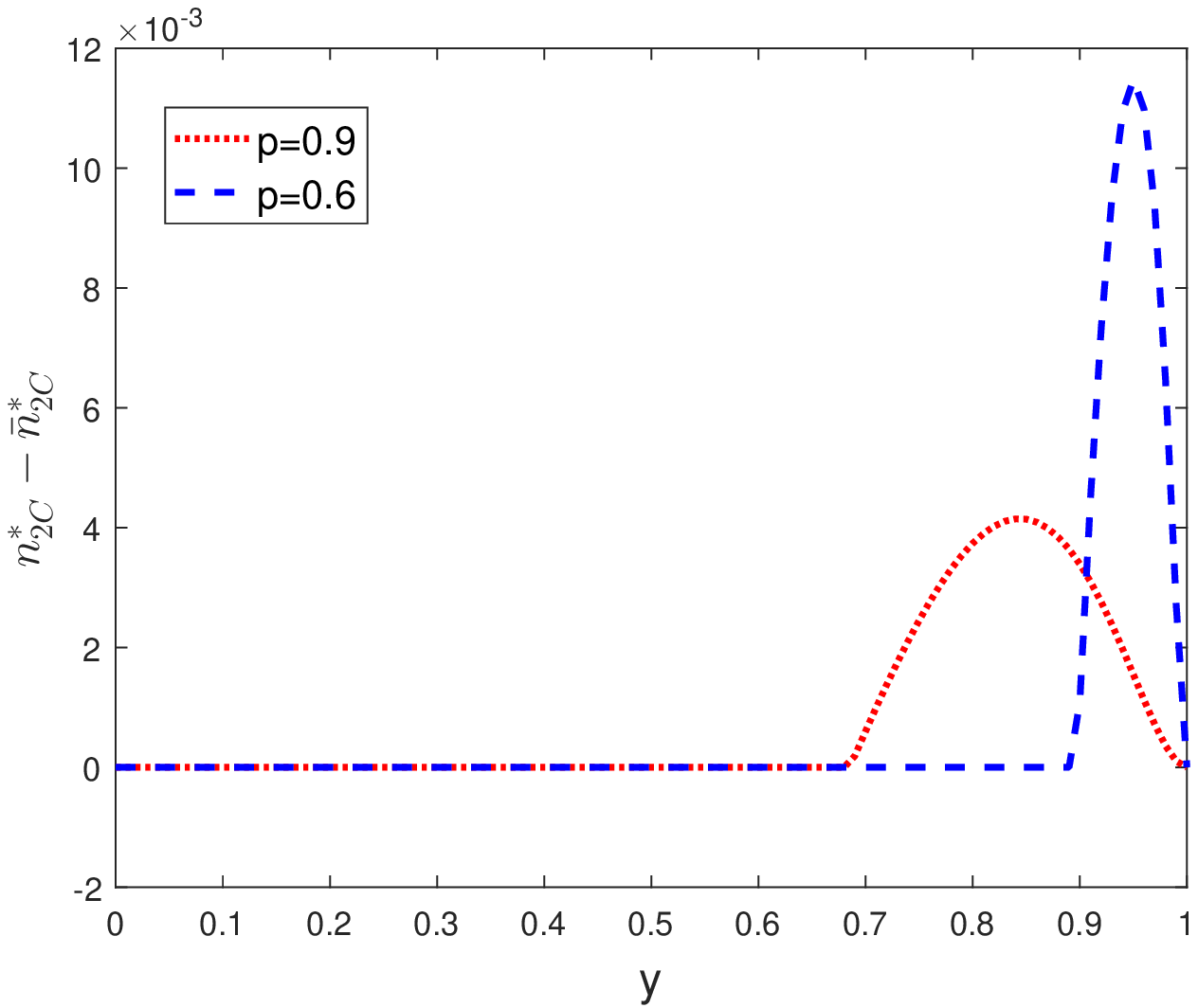}
}
\subfigure[Fractions of diverted output of firm $2$]{
\includegraphics[ width = 0.31\textwidth]{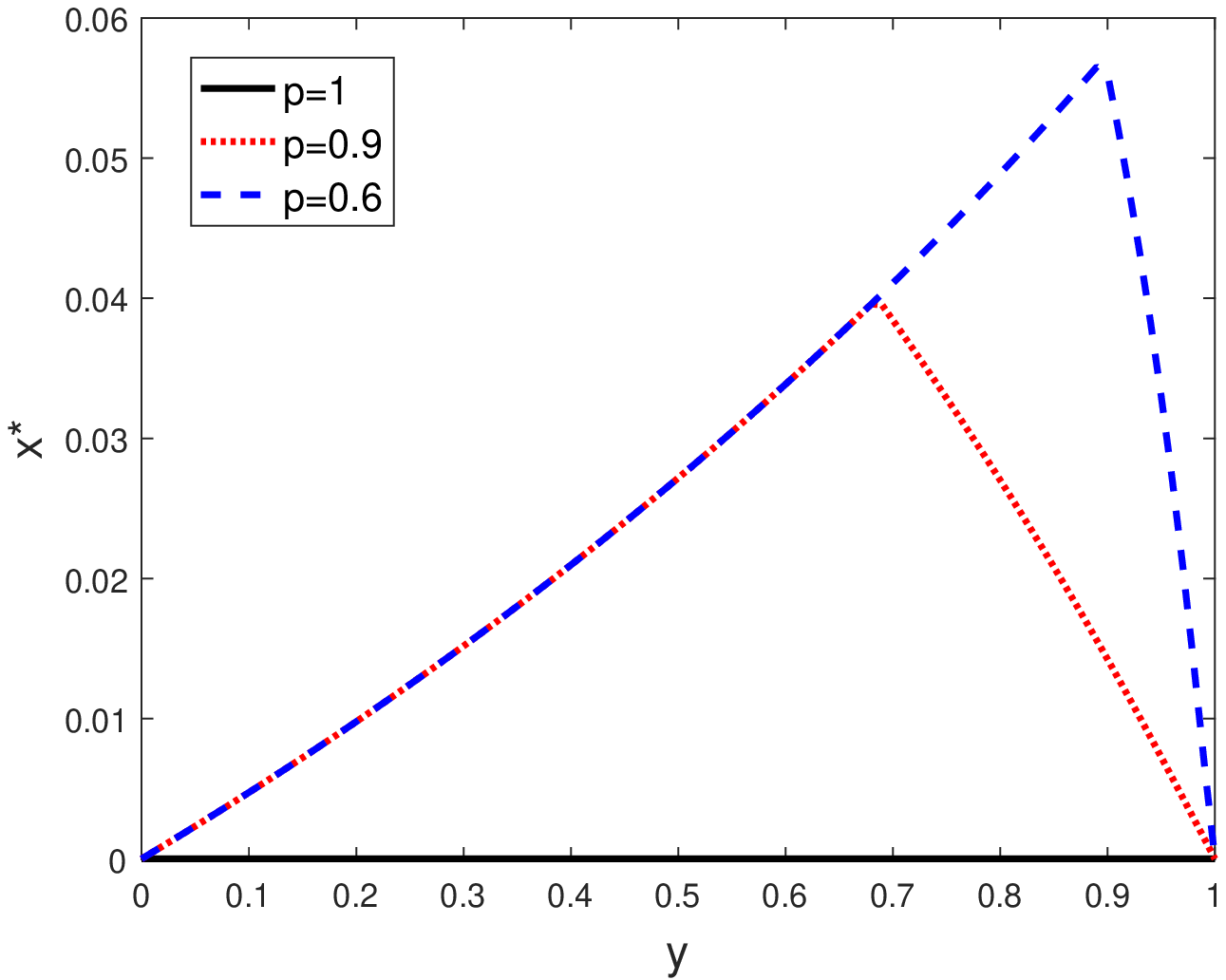}
}
\caption{ The effect of investor protection in both firms ($k=15$)}
\label{fig_k15}
\end{figure}

Figure \ref{fig_k6} displays a different economy where the equilibrium with imperfect investor protection could not be achieved for higher consumption share of the controlling shareholder (i.e., lower $y$). There are two main reasons: the parameters of stealing costs (i.e., $k'$ and $k$) which are too low to temper the diversion of output, and sufficiently high covariance risk which is cased by the cross-sectional economy. Sufficiently high consumption share of the controlling shareholder requires him to hold sufficiently high stock shares in both firms such that neither the investor protection constraint (\ref{protectC1}) nor (\ref{protectC}) binds, and hence stealing costs act crucially in determining the equilibrium. However,   sufficiently high stock shares could have adverse effects on the controlling shareholder's benefit. For one thing, lower parameters of stealing costs encourage the controlling shareholder to divert more output and this stops him from holding sufficiently high stock shares in both firms ($(x')^*=\frac{1-n_{1C}^*/\tau}{k'},x^*=\frac{1-n_{2C}^*}{k}$). For another, sufficiently high stock shares in both firms force the controlling shareholder to face excessive covariance risk, which discourages him to hold so many stock shares. Hence, the contradiction between low stealing costs and high covariance risk leads to the failure of equilibrium. Such contradiction can be avoided through raising the parameters of stealing costs, which is confirmed in Figure \ref{fig_k15} with $k'=k=15$.

Figure \ref{fig_k15} reveals the effect of investor protection in the cross-sectional economy where the equilibrium is achieved for imperfect protection in both firms. Panels (a) and (d) (or equivalently, panels (b) and (e)) show that comparing to the economy with perfect protection, the controlling shareholder requires more shares of both firms in the economy with imperfect protection. This is consistent with panel (c) of Figure \ref{fig_sx} and the result with single firm in \cite{Basak}, and the main reason is similar. However, as is showed in Figure \ref{fig_k6}, compared with the economy with imperfect protection only in firm $2$ (or the economy with a single firm), the equilibrium is more difficult to be achieved in the economy with imperfect protection in both firms. Furthermore, since the stock volatilities in firm $2$ are always higher than those in firm $1$, the controlling shareholder prefers holding more shares in firm $1$. Panels (c) and (f) verify that better invest protection decrease the fractions of diverted output in both firms, and the same reason has been analyzed in panel (e) of Figure \ref{fig_sx}.

\section{Conclusions}\label{section6}\noindent

In this paper, we establish a dynamic asset pricing model for a cross-sectional economy with two firms: one firm is with perfect protection and the controlling shareholder could not divert the output while the other firm is with different levels of protection and the controlling shareholder may divert the output. Via Karush-Kuhn-Tucker conditions and equilibrium market conditions, not only are shareholders' optimal consumption-portfolio choices and parameters of asset prices derived, but related survival analysis of shareholders in the cross-sectional economy is also considered. In numerical analysis, the results are mainly twofold. For one thing, better investor protection indeed protects minority shareholders in the cross-sectional economy and leads to higher stock gross returns, lower stock volatilities and higher interest rates. For another, cross-section brings complex competition to the economy, and it is negatively associated with the growth of the firm stock with perfect protection and positively correlated with the growth of the firm stock with imperfect protection. However, some problems remain unanswered. For one thing, since there always exists one solution to the related fixed-point equation in each region of \cite{Basak}, is it possible to give some sufficient conditions on the existence of related fixed-point equations in Subsection \ref{subsection3.2}? For another, the sufficient conditions on extinction of shareholders in Subsection \ref{subsection3.3} are in the form of short run, so survival analysis in the form of long run is still unsolved. Those problems could be our future work.

\section*{Appendices}
\setcounter{equation}{0}
\setcounter{subsection}{0}
\renewcommand{\theequation}{A.\arabic{equation}}
\renewcommand{\thelemma}{A.\arabic{lemma}}

\noindent{\bf Proof of Proposition \ref{th2}.}
In order to prove Proposition \ref{th2}, we first give two lemmas.
\begin{lemma}\label{lemma1}{\it
The global optimal solution $(x^*,\mathbf{n}_C^*)$ to (\ref{problemCn}) (or equivalently, (\ref{pC})) exists. The optimal fraction of diverted output $x^*$ is give by (\ref{xx}), and the optimal stock holding $\mathbf{n}_C^*$ is given by (\ref{ps_nC}).
}\end{lemma}
\begin{proof}
Since the function $J_C$ is continuous on the compact set $\{(x,n_1,n_2)\in \mathbb{R}^3: 0\leq x\leq (1-p)n_2,0\leq n_1\leq \tau,n_2\leq 1\}$, there exists at least one solution to (\ref{pC}). In the following proof, we obtain all possible solutions through Karush-Kuhn-Tucker conditions, and then by searching over all possible solutions, we can obtain the optimal solution $(x^*,\mathbf{n}_C^*)$ to (\ref{pC}).

Define the associated Lagrangian function $L_0:\mathbb{R}^7\rightarrow[-\infty,+\infty)$ as
\[
L_0(x,\mathbf{n}_{C},\lambda)=\left\{\begin{aligned}
&-J_C(x,\mathbf{n}_{C})+\sum_{j=0}^4\lambda_jf_j(x,n_{1C},n_{2C}),&&\lambda_j\geq 0;\\
&-\infty,&&otherwise.
\end{aligned}
\right.
\]
Note that constraints in (\ref{pC}) satisfy the Slater constraint qualification, i.e.,
\begin{equation*}
\text{there exists}\; (x,n_{1C},n_{2C})^\top=\left(\frac{1-p}{4},\frac{\tau}{2},\frac{1}{2}\right)^\top \;\text{with}\; f_j<0\; \text{for all}\; j=0,1,\cdots,4.
\end{equation*}
By Karush-Kuhn-Tucker conditions (see, for example, Proposition 3.2.3 and Theorem 3.2.8 in \cite{Borwein}), if $(x^*,\mathbf{n}_{C}^*)$ is the global optimal solution to (\ref{pC}), then there must be a Lagrange multiplier vector $\lambda^*$ with $\lambda^*=(\lambda_0^*,\lambda_1^*,\cdots,\lambda^*_4)^\top$ such that $(x^*,\mathbf{n}_{C}^*,\lambda^*)$ satisfies
\begin{align}
&\frac{\partial}{\partial x}L_0=-(1-n_{2C}^*-kx^*)\alpha_{2C}-\lambda_0^*+\lambda_1^*=0;\label{KKTx}\\
&\frac{\partial}{\partial n_{1C}}L_0=-\theta_{1C}+n_{1C}^*\xi_{1C}+n_{2C}^*\xi_{0C}-{\alpha_{1C}}-\lambda_2^*+\lambda_3^*=0;\label{KKT1}\\
&\frac{\partial}{\partial n_{2C}}L_0=-\theta_{2C}+n_{2C}^*\xi_{2C}+n_{1C}^*\xi_{0C}-(1-x^*)\alpha_{2C}-\lambda_1^*(1-p)+\lambda_4^*=0;\label{KKT2}\\
&\lambda_j^*\geq0,\;f_j(x^*,n_{1C}^*,n_{2C}^*)\leq 0,\;\lambda_j^*f_j(x^*,n_{1C}^*,n_{2C}^*)=0,\; j=0,1,2,3,4.\label{KKT0}
\end{align}
In fact, whether $\lambda_j^*,j=0,1,\cdots,4$ are equal to zero divides the possible solutions to (\ref{KKTx})-(\ref{KKT0}) into following cases:

\begin{description}
  \item[Region 1:] $\lambda_0^*=0,\;\lambda_1^*=[1-\eta_{2,1}^*(1+k(1-p))]\alpha_{2C},\;\lambda_2^*=0,\;\lambda_3^*=0,\;
      \lambda_4^*=0,\;x_1^*=(1-p)\eta_{2,1}^*,\;
      \eta_1^*=\left(\xi_C+(1-p)(2+k(1-p))\alpha_{2C}\mathbf{e}_{22}\right)^{-1}\cdot\left[\theta_C
+{\alpha_{1C}}\mathbf{e}_1+\left(2-p\right)\alpha_{2C}\mathbf{e}_{2}\right]$ with conditions $\lambda_1^*>0,\;0\leq\eta_{1,1}^*\leq \tau,\; 0\leq\eta_{2,1}^*\leq 1$;
  \item[Region 2:]
      $\lambda_0^*=0,\;\lambda_1^*=[1-\eta_{2,2}^*(1+k(1-p))]\alpha_{2C},
      \;\lambda_2^*=\eta_{2,2}^*\xi_{0C}-\theta_{1C}-\alpha_{1C},
      \;\lambda_3^*=0,\;\lambda_4^*=0,\;x_2^*=(1-p)\eta_{2,2}^*,\;\eta_2^*=\left(0,\frac{\theta_{2C}+(2-p)\alpha_{2C}}{\xi_{2C}+(1-p)[2+(1-p)k]\alpha_{2C}}
      \right)^\top$ with conditions $\lambda_1^*>0,\;\lambda_2^*>0,\; 0\leq\eta_{2,2}^*\leq 1$;
  \item[Region 3:]
      $\lambda_0^*=0,\;\lambda_1^*=[1-\eta_{2,3}^*(1+k(1-p))]\alpha_{2C},
      \;\lambda_2^*=0,\;\lambda_3^*=\theta_{1C}+\alpha_{1C}-\eta_{2,3}^*\xi_{0C}-\tau \xi_{1C},
      \;\lambda_4^*=0,\;x_3^*=(1-p)\eta_{2,3}^*,\;\eta_3^*=\left(\tau,\frac{\theta_{2C}+(2-p)\alpha_{2C}-\tau\xi_{0C}}{\xi_{2C}+(1-p)[2+(1-p)k]\alpha_{2C}}
\right)^\top$ with conditions $\lambda_1^*>0,\;\lambda_3^*>0,\; 0\leq\eta_{2,2}^*\leq 1$;
\item[Region 4:]
      $\lambda_0^*=\lambda_1^*-\alpha_{2C},\;\lambda_1^*=\frac{\eta_{1,4}^*\xi_{0C}-\theta_{2C}-\alpha_{2C}}{1-p},
      \;\lambda_2^*=0,\;\lambda_3^*=0,
      \;\lambda_4^*=0,\;x_4^*=0,\;\eta_4^*=\left(\frac{\theta_{1C}+{\alpha_{1C}}}{\xi_{1C}},0\right)^\top$ with conditions $\lambda_0^*>0,\;\lambda_1^*>0,\; 0\leq\eta_{1,4}^*\leq \tau$;
\item[Region 5:]
      $\lambda_0^*=0,\;\lambda_1^*=0,
      \;\lambda_2^*=0,\;\lambda_3^*=0,
      \;\lambda_4^*=\theta_{2C}+\alpha_{2C}-\eta_{1,5}^*\xi_{0C}-\xi_{2C},\;x_5^*=0,
      \;\eta_5^*=\left(\frac{\theta_{1C}+{\alpha_{1C}}-\xi_{0C}}
{\xi_{1C}},1\right)^\top$ with conditions $\lambda_4^*>0,\; 0\leq\eta_{1,5}^*\leq \tau$;
\item[Region 6:]
      $\lambda_0^*=\lambda_1^*-\alpha_{2C},\;\lambda_1^*=-\frac{\theta_{2C}+\alpha_{2C}}{1-p},
      \;\lambda_2^*=-\theta_{1C}-\alpha_{1C},\;\lambda_3^*=0,
      \;\lambda_4^*=0,\;x_6^*=0,
      \;\eta_6^*=\left(0,0\right)^\top$ with conditions $\lambda_0^*>0,\; \lambda_1^*>0,\;\lambda_2^*>0$;
\item[Region 7:]
      $\lambda_0^*=0,\;\lambda_1^*=0,
      \;\lambda_2^*=\xi_{0C}-\theta_{1C}-\alpha_{1C},\;\lambda_3^*=0,
      \;\lambda_4^*=\theta_{2C}+\alpha_{2C}-\xi_{2C},\;x_7^*=0,
      \;\eta_7^*=\left(0,1\right)^\top$ with conditions $\lambda_2^*>0,\; \lambda_4^*>0$;
\item[Region 8:]
      $\lambda_0^*=\lambda_1^*-\alpha_{2C},\;\lambda_1^*=\frac{\tau\xi_{0C} -\theta_{2C}-\alpha_{2C}}{1-p},
      \;\lambda_2^*=0,\;\lambda_3^*=\theta_{1C}+\alpha_{1C}-\tau\xi_{1C},
      \;\lambda_4^*=0,\;x_8^*=0,
      \;\eta_8^*=\left(\tau,0\right)^\top$ with conditions $\lambda_0^*>0,\;\lambda_1^*>0,\; \lambda_3^*>0$;
\item[Region 9:]
      $\lambda_0^*=0,\;\lambda_1^*=0,
      \;\lambda_2^*=0,\;\lambda_3^*=\theta_{1C}+\alpha_{1C}-\tau\xi_{1C}-\xi_{0C},
      \;\lambda_4^*=\theta_{2C}+\alpha_{2C}-\tau\xi_{0C}-\xi_{2C},\;x_9^*=0,
      \;\eta_9^*=\left(\tau,1\right)^\top$ with conditions $\lambda_3^*>0,\;\lambda_4^*>0$;
\item[Region 10:]  in the case of $\det\left(\xi_C-\frac{\alpha_{2C}}{k}\mathbf{e}_{22}\right)\neq 0$,
      $\lambda_0^*=0,\;\lambda_1^*=0,
      \;\lambda_2^*=0,\;\lambda_3^*=0,
      \;\lambda_4^*=0,\;x_{10}^*=\frac{1-\eta_{2,10}^*}{k},
      \;\eta_{10}^*=\left(\xi_C-\frac{\alpha_{2C}}{k}\mathbf{e}_{22}\right)^{-1}\cdot\left[\theta_C
      +{\alpha_{1C}}\mathbf{e}_1+\left(1-\frac{1}{k}\right)\alpha_{2C}\mathbf{e}_{2}\right]$ with
      conditions $\eta_{2,10}^*\leq 1,\; 0\leq\eta_{2,10}^*\leq \tau,\; x_{10}^*\leq (1-p)\eta_{2,10}^*$;\\
      or  in the case of $\det\left(\xi_C-\frac{\alpha_{2C}}{k}\mathbf{e}_{22}\right)=0$,
      $\lambda_0^*=0,\;\lambda_1^*=0,
      \;\lambda_2^*=0,\;\lambda_3^*=0,
      \;\lambda_4^*=0,\;x_{10}^*=\frac{1-\eta_{2,10}^*}{k},
      \;\eta_{10}^*=(\eta_{1,10}^*,\eta_{2,10}^*)^\top$ with
      conditions $\eta_{2,10}^*\leq 1,\; 0\leq\eta_{2,10}^*\leq \tau,\; x_{10}^*\leq (1-p)\eta_{2,10}^*,\; \xi_{1C}\eta_{1,10}^*+\xi_{0C}\eta_{2,10}^*=\theta_{1C}+{\alpha_{1C}},\;
      \frac{\xi_{0C}(\theta_{1C}+{\alpha_{1C}})}{\xi_{1C}}
       =\theta_{2C}+\left(1-\frac{1}{k}\right)\alpha_{2C}$;
\item[Region 11:]  in the case of $\xi_{2C}-\frac{\alpha_{2C}}{k}\neq 0$,
      $\lambda_0^*=0,\;\lambda_1^*=0,
      \;\lambda_2^*=\eta_{2,11}^*\xi_{0C}-\theta_{1C}-\alpha_{1C},\;\lambda_3^*=0,
      \;\lambda_4^*=0,\;x_{11}^*=\frac{1-\eta_{2,11}^*}{k},
      \;\eta_{11}^*=\left(0,\frac{\theta_{2C}+(1-\frac{1}{k})\alpha_{2C}}{\xi_{2C}-
\frac{\alpha_{2C}}{k}}\right)^\top$ with
      conditions $\lambda_2^*>0,\;\eta_{2,11}^*\leq 1,\; x_{11}^*\leq (1-p)\eta_{2,11}^*$;\\
      or in the case of $\xi_{2C}-\frac{\alpha_{2C}}{k}= 0$,
      $\lambda_0^*=0,\;\lambda_1^*=0,
      \;\lambda_2^*=\eta_{2,11}^*\xi_{0C}-\theta_{1C}-\alpha_{1C},\;\lambda_3^*=0,
      \;\lambda_4^*=0,\;x_{11}^*=\frac{1-\eta_{2,11}^*}{k},
      \;\eta_{11}^*=(0,\eta_{2,11}^*)^\top$ with
      conditions $\lambda_2^*>0,\;\eta_{2,11}^*\leq 1,\; x_{11}^*\leq (1-p)\eta_{2,11}^*,\;\theta_{2C}+(1-\frac{1}{k})\alpha_{2C}=0$;
\item[Region 12:]  in the case of $\xi_{2C}-\frac{\alpha_{2C}}{k}\neq 0$,
      $\lambda_0^*=0,\;\lambda_1^*=0,
      \;\lambda_2^*=0,\;\lambda_3^*=\theta_{1C}+\alpha_{1C}-\tau\xi_{1C}-\eta_{2,12}^*\xi_{0C},
      \;\lambda_4^*=0,\;x_{12}^*=\frac{1-\eta_{2,12}^*}{k},
      \;\eta_{12}^*=\left(\tau,\frac{\theta_{2C}+(1-\frac{1}{k})\alpha_{2C}-\tau\xi_{0C}}
{\xi_{2C}-\frac{\alpha_{2C}}{k}}\right)^\top$ with
      conditions $\lambda_3^*>0,\;\eta_{2,12}^*\leq 1,\; x_{12}^*\leq (1-p)\eta_{2,12}^*$;\\
      in the case of $\xi_{2C}-\frac{\alpha_{2C}}{k}= 0$,
      $\lambda_0^*=0,\;\lambda_1^*=0,
      \;\lambda_2^*=0,\;\lambda_3^*=\theta_{1C}+\alpha_{1C}-\tau\xi_{1C}-\eta_{2,12}^*\xi_{0C},
      \;\lambda_4^*=0,\;x_{12}^*=\frac{1-\eta_{2,11}^*}{k},
      \;\eta_{12}^*=\left(\tau,\eta_{2,12}^*\right)^\top$ with
      conditions $\lambda_3^*>0,\;\eta_{2,12}^*\leq 1,\; x_{12}^*\leq (1-p)\eta_{2,12}^*,\;\theta_{2C}+(1-\frac{1}{k})\alpha_{2C}-\tau\xi_{0C}=0$;
\end{description}
where in Region $v$ for $v=1,2,\cdots, 12$, the solution $(x^*,\mathbf{n}_{C}^*)$ to (\ref{KKTx})-(\ref{KKT0}) is assumed to be $(x_v^*,\eta_{v}^*)$ with $\eta_{v}^*=(\eta_{1,v}^*,\eta_{2,v}^*)^\top$. It is not difficult to verify that, provided the region is given (i.e., whether $\lambda_i^*,i=0,1,\cdots,4$ are equal to zero has been determined),
there exists a solution to (\ref{KKTx})-(\ref{KKT0}) if and only if the equations and conditions in the related region are satisfied. Furthermore, with conditions in related region, we can check that $x^*=\min\left\{\frac{1-n_{2C}^*}{k},(1-p)n_{2C}^*\right\}$ always holds.

To complete the proof, it suffices to show that there are no solutions to (\ref{KKTx})-(\ref{KKT0}) in other cases of $\lambda$. If $\lambda_0^*>0,\lambda_1^*=0$, then $\lambda_0^*=-(1-n_{2C}^*)\alpha_{2C}\leq 0$ in (\ref{KKTx}) is a contradiction. Hence, no solutions to (\ref{KKTx})-(\ref{KKT0}) exist in the cases $\lambda^*\in\{\lambda:\lambda_0>0,\lambda_1=0,\lambda_j\geq0,j=2,3,4\}$. If $\lambda_2^*>0,\lambda_3^*>0$, then $n_{1C}^*=0,n_{1C}^*=\tau$ is a contradiction. Hence, no solutions to (\ref{KKTx})-(\ref{KKT0}) exist in the cases $\lambda^*\in\{\lambda:\lambda_2>0,\lambda_3>0,\lambda_j\geq0,j=0,1,4\}$. If $\lambda_0^*>0,\lambda_1^*>0,\lambda_4^*>0$, then $n_{2C}^*=0,n_{2C}^*=1$ is a contradiction.
If $\lambda_0^*=0,\lambda_1^*>0,\lambda_4^*>0$, then $\lambda_1^*=-k(1-p)\alpha_{2C}<0$ is a contradiction. Hence, no solutions to (\ref{KKTx})-(\ref{KKT0}) exist in the cases $\lambda^*\in\{\lambda:\lambda_1>0,\lambda_4>0,\lambda_j\geq0,j=0,2,3\}$.

Summarizing, the optimal fraction of diverted output $x^*$ is give by (\ref{xx}) and the optimal stock holding $\mathbf{n}_C^*$ is given by (\ref{ps_nC}) with Regions $1-12$ in Proposition \ref{th2}. And this completes the proof.
\hfill$\Box$
\end{proof}

\begin{lemma}\label{lemma2}{\it
The optimal stock holding $\mathbf{n}^*_M$ is given by (\ref{ps_nM}), i.e.,
\[
\mathbf{n}_M^*=\xi_M^{-1}\cdot\left[\theta_M+{\alpha_{1M}}\mathbf{e}_1+(1-x^*)\alpha_{2M}\mathbf{e}_2\right],
\]
where $x^*$ is given by (\ref{xx}).
}\end{lemma}
\begin{proof}
Notice that the function $-J_M$ is coercive on $\mathbb{R}^2$, i.e.,
\[
(n_1,n_2)\in \mathbb{R}^2, n_1^2+n_2^2\rightarrow+\infty\Rightarrow -J_M(n_1,n_2)\rightarrow+\infty,
\]
and then by Theorem 9.3-1 in \cite{Ciarlet}, this guarantees the existence of the solutions to (\ref{pM}).
From (\ref{JMM}), simple computation gives
\begin{equation*}
\bigtriangledown_{\mathbf{n}}(-J_M)=\left(\frac{\partial(-J_M)}{\partial n_{1M}},\frac{\partial(-J_M)}{\partial n_{2M}}\right)^\top=\xi_M\mathbf{n}_M-\left[\theta_M+{\alpha_{1M}}\mathbf{e}_1+(1-x^*)\alpha_{2M}\mathbf{e}_2\right],
\end{equation*}
and the Hesse matrix
\[
\bigtriangledown^2 (-J_M)=\left(\frac{\partial^2(-J_M)}{\partial n_{iM}\partial n_{jM}}\right)_{i,j=1,2}=\xi_M.
\]
Since $\bigtriangledown^2 (-J_M)$ is positive definite for each $(n_{1M},n_{2M})^\top\in \mathbb{R}^2$, the function $-J_M$ is strictly convex on $\mathbb{R}^2$ (see, for example, Theorem 3.3 of Chapter \uppercase\expandafter{\romannumeral3} in \cite{Berkovitz}). Then there exists at most one optimal solution to (\ref{pM}) (see, for example, Theorem 3.1 of Chapter \uppercase\expandafter{\romannumeral4} in \cite{Berkovitz}), and this implies that the optimal solution $\mathbf{n}^*_M$ to (\ref{pM}) exists uniquely.
Now making use of Karush-Kuhn-Tucker conditions, the unique optimal solution to the problem (\ref{pM}) is obtained by solving $\bigtriangledown_{\mathbf{n}}(-J_M)=0$ for $\mathbf{n}_M$ which is just (\ref{ps_nM}). \hfill$\Box$
\end{proof}

Now we turn to the proof of Proposition \ref{th2}. The optimal consumptions $c_{i}^* \;(i=C,M)$ are easily obtained from (\ref{problemCc}) and (\ref{problemMc}). The rest parts of proposition have been proved in Lemmas \ref{lemma1} and \ref{lemma2}. Furthermore, the Hesse matrix of $-J_{C}$ with respect to $(x,n_{1C},n_{2C})^\top$ is
\[
\bigtriangledown^2 (-J_{C})=\left(
\begin{aligned}
k\alpha_{2C}&&0\;&&\alpha_{2C}\\
0\;\;&&\xi_{1C}&&\xi_{0C}\\
\alpha_{2C}&&\xi_{0C}&&\xi_{2C}
\end{aligned}
\right).
\]
When the condition $\det(\xi_C)>\frac{\alpha_{2C}\xi_{1C}}{k}$ holds, the Hesse matrix $\bigtriangledown^2 (-J_{C})$ is positive definite for each $(x,n_{1C},n_{2C})\in[0,1]\times[0,1]\times [0,1]$. Then the optimal solution $\mathbf{n}_C^*$ exists uniquely (see, also, Theorem 3.1 of Chapter \uppercase\expandafter{\romannumeral4} in \cite{Berkovitz}) and is determined by the conditions in related region. We note here that in later cases of Regions $10,11$, and $12$, conditions therein always contradict  $\det(\xi_C)>\frac{\alpha_{2C}\xi_{1C}}{k}$.

%%%%%%%%%%%%%%%%%%%%%%%%%%%%%%%%%%%%%%%%%%%%%%%%%%%%%%%%%%%%%%%%%%%%%%%%%%%%%%%%%%%%%
\noindent{\bf Proof of Proposition \ref{th1}.}
Firstly, we prove that the optimal solution $(\overline{x},\overline{\mathbf{n}}_C^*)$ to (\ref{problemCn}) exists uniquely in the case of $p=1$. Since $\overline{x}^*=0$ holds obviously, then it is equivalent to prove that the optimal solution $\overline{\mathbf{n}}_C^*$ to (\ref{pCC}) exists uniquely. Indeed, since the function $-J_{0C}$ is continuous on the compact set $[0,\tau]\times[0,1]$, there exists at least one solution to (\ref{pCC}). On the other hand, similar to the proof of Lemma \ref{lemma2}, it is seen that the function $-J_{0C}$ is strictly convex on $[0,\tau]\times[0,1]$ such that there exists at most one optimal solution to (\ref{pCC}). Hence, the optimal solution $\overline{\mathbf{n}}_C^*$ to (\ref{pCC}) exists uniquely. Analogous to the proof of Proposition \ref{th2}, the rest parts of Proposition \ref{th1} (i.e., deriving the expressions of $\overline{c}_i^*$ and $\overline{\mathbf{n}}_i^*$ for $i=C,M$) are obtained via Karush-Kuhn-Tucker conditions.

%%%%%%%%%%%%%%%%%%%%%%%%%%%%%%%%%%%%%%%%%%%%%%%%%%%%%%%%%%%%%%%%%%%%%%%%%%%%%%%%%%%%%
\noindent{\bf Proof of Proposition \ref{th4}.}
Using (\ref{W}), (\ref{ps_c}) and the market clearing conditions (\ref{ec-1})-(\ref{ec-4}), we have
\[\tau S_1+S_2=X_{C}+X_{M},\quad X_{i}=c_{i}^*\rho^{-\frac{1}{\gamma_i}},\;i=C,M,\]
and
\begin{equation}
\widehat{D}=\rho^{\frac{1}{\gamma_C}}X_{C}+\rho^{\frac{1}{\gamma_M}}X_{M}.\label{dX}
\end{equation}
Then it is easy to see that
\begin{align*}
&\widehat{D}_1=y_1\widehat{D},&&c_{M}^*=y\widehat{D},&&X_M=y\widehat{D}\rho^{-\frac{1}{\gamma_M}},
&&S_1=\frac{y_2}{\tau}\Gamma_0\widehat{D},\\
&\widehat{D}_2=(1-y_1)\widehat{D},&&c_{C}^*=(1-y)\widehat{D},&&X_C=(1-y)\widehat{D}\rho^{-\frac{1}{\gamma_C}},
&& S_2=(1-y_2)\Gamma_0\widehat{D},
\end{align*}
and so (\ref{ratio}) is obtained.

Making use of (\ref{ps_nM}), it is seen that the stock growth rate $\mu$ can be expressed as (\ref{mu}). Applying It\^{o}'s Lemma to both sides of (\ref{dX}), then matching the terms of $dt,dW$ and $dB$ gives
\begin{align*}
\widehat{D}_1\mu_{1D}+\widehat{D}_2\mu_{2D}=&r\widehat{D}+\Gamma_1S_1(\mu_1-r)+\frac{D_1\Gamma_1}{\tau}
+\Gamma_2S_2(\mu_2-r)+(1-x^*)\Gamma_2D_2\\
&+\left(\rho^{\frac{1}{\gamma_C}}l_{1C}+\rho^{\frac{1}{\gamma_M}}l_{1M}\right)\widehat{D}_1
+\left(\rho^{\frac{1}{\gamma_C}}l_{2C}+\rho^{\frac{1}{\gamma_M}}l_{2M}\right)\widehat{D}_2\\
 &+\left[\rho^{\frac{1}{\gamma_C}}\left(x^*-\frac{k}{2}(x^*)^2\right)
+\rho^{\frac{1}{\gamma_M}}\frac{k}{2}(x^*)^2\right]D_2-\rho^{\frac{1}{\gamma_C}}c_C^*
-\rho^{\frac{1}{\gamma_M}}c_M^*,
\end{align*}
\[
%\sigma
\sigma_D\widehat{D}=(\Gamma_1S_1+\Gamma_2S_2)\sigma,\qquad
%\delta
\delta_D\widehat{D}_2=\Gamma_2S_2\delta.
\]
And the interest rate $r$, the stock volatilities $\sigma$ and $\delta$ are shown to be (\ref{r}), (\ref{sigma}) and (\ref{delta}) respectively.

Using exactly the same steps as above, we obtain the parameters in $y$, $y_1$ and $y_2$ respectively:
in the case of $y\widehat{D}=\rho^{\frac{1}{\gamma_M}}X_M$, equations
\begin{align*}
y(\widehat{D}_1\mu_{1D}+\widehat{D}_2\mu_{2D})&+\widehat{D}(\mu_y+\sigma_D\sigma_y)+\widehat{D}_2\delta_D\delta_y\\
=&\rho^{\frac{1}{\gamma_M}}\bigg[X_Mr+n_{1M}^*\left(S_1(\mu_{1}-r)+\frac{D_1}{\tau}\right)
+n_{2M}^*\left(S_2(\mu_{2}-r)+(1-x^*)D_2\right)\\
&+l_{1M}\widehat{D}_1+l_{2M}\widehat{D}_2-c_M^*+\frac{k(x^*)^2}{2}D_2\bigg],
\end{align*}
\[y\widehat{D}\sigma_D+\widehat{D}\sigma_y=\rho^{\frac{1}{\gamma_M}}(n_{1M}^*S_1+n_{2M}^*S_2)\sigma,\quad
y\widehat{D}_2\delta_D+\widehat{D}\delta_y=\rho^{\frac{1}{\gamma_M}}n_{2M}^*S_2\delta\]
gives (\ref{y}); in the case of $\widehat{D}_1=y_1\widehat{D}$, equations
\begin{align*}
&\widehat{D}_1\mu_{1D}=y_1(\widehat{D}_1\mu_{1D}+\widehat{D}_2\mu_{2D})+\widehat{D}\mu_{1y}
+\widehat{D}\sigma_D\sigma_{1y}+\widehat{D}_2\delta_{1y}\delta_{D},\\
&\widehat{D}_1\sigma_{D}=y_1\widehat{D}\sigma_D+\widehat{D}\sigma_{1y},\quad
0=y_1\widehat{D}_2\delta_D+\widehat{D}\delta_{1y}
\end{align*}
gives (\ref{y1}); in the case of $\tau S_1=y_2(\tau S_1+S_2)$, equations
\begin{align*}
&\tau S_1\mu_1=y_2(\tau S_1\mu_1+S_2\mu_2)+(\tau S_1+S_2)(\mu_{2y}+\sigma_{2y}\sigma)+S_2\delta_{2y}\delta,\\
&\tau S_1\sigma=y_2(\tau S_1+S_2)\sigma+(\tau S_1+S_2)\sigma_{2y},\quad
0=y_2S_2\delta+(\tau S_1+S_2)\delta_{2y}
\end{align*}
gives (\ref{y2}).

Finally, the shareholders' optimal stock holdings are easily obtained through (\ref{problemCn}) and market clearing conditions (\ref{ec-1})-(\ref{ec-2}).

%%%%%%%%%%%%%%%%%%%%%%%%%%%%%%%%%%%%%%%%%%%%%%%%%%%%%%%%%%%%%%%%%%%%%%%%%%%%%%%%%%%%%
\noindent{\bf Proof of Proposition \ref{th3}.} The proof is analogous to that of Proposition \ref{th4}.

%%%%%%%%%%%%%%%%%%%%%%%%%%%%%%%%%%%%%%%%%%%%%%%%%%%%%%%%%%%%%%%%%%%%%%%%%%%%%%%%%%%%%
\noindent{\bf Proofs used in Remark \ref{boundary}.}
We just prove the results for Proposition \ref{th4}, and the proofs of results for Proposition \ref{th3} is analogous. It is clear that $\mathbf{n}_C^*$ converges to $(\tau,1)^\top$ when $y$ converges to $0$, i.e., (\ref{b1}) holds, which implies $x^*\langle0\rangle=0,\Gamma_0\langle0\rangle=\rho^{-\frac{1}{\gamma_C}},\Gamma_1\langle0\rangle=\tau\rho^{\frac{1}{\gamma_C}}$ and $\Gamma_2\langle0\rangle=\rho^{\frac{1}{\gamma_C}}$. Then (\ref{b4})-(\ref{b8}) are easily obtained from (\ref{sigma}), (\ref{delta}) and (\ref{y}). In fact, noticing that $\Lambda=\{\mu-r(\mathbf{e}_1+\mathbf{e}_2)\}\langle0\rangle$, direct compution gives the expressions (\ref{b2}) and (\ref{b3}). Hence, it suffices to derive $\Lambda$ to complete the proofs.

Observing regions in Proposition \ref{th2}, the facts $\mathbf{n}_C^*\langle0\rangle=(\tau,1)^\top$ and $\frac{1}{(1-p)k+1}<1=\mathbf{n}_{2C}^*\langle0\rangle$ hold if there exists a constant $0<\delta_v<1$ such that $\mathbf{n}_C^*$ is located in Region $v$ ($v=5,9,10,12$) for all $y\in(0,\delta_v)$. So we derive $\Lambda^v$ in Regions $v=5,9,10,12$ respectively.

In Region $5$, $\eta^*_{1,5}=\frac{\theta_{1C}+{\alpha_{1C}}-\xi_{0C}}
{\xi_{1C}}$ implies
\begin{equation}\label{eq1}
\mu_1-r=\frac{1}{S_1/X_C}(\xi_{1C}\eta^*_{1,5}- \alpha_{1C}+\xi_{0C}),
\end{equation}
which, by letting $y$ converge to $0$, gives $\Lambda^5_1$ in (\ref{bb5}). The parameter $\mu$ in (\ref{mu}) deduces
\[
\frac{S_1}{X_M}(\mu_1-r)=\xi_{1M}(\tau-\eta^*_{1,5})-\alpha_{1M},\quad
\frac{S_2}{X_M}(\mu_2-r)=\xi_{0M}(\tau-\eta^*_{1,5})-(1-x^*)\alpha_{2M}.
\]
Combining above two equations, it is seen
\begin{equation}\label{eq2}
\mu_2-r=\frac{\xi_{0M}/\xi_{1M}}{S_2/X_M}\bigg(\frac{S_1}{X_M}(\mu_1-r)+\alpha_{1M}\bigg)
-\frac{(1-x^*)\alpha_{2M}}{S_2/X_M}.
\end{equation}
Now substituting $\mu_1-r$ of (\ref{eq1}) in (\ref{eq2}) and then letting $y$ converges to $0$, $\Lambda^5_2$ in (\ref{bb5}) is obtained.

In Region 9, substituting $(\tau,1)^\top$ for $\eta^v$  in (\ref{mu}) gives
\[\mu-r(\mathbf{e}_1+\mathbf{e}_2)=\frac{\frac{\tau y}{y_2}\mathbf{e}_{11}+\frac{y}{1-y_2}\mathbf{e}_{22}}
{\rho^{\frac{1}{\gamma_M}}\Gamma_0}(
-\alpha_{1M}\mathbf{e}_1-(1-x^*)\alpha_{2M}\mathbf{e}_2).\]
Then letting $y$ converge to $0$, $\Lambda^9$ in (\ref{bb9}) is obtained.

In the former case of Region 10, $\Lambda^{10}$ in (\ref{bb10}) is easily derived by letting $y$ converge to $0$ in
\[\eta_{10}^*=\left(\xi_C-\frac{\alpha_{2C}}{k}\mathbf{e}_{22}\right)^{-1}\cdot\left[\theta_C
+{\alpha_{1C}}\mathbf{e}_1+\left(1-\frac{1}{k}\right)\alpha_{2C}\mathbf{e}_{2}\right].\]
In the later case of Region $10$, $\Lambda^{10}$ in (\ref{bb0}) can be derived by letting $y$ converge to $0$ in conditions
\[
\xi_{1C}\eta_{1,10a}^*+\xi_{0C}\eta_{2,10a}^*=\theta_{1C}+{\alpha_{1C}},
\frac{\xi_{0C}(\theta_{1C}+{\alpha_{1C}})}{\xi_{1C}}
=\theta_{2C}+\left(1-\frac{1}{k}\right)\alpha_{2C}.
\]

The proof in Region 12 is analogous to that in Region 5. And finally we complete the proofs.

%%%%%%%%%%%%%%%%%%%%%%%%%%%%%%%%%%%%%%%%%%%%%%%%%%%%%%%%%%%%%%%%%%%%%%%%%%%%%%%%%%%%%
\noindent{\bf Proofs of Propositions \ref{th5} and \ref{th6}.}
As is stated in Remark \ref{remark3.3}, in Region $v\;(v=1,2,\cdots,9)$, if the solution $\overline{\eta}_v^*$ to related fixed-point equation of Step 1 in Remark \ref{remark3.2} exists and if related conditions are satisfied for $\overline{\eta}_v^*$, then $\overline{\eta}_v^*$ is indeed the controlling shareholder's optimal stock holding in equilibrium. Hence, in order to verify the extinction of shareholders, we can give conditions on which not only does the controlling shareholder's optimal stock holding exist in related regions satisfying $\overline{n}^*_{jM}=0$ or $\overline{n}^*_{jC}=0$ ($j=1,2$) in Proposition \ref{th1} but associated conditions also hold in equilibrium. We complete the proof by listing the relations as follows:
\begin{description}
  \item[$(1)$] {in Regions 2 and 7, $\overline{n}^*_{1C}=0$ implies the controlling shareholder becomes extinct in firm $1$, which is relative to the conditions $(\mathcal{C}'1)$ and $(\mathcal{C}'2)$;}
  \item[$(2)$] in Regions 4 and 8, $\overline{n}^*_{2C}=0$ implies the controlling shareholder becomes extinct in firm $2$, which is relative to the conditions $(\mathcal{C}'3)$ and $(\mathcal{C}'4)$;
  \item[$(3)$] in Regions 3 and 8, $\overline{n}^*_{1C}=\tau$ implies the minority shareholder becomes extinct in firm $1$, which is relative to the conditions $(\mathcal{M}'1)$ and $(\mathcal{M}'2)$;
  \item[$(4)$] in Regions 5 and 7, $\overline{n}^*_{2C}=1$ implies the minority shareholder becomes extinct in firm $2$, which is relative to the conditions $(\mathcal{M}'3)$ and $(\mathcal{M}'4)$.
\end{description}
Note that Regions 6 ($y=1$) and 9 ($y=0$) are not included because it requires $0<y<1$ in Definition \ref{def_ex}.

%%%%%%%%%%%%%%%%%%%%%%%%%%%%%%%%%%%%%%%%%%%%%%%%%%%%%%%%%%%%%%%%%%%%%%%%%%%%%%%%%%%%%
\noindent{\bf Proofs of Propositions \ref{th7} and \ref{th8}.}
Similarly, in Region $v\;(v=1,2,\cdots,12)$, if the solution ${\eta}_v^*$ to related fixed-point equation of Step 1 in Remark \ref{remark3.2} exists and if related conditions and the condition $\det(\xi_C)>\frac{\alpha_{2C}\xi_{1C}}{k}$ are satisfied for ${\eta}_v^*$, then ${\eta}_v^*$ is indeed the controlling shareholder's optimal stock holding in equilibrium. Hence, in order to verify the extinction of shareholders, we can give conditions on which not only does the controlling shareholder's optimal stock holding exist in related regions satisfying $n^*_{jM}=0$ or $n^*_{jC}=0$ ($j=1,2$) in Proposition \ref{th2} but associated conditions and the condition $\det(\xi_C)>\frac{\alpha_{2C}\xi_{1C}}{k}$ also hold in equilibrium. We complete the proof by listing the relations as follows:
\begin{description}
  \item[$(1)$] {in Regions 2, 7 and 11, $n^*_{1C}=0$ implies the controlling shareholder becomes extinct in firm $1$, which is relative to the conditions $(\mathcal{C}1)$-$(\mathcal{C}3)$;}
  \item[$(2)$] in Regions 4 and 8, $n^*_{2C}=0$ implies the controlling shareholder becomes extinct in firm $2$, which is relative to the conditions $(\mathcal{C}4)$ and $(\mathcal{C}5)$;
  \item[$(3)$] in Regions 3, 8 and 12, $n^*_{1C}=\tau$ implies the minority shareholder becomes extinct in firm $1$, which is relative to the conditions $(\mathcal{M}1)$-$(\mathcal{M}3)$;
  \item[$(4)$] in Regions 5 and 7, $n^*_{2C}=1$ implies the minority shareholder becomes extinct in firm $2$, which is relative to the conditions $(\mathcal{M}4)$ and $(\mathcal{M}5)$.
\end{description}
The later cases of Regions $11$ and $12$ are excluded because the condition $\frac{\alpha_{2C}}{k}= 0$ contradicts $\det(\xi_C)>\frac{\alpha_{2C}\xi_{1C}}{k}$, and Regions 6 ($y=1$) and 9 ($y=0$) are not included because of $0<y<1$ in Definition \ref{def_ex}.

%%%%%%%%%%%%%%%%%%%%%%%%%%%%%%%%%%%%%%%%%%%%%%%%%%%%%%%%%%%%%%%%%%%%%%%%%%%%%%%%%%%%%
\noindent{\bf Proofs of Propositions \ref{th9} and \ref{th10}.} The proofs are analogous to those of Propositions \ref{th2} and \ref{th4}, respectively.

\end{document}